\documentclass[oneside,a4paper,11pt]{article}
\usepackage{mathtools}
\usepackage{amsmath}
\allowdisplaybreaks[4]
\usepackage{csquotes}
\usepackage{amsthm}
\usepackage{amssymb}
\usepackage{mathrsfs}
\usepackage{color}
\usepackage{bm}
\usepackage{comment}
\usepackage{fancyhdr}
\usepackage{geometry}
\usepackage{hyperref}
\usepackage{enumerate}

\usepackage{times}
\usepackage{graphicx}
\usepackage{subfig}
\usepackage{float}
\usepackage{hyperref} 
\usepackage{verbatim}

\theoremstyle{definition}

\newcommand{\pp}[2]{\frac{\partial#1}{\partial#2}}
\newtheorem{thm}{Theorem}[section]
\newtheorem{cor}[thm]{Corollary}
\newtheorem{lem}[thm]{Lemma}
\newtheorem{prop}[thm]{Proposition}

\newtheorem{defn}[thm]{Definition}
\newtheorem{rem}[thm]{Remark}
\newtheorem{ex}[thm]{Example}
\newtheorem{question}[thm]{Problem}
\numberwithin{equation}{section}

\title{Characterizations of infinite circle patterns and convex polyhedra in hyperbolic 3-space}
\author{Huabin Ge, Longsong Jia, Hao Yu, Puchun Zhou
}

\newcommand{\HH}{\mathbb{H}}

\newcommand{\vel}{\mathrm{VEL}}
\newcommand{\cond}{\mathrm{COND}}
\newcommand{\modc}{\mathrm{MOD}}
\newcommand{\res}{\mathrm{RES}}
\newcommand{\carrier}{\mathrm{Carr}}
\newcommand{\pac}{\mathcal{P}}
\newcommand{\area}{\mathrm{area}}

\newcommand{\mT}{\mathcal{T}}

\date{}
\usepackage{fancyhdr}
\begin{document}
	\maketitle
	\begin{abstract}	
		Since Thurston pioneered the connection between circle packing (abbr. CP) and three-dimensional geometric topology, the characterization of CPs and hyperbolic polyhedra has become increasingly profound. Some milestones have been achieved, for example, Rodin-Sullivan \cite{Rodin-Sullivan} and Schramm \cite{schramm91} proved the rigidity of infinite CPs with the intersection angle $\Theta=0$. Rivin-Hodgson \cite{RH93} fully characterized the existence and rigidity of compact convex polyhedra in $\mathbb{H}^3$. He \cite{He} proved the rigidity and uniformization theorem for infinite CPs with $0\leq\Theta\leq \pi/2$. Therefore, the remaining unresolved issues are the rigidity and uniformization theorems for infinite CPs with $0\leq\Theta<\pi$, as well as for infinite hyperbolic polyhedra. In fact, He specifically claimed in the abstract of \cite{He} that ``in a future paper, the techniques of this paper will be extended to the case when $0\leq\Theta<\pi$. In particular, we will show a rigidity property for a class of infinite convex polyhedra in the 3-dimensional hyperbolic space".
		
		\medskip
		
        The objective of the article is to accomplish the work claimed in \cite{He} by proving the rigidity and uniformization theorem for infinite CPs with $0\leq\Theta<\pi$, as well as infinite trivalent hyperbolic polyhedra. We will pay special attention to CPs whose contact graphs are disk triangulation graphs. Such CPs are called regular because they exclude some singular configurations and correspond well to hyperbolic polyhedra. We will establish the existence and rigidity of infinite regular CPs. Moreover, we will prove a uniformization theorem for regular CPs, which solves the classification problem for regular CPs. Thereby, the existence and rigidity of infinite convex trivalent polyhedra are obtained. 
		\\[5pt]
		
	\end{abstract}
	\tableofcontents
	\section{Introduction}
	A \textbf{circle pattern}\footnote{Our article does not distinguish between circle packing, circle pattern, and disk pattern. In other literature, circle packing sometimes refers specifically to tangent types, while the circle (or disk) pattern allows circles to overlap.} $\mathcal{P}$ on the sphere $\hat{\mathbb{C}}=\partial\mathbb{H}^3$ is a collection of closed disks in $\hat{\mathbb{C}}$ where each disk intersects only a finite number of disks. The \textbf{contact graph} of $\mathcal{P}$ is the graph $G=G(\pac)$ whose vertices correspond to the centers of the disks, and an edge appears in $G$ if the corresponding disks intersect each other. Let $E$ be the set of edges in $G$, for each edge $e\in E$ connecting two disks $D_1$ and $D_2$, denote $\Theta(e)\in[0,\pi)$ by the intersection angle between $\partial D_1$ and $\partial D_2$. $\Theta(e)$ is also called the dihedral angle, as it is exactly the dihedral angle between two geodesic planes $\Pi_1$ and $\Pi_2$ with $\Pi_i\cap\hat{\mathbb{C}}=\partial D_i$, $i=1,2$. See \cite{He, st} for more background on circle patterns. 
	
	Each $\mathcal{P}$ produces an angled graph $(G,\Theta)$ with $0\leq\Theta<\pi$, in which case we say that $\pac$ \textbf{realizes} $(G,\Theta)$. Conversely, given an angled graph $(G,\Theta)$, can it be the contact graph of some circle patterns? And if it does, is the circle pattern unique? In case $G$ is finite and $0\leq\Theta\leq \pi/2$, Koebe-Andreev-Thurston's circle packing theorem answers this question quite well. Their theorem for $\Theta=0$ was discovered by Koebe \cite{Ko36}, and the theorem for $\Theta\in[0,\pi/2]$ comes from Thurston's interpretation of Andreev's characterization \cite{An70A} about compact hyperbolic polyhedra with non-obtuse dihedral angles (see e.g. \cite{Hodgson,RHD07,Zhou23} for proofs and generalizations).
	In case $G$ is infinite, the related problems are extremely difficult to handle and the methods are completely different. Some great milestones have been achieved. Rodin-Sullivan \cite{Rodin-Sullivan}, Schramm \cite{schramm91} proved the rigidity of infinite circle patterns with $\Theta=0$. He-Schramm \cite{He1} extended KAT's theorem to infinite triangulation graphs with $\Theta=0$. He \cite{He} further proved the rigidity and uniformization theorem for infinite circle patterns with $0\leq\Theta\leq \pi/2$. Rivin-Hodgson \cite{RH93} fully characterized the existence and rigidity of compact convex polyhedra in $\mathbb{H}^3$, which greatly generalizes Andreev's theorem to allow the prescribed dihedral angles $\Theta \in (0, \pi)^E$. Rivin \cite{Ri96} further characterized the existence and rigidity of finite ideal polyhedra in $\mathbb{H}^3$. Bao-Banahon \cite{Bao} generalized Rivin's theorem to hyperideal polyhedra. Our work was inspired by \cite{He}, where He mentioned that ``in a future paper, the techniques of this paper will be extended to the case when $0\leq\Theta<\pi$. In particular, we will show a rigidity property for a class of infinite convex polyhedra in the 3-dimensional hyperbolic space". In this article, we aim to finish He's unpublished work. 
	
	$\pac$ is called \textbf{regular} if its contact graph $G=(V,E)$ is a disk triangulation graph, that is, $G$ is induced by an infinite and locally finite triangulation $\mathcal{T}= (V, E, F)$  of \textbf{the open unit disk} $\mathbb{U}$. An arc $\Gamma$ formed by some edges in $E$ is called \textbf{homologically adjacent} if there is an edge in $E$ that connects the two endpoints of $\Gamma$. Conversely, we call it \textbf{homologically non-adjacent}.
	Given a disk triangulation graph $(G,\Theta)$ with angle $\Theta\in[0,\pi)^E$, we will use the following conditions:
	\begin{itemize}
		\item[($Z_1$)] If \( e_{1}, e_{2}, e_{3} \) forms the boundary of a triangle, 
		and \( \sum_{i=1}^{3} \Theta(e_{i}) > \pi \), then 
		$\Theta(e_{i}) + \Theta(e_{j}) < \pi + \Theta(e_{k})$ hold for all permutations $\{i,j,k\}=\{1,2,3\}$.\\[-15pt]
		\item[($Z_2$)] If \( e_{1}, e_{2}, \cdots, e_{s} \) form a simple closed curve that is not the boundary of any triangle face, then
		$\sum_{i=1}^{s} \Theta(e_{i}) < (s - 2)\pi.$	\\[-15pt]
		\item[($Z_3$)] If \( e_{1} \) and \( e_{2} \) are homologically non-adjacent, then
		$\Theta(e_{1}) + \Theta(e_{2}) < \pi.$\\[-15pt]
		\item[($Z_4$)] If \( e_{1}, e_{2}, e_{3} \) forms the boundary of a triangle, then
		$\cos\Theta(e_i) + \cos\Theta(e_j) \cos\Theta(e_k) \geq 0$
		hold for all permutations $\{i,j,k\}=\{1,2,3\}$.
	\end{itemize}
  From Thurston \cite{Th76} and He \cite{He}, if $(G,\Theta)$ is realized by some $\pac$ and $\Theta<\pi/2$, all conditions ($Z_1$)-($Z_4$) are satisfied. If $\Theta$ is still allowed to take $\pi/2$, then they are satisfied except in one case, that is, He's reducible example (see Figure \ref{subgraphgivingextra}). In addition, these conditions are closely related, for example, ($Z_4$) implies  ($Z_1$) (see Section \ref{section:cp-notation}). 
	
	
	\begin{thm}[Existence of RCPs]\label{infinite_existence}
		Let $G= (V, E)$ be a disk triangulation graph with an angle function $\Theta \in [0, \pi)^E$. Assume that $\sup_{e\in E}\Theta(e)<\pi$. If conditions ($Z_1$)-($Z_3$) hold, then there is a \textbf{regular circle pattern} (abbr. \textbf{RCP}) realizing $(G,\Theta)$. If conditions ($Z_2$) and ($Z_4$) hold, then there is an embedded circle pattern $\pac$ weakly realizing $(G,\Theta)$, that is, $G$ is a subgraph of the contact graph $G(\pac)$. 
	\end{thm}
	
	By the correspondence between RCPs and hyperbolic polyhedra (see Section \ref{section-correspond-cp-hcp}), we have the following existence result for infinite trivalent hyperbolic polyhedra.
	
	\begin{thm}[Existence of THP]\label{thm-exist-IP}
	Let $\mathcal{T}=(V, E, F)$ be a disk triangulation with an angle function $\Theta \in (0, \pi)^E$. Assume that $\sup_{e\in E}\Theta(e)<\pi$. If conditions ($Z_1$)-($Z_3$) hold, then there is an infinite \textbf{trivalent hyperbolic polyhedron} (abbr. \textbf{THP}) $P$ in $\mathbb{H}^3$ which is combinatorially equivalent to the Poincar\'e dual of $\mathcal{T}$, with dihedral angles satisfying $\Theta(e^*) = \Theta(e)$ for $e\in E$.
	\end{thm}
	It should be noted that ($Z_1$) and ($Z_2$) are necessary conditions for the existence of THP (see Proposition \ref{prop-thp-z1+z2}), but not
    necessary conditions for the existence of RCPs (see Remark \ref{remark-rcp-notimply-z1z2}). For more precise and detailed relationships between conditions ($Z_1$)-($Z_3$), RCPs, and THP, we refer to Section \ref{section-correspond-cp-hcp}. Let $\pac$ be a circle pattern in $\mathbb{C}$, $z\in\hat{\mathbb{C}}$ is called an accumulation point of $\pac$ if every neighborhood of $z$ intersects infinitely many disks of $\pac$. If $\pac$ is regular, one can directly verify that there is a unique domain $\Omega \subset \hat{\mathbb{C}}$ such that $\pac$ is \textbf{locally finite} (abbr. \textbf{l.f.}) in $\Omega$, that is, every disk in $\pac$ is contained in $\Omega$ and every compact subset of $\Omega$ intersects only finitely many disks of $\pac$. Moreover, $\partial \Omega$ coincides with the set of accumulation points of $\pac$ and $\Omega = \carrier(\pac)$.

	\begin{thm}[Rigidity of CPs]\label{thm-intro-rigidity-cp}
	Let $G=(V,E)$ be a disk triangulation graph with an angle function $\Theta\in [0 ,\pi)^E$ that satisfies ($Z_2$) and ($Z_4$). Let $\mathcal{P}$ and $\mathcal{P}'$ be two CPs that realize $(G,\Theta)$ (by the definition of RCP, both $\mathcal{P}$ and $\mathcal{P}'$ are automatically RCPs). 
		\begin{itemize}
			\item [(1)] If $\pac$ and $\pac'$ are l.f. in $\mathbb{U}$, there is a M\"obius transformation $h$ such that $\mathcal{P}'=h(\mathcal{P})$.\\[-18pt]
			\item [(2)]If $\pac$ is l.f. in $\mathbb{C}$ and $\sup\limits_{e\in E}\Theta<\pi$, there is an Euclidean similarity $h$ such that $\mathcal{P}'=h(\mathcal{P}).$
		\end{itemize}
	\end{thm}
	
	The \textbf{carrier} of $\pac$, denoted by $\carrier(\pac)$, is defined as the union of all disks in $\pac$ together with all closed geometric triangles $\Delta v_i v_j v_k$ associated with triangular faces of the underlying triangulation.
	By He-Schramm \cite[$\Theta=0$]{He2} and He \cite[$0\leq\Theta\leq\pi/2$]{He}, every disk triangulation graph can be circle packed in $\mathbb{C}$ with carrier either $\mathbb{C}$ (called CP-parabolic) or the open unit disk $\mathbb{U}$ (called CP-hyperbolic), but not both. An angled disk triangulation graph $(G,\Theta)$ with $\Theta\in [0 ,\pi)^E$ is called \textbf{RCP-parabolic} (\textbf{RCP-hyperbolic} resp.), if there is a regular circle pattern $\pac$ with $\carrier(\pac)= \mathbb C$ ($=\mathbb U$ resp.) that realizes $(G,\Theta)$. Similarly, consider an infinite convex THP $P$ in $\mathbb{H}^3$. Let $P^{\text{trun}}$ be its truncation (that is, we truncate $P$ at each hyperideal vertex and ideal vertex), then $P$ is called \textbf{parabolic} if $P^{\text{trun}}\setminus K_n$ tends to a unique point in $\partial\mathbb{H}^3$, and called \textbf{hyperbolic} if $P^{\text{trun}}\setminus K_n$ tends to a half sphere in $\partial\mathbb{H}^3$, where $\{K_n\}$ are compact in $\mathbb{H}^3$ and exhaust $\mathbb{H}^3$ (see Section \ref{polyhedra_section} for more explanations). The following uniformization theorem distinguishes the different types of RCPs and THPs.

	\begin{thm}[Uniformization of RCP/THP]\label{uniformization}
		Let $\mathcal{T}= (V, E, F)$ be a disk triangulation, and let $\Theta\in[0 ,\pi)^E$ be an angle function. Let $G=(V,E)$ and consider the following four properties:
		\begin{enumerate}
			\item [($U_1$)] $G$ is VEL-parabolic (VEL-hyperbolic resp.).
			\item [($U_2$)] $(G,\Theta)$ is RCP parabolic (RCP hyperbolic resp.).
			\item [($U_3$)] $\mathcal{T}$ is the dual 1-skeleton of some parabolic (hyperbolic resp.) THP in $\mathbb{H}^3$.
			\item [($U_4$)] $\mathcal{T}$ is recurrent (transient resp.).
		\end{enumerate}
		Then we have the following uniformization results for infinite RCPs and THPs:
		\begin{itemize}
			\item[(1)] If conditions ($Z_1$)-($Z_3$) hold and $\sup_{e\in E}\Theta(e)<\pi$, then ($U_1$) and ($U_2$) are equivalent. 
            If further assume that $\mT$ has bounded degrees, then ($U_1$), ($U_2$) and ($U_4$) are equivalent.
			\item[(2)] If conditions ($Z_1$)-($Z_3$) hold, $\Theta>0$ and $\sup_{e\in E}\Theta(e)<\pi$, then ($U_1$)-($U_3$) are all equivalent. If further assume that $\mT$ has bounded degrees, then ($U_1$)-($U_4$) are all equivalent. 
		\end{itemize}
	\end{thm}

    To our knowledge, this is the first uniformization result concerning infinite hyperbolic polyhedra. The equivalence between ($U_1$) and ($U_2$) in Theorem \ref{uniformization} extends the results in \cite{BS90,BS96,He1,He} (which studied the cases of $\Theta=0$ and $0\leq\Theta \leq\pi/2$ respectively). It should be noted that the equivalence between ($U_1$) and ($U_4$) was previously obtained by He-Schramm in \cite{He2}. 
 Furthermore, it is worth emphasizing that the THP in ($U_3$) may have non-compact faces, and each face of the corresponding THP must be compact if ($Z_1$) is replaced by the following ($Z'_1$): if \( e_{1}, e_{2}, e_{3} \) forms the boundary of a triangle, then $\sum\limits_{i=1}^{3} \Theta(e_i) > \pi$ and $\Theta(e_{i}) + \Theta(e_{j}) < \pi + \Theta(e_{k})$ for all \( \{i,j,k\} = \{1,2,3\} \). Using Theorem \ref{thm-intro-rigidity-cp}, Theorem \ref{uniformization}, and the correspondence between RCPs and THP, we have the following rigidity theorem for THP.

	\begin{thm}[Rigidity of THP]\label{thm-rigidity-IP}
	Let $\mathcal{T}=(V,E,F)$ be a disk triangulation, and let $\Theta\in(0 ,\pi)^E$ be an angle function. Assume that conditions ($Z_3$) and ($Z_4$) hold. Let $P_1$, $P_2$ be two hyperbolic (or parabolic, with an additional assumption that $\sup\limits_{e\in E}\Theta(e)<\pi$) THP whose Poincar\'e dual is combinatorially isomorphic to $\mathcal{T}$ with $\Theta(e^*)=\Theta(e)$ at each $e\in E$, then $P_1$ is isometric to $P_2$. 
	\end{thm}
    \begin{rem}
    For infinite non-parabolic THP, if we do not require the accumulation points to form a circle, then \emph{rigidity phenomena do not exist, which is completely different from the case of finite polyhedra} (compare the rigidity results for finite polyhedra summarized in Section \ref{section-character-finite-hp}). Actually, using Theorem \ref{uniformization}, it is easy to construct two infinite THP $P_1$ and $P_2$ that have the same combinatorial structure and corresponding dihedral angles. Moreover, $P_1$ is a hyperbolic THP, but $P_2$ is not (the accumulation points of $P_2$ form an ellipse). Therefore, $P_1$ and $P_2$ are not isometric in $\HH^3$. This shows the non-rigidity of infinite THP. For more details, see Section \ref{section-correspond-cp-hcp}. 
    \end{rem}

	
	Theorems \ref{uniformization} and \ref{thm-rigidity-IP} provide detailed and satisfactory characterizations of infinite THP, and extend the series of pioneering works by Andreev \cite{An70A}, Rivin-Hodgson \cite{RH93}, Bao-Banahon \cite{Bao}, Huang-Liu \cite{HL17}, Zhou \cite{Zhou23} and others to the infinite case. Moreover, the infinite THP considered in this article allows for both compact and noncompact faces (see Section \ref{polyhedra_section}), which characterizes infinite hyperbolic polyhedra in a very general context. The characterization of convex polyhedra is very ancient, in fact, the Platonic solids (also called regular polyhedra) were known to the ancient Greeks. In modern times, there have been even more profound descriptions and understandings. Readers may refer to the following references for some relevant research progress: Alexandrov \cite{Alex05,Alex42}, Andreev \cite{An70A,An70B}, Bobenko–Izmestiev \cite{BI06}, Bobenko–Pinkall–Springborn \cite{BPS2015}, Bobenko–Springborn \cite{BS04}, Bowers–Bowers–Pratt \cite{BBP}, Chen-Schlenker \cite{Chen-Schlenker}, Gu\'eritaud \cite{Gueritaud}, Hodgson \cite{Hodgson}, Liu–Zhou \cite{Liu-Zhou}, Marden–Rodin \cite{MR90}, Rivin \cite{Ri03,Ri04}, Roeder–Hubbard–Dunbar \cite{RHD07}, Rousset \cite{Rou04}, Schlenker \cite{Sch00,Sch05}, Schramm \cite{schramm2}, Springborn \cite{Springborn1,Springborn2} (we apologize that some important relevant literature may have been overlooked, and we sincerely look forward to feedback from peer experts so that we can make changes and improvements in subsequent versions). Previous works involve most polyhedra with finite faces. Very few results of existence and rigidity for infinite polyhedra are known, and Rivin has done some pioneering work \cite{Ri03}. In this sense, it can be said that our work systematically characterizes a class of infinite hyperbolic polyhedra for the first time. Our work is inspired by He \cite{He}, Huang-Liu \cite{HL17}, Rivin \cite{Ri94,Ri96}, Rivin-Hodgson \cite{RH93}, and Zhou \cite{Zhou23}. We express our gratitude for these groundbreaking works.
	
	In addition to its early applications in fields such as 3-dimensional geometric topology, circle packing has in recent years developed an increasing number of deep connections with discrete geometry and combinatorics \cite{BHS06, Lam20}, complex analysis \cite{He-Liu-2013, He1, Lam21, Rod87,Rodin-Sullivan, st}, and other fields \cite{Aga03, BH03, BW25, KMT03}. It is particularly worth mentioning that since Schramm pioneered the SLE theory (see \cite{Roh11} for example), people have found increasingly profound connections between circle packing and fields such as probability theory and mathematical physics. Discoveries by several great mathematicians have revealed more and more profound phenomena and their growing importance. In fact, it is believed that under various notions of discrete conformal embeddings, uniform random triangulation converges to $\sqrt{8/3}$-LQG, where the field is given by Liouville CFT. The only proved case is Cardy-Smirnov's embedding, see \cite{HS2023}. By contrast, circle packing is also a classical model of discrete conformal embedding. Whether its scaling limit converges to some $\gamma$-LQG for a certain parameter $\gamma$ is still unknown. We refer to references such as \cite{AddAlb2017, ABGN16, AHNG16}, \cite{BLPS2001}-\cite{BHS2023}, \cite{CLR2023}, \cite{Curien2023}-\cite{DS2011}, \cite{Geo16,GN13,Gwy2020}, \cite{IM23,KLRR22}, \cite{LSW2001}-\cite{Le2014}, \cite{Nachmias2020}, \cite{Smirnov20}.

	
	The organization of the paper is as follows. In Section \ref{Preliminaries}, we will give some notations and conventions about circle patterns. After that, we will give a comprehensive discussion of their geometric properties, especially the three-circle configurations, the intersection properties under various conditions, and the relations between conditions ($Z_1$)-($Z_4$). In Section \ref{section:existence}, we will establish Theorem \ref{infinite_existence}, i.e. the existence of infinite circle patterns with the help of a ``Ring Lemma" for circle patterns with (possibly) obtuse intersection angles. In Section \ref{sec:uniform}, we will go through the concept of vertex extremal length and prove the uniformization theorems for RCP, which ensure the equivalence of ($U_1$) and ($U_2$) in Theorem \ref{uniformization}. In Section \ref{sec:rigid}, we will prove Theorem \ref{thm-intro-rigidity-cp}, i.e. the rigidity of RCPs. Finally, we will discuss the relationship between infinite RCPs and infinite THP in Section \ref{polyhedra_section}, derive Theorem \ref{thm-exist-IP} and finish the proof of Theorem \ref{uniformization} and Theorem \ref{thm-rigidity-IP}.\\
	
	\textbf{Acknowledgments.} 
	The first author expresses his sincere thanks to Professor Jinsong Liu and Wenfeng Jiang for useful discussions in the early stages of this work. The second author thanks Chuwen Wang for his figure support. The first and third authors thank Professor Bobo Hua and Ze Zhou for their interest in this work and some inspiring discussions. The first author is supported by NSFC, no.12341102, no.12122119, and no.12525103. The third author was supported by NSFC, no.12531001. The fourth author was supported in part by the European Union (ERC, GeoFEM, 101164551). Views and opinions expressed are however those of the authors only and do not necessarily reflect those of the European Union or the European Research Council. Neither the European Union nor the granting authority can be held responsible for them.

	\section{Preliminaries}\label{Preliminaries}
	\subsection{Circle patterns and regular circle patterns}   
	Let $V$ be a discrete subset of $\mathbb{C}$, where each element $v\in V$ is the center of some closed disk $D_v$ in $\mathbb{C}$. A disk pattern embedded in $\mathbb{C}$ (with the standard Euclidean metric) is a collection of disks \(\{D_v\}_{v \in V}\) in $\mathbb{C}$ where each disk intersects only a finite number of disks. Let $C_v=\partial D_v$ be the boundary circle of $D_v$, then \(\{C_v\}_{v \in V}\) is called a circle pattern. For convenience, both $\{D_v\}$ and $\{C_v\}$ are referred to as a \textbf{circle pattern} \(\mathcal{P}\). 
	The \textbf{contact graph} $G=G(\mathcal{P})$ 
	of \(\mathcal{P}\) is a graph whose vertices correspond to the centers of the circles, and an edge appears in $G$ exactly when the corresponding circles intersect each other. The \textbf{intersection angle} (or \textbf{dihedral angle}) between two intersecting circles
    $C_i$ and $C_j$ is defined as the angle $\Theta_{ij}\in[0,\pi)$ formed by their oriented
    tangent lines at a point of $C_i\cap C_j$. Let $E$ be the set of edges in the graph $G$. For any $[v_i, v_j]\in E$, the dihedral angle of the circles $C_i$ and $C_j$ is often denoted by \(\Theta([v_i, v_j])\) or \(\Theta_{ij}\), see Figure \ref{two-disk-reduce-config}. A point $z\in\hat{\mathbb{C}}$ is called an accumulation point of $\pac$ if every neighborhood of $z$ intersects infinitely many disks of $\pac$. Let $\Omega \subset \mathbb{C}$ be an open domain; then $\pac$ is called \textbf{locally finite} (abbr. \textbf{l.f.}) \textbf{in $\Omega$} if each disk in $\pac$ is contained in $\Omega$ and $\pac$ has no accumulation point in $\Omega$, or equivalently, every compact subset of $\Omega$ intersects only finitely many disks of $\pac$. 
    	\label{section:cp-notation}
        \begin{figure}[h]
		\centering
		\includegraphics[width=0.42\textwidth]{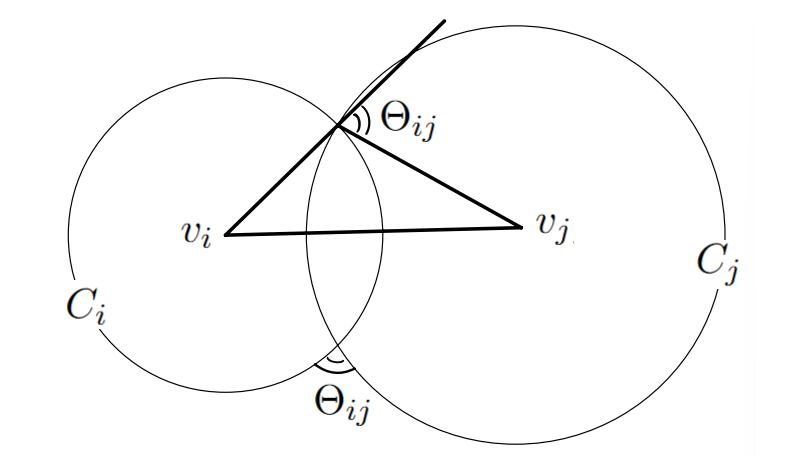}
		\caption{two circle configuration}
		\label{two-disk-reduce-config}
	      \end{figure}
    
    Consider four disks $D_1$, $D_2$, $D_3$, and $D_4$ in $\pac$, assuming that each pair of them intersects and that their centers $v_1$, $v_2$, $v_3$, and $v_4$ form a convex planar quadrilateral. In this case, the two diagonals $v_1v_3$ and $v_2v_4$ must intersect, and each diagonal is called \textbf{reducible}, see Figure \ref{reduce-edge}. The presence of reducible edges will result in the graph $G$ being immersed in $\mathbb{C}$ but not embedded. However, the reduced graph of $\pac$, consisting of the same vertex set as $G$ and the irreducible edges, is embedded in $\mathbb{C}$. Any circle pattern $\mathcal{P}$ produces an angled graph $(G,\Theta)$, which is a graph $G$ with an intersection angle function $\Theta:E\rightarrow[0,\pi)^E$. Conversely, given an abstract angled graph $(G,\Theta)$, it is called \textbf{realized} by a circle pattern $\pac$ if $G$ is the contact graph of $\pac$ and $\Theta$ is the intersection angle of $\pac$, and it is called \textbf{weakly realized} by $\pac$ if it can be extended to an angled graph $(\tilde{G},\tilde{\Theta})$ that is realized by $\pac$ and has the same vertex set as $G$. In order to characterize the geometry and combinatorics of CPs, a natural question arises:
            \begin{figure}[h]
		\centering
		\includegraphics[width=0.45\textwidth]{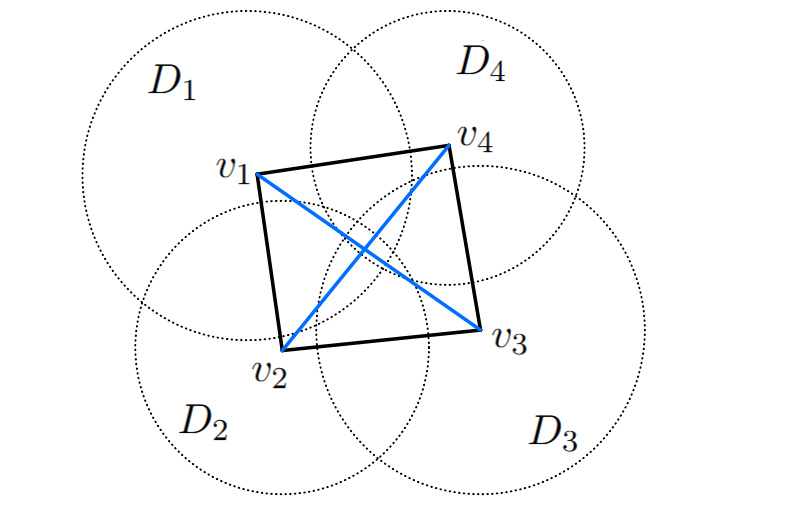}
		\caption{two reducible edges}
		\label{reduce-edge}
	      \end{figure}
	
	\begin{question}
		\label{problem-realize}
		Given an abstract angled graph $(G,\Theta)$ with angle $0\leq\Theta<\pi$, can it be realized (or weakly realized) by some circle pattern $\pac$? And if it does, is the circle pattern unique? 
	\end{question}
    
	Before addressing Problem \ref{problem-realize}, we first clarify the difference between the definition of ``$\pac$ realizes $(G,\Theta)$" here and the original definition of He in \cite{He}, where the angled graph $(G,\Theta)$ being considered satisfies $0\leq \Theta\leq \pi/2$. 
	In \cite{He}, whenever a simple loop $v_1$, $v_2$, $v_3$, $v_4$ in $G$ has the property that $\Theta_{12}=\Theta_{23}=\Theta_{34}=\Theta_{41}=\pi/2$ and $\Theta_{24}=0$, as shown in Figure \ref{subgraphgivingextra}, then add the other reducible edge $v_1v_3$ (if it is not in $G$) to the graph $G$. Denote $\tilde{G}$ by the graph thus obtained, and define $\tilde{\Theta}$ by letting $\tilde{\Theta}(e)=\Theta(e)$ if $e$ is an edge in $G$, and $\tilde{\Theta}(e)=0$ if $e$ is in $\tilde{G}-G$. A disk pattern $\pac$ is said to realize the data $(G,\Theta)$ if its contact graph is combinatorial isomorphic to $\tilde{G}$ and the corresponding dihedral angle function is equal to $\tilde{\Theta}$. Readers who are not familiar with the circle pattern can refer to Section \ref{section-thurston-cp} before reading the following Lemma \ref{lem-explain-he} and Example \ref{ex-four-circle}.
    
    	\begin{figure}[h]
		\centering
		\includegraphics[width=0.9\textwidth]{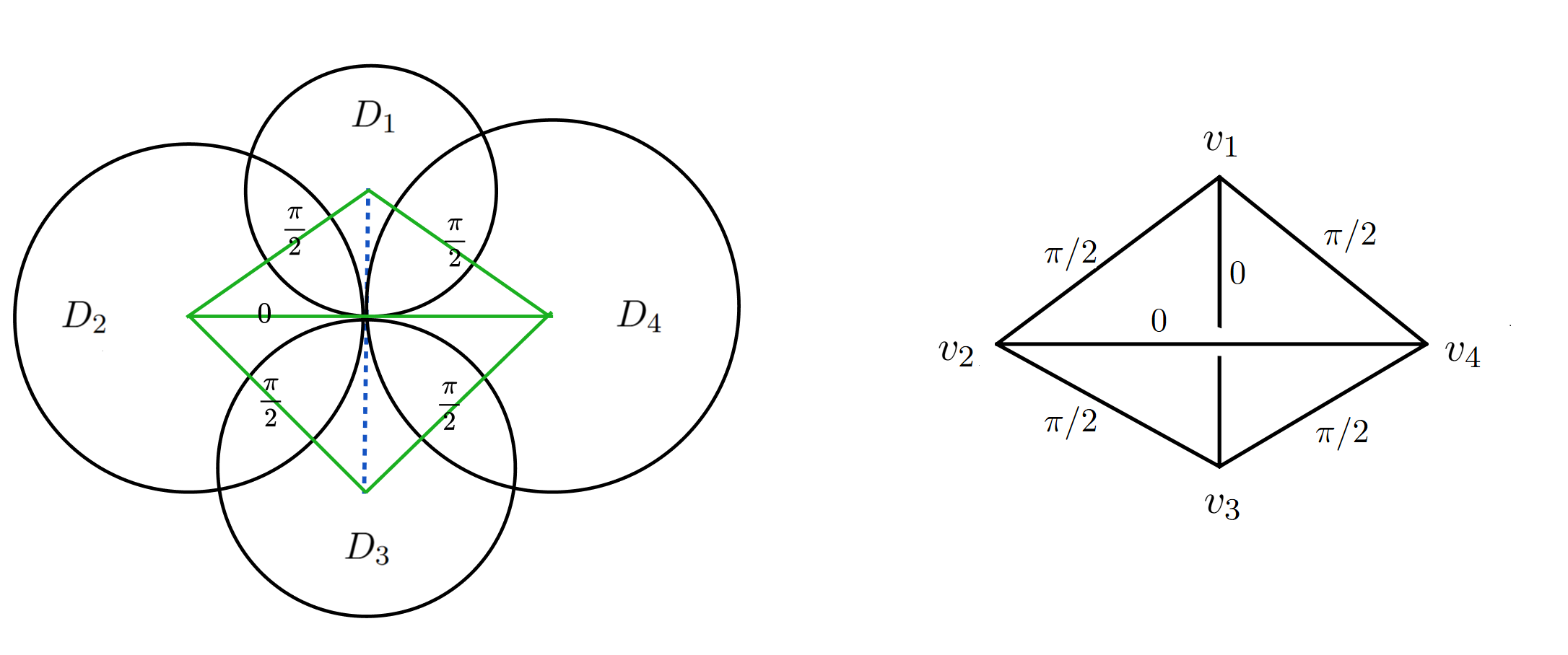}
		\caption{He's reducible configuration (Bowers-Stephenson's extraneous tangencies)}
		\label{subgraphgivingextra}
	\end{figure}

    \begin{lem}\label{lem-explain-he}
    Given a convex quadrilateral $v_1v_2v_3v_4$ in the Euclidean plane, which is composed of two triangles $\Delta v_1v_2v_4$ and $\Delta v_3v_2v_4$ sharing the common edge $v_2v_4$, as shown in Figure \ref{He-reduce-config}. The side lengths of each edge $ij\in\{12, 23, 34, 41, 24\}$ are known to be $$l_{ij}=\sqrt{r_i^2+r_j^2+2r_ir_j\cos\Theta_{ij}},$$ 
    where $\Theta_{ij}\in[0,\pi/2]$ are given. Then, for all $r_1$, $r_2$, $r_3$, $r_4>0$, we have $l_{13}\geq r_1+r_3$. Moreover, the equality holds if and only if $\Theta_{12}=\Theta_{23}=\Theta_{34}=\Theta_{41}=\pi/2$ and $\Theta_{24}=0$.
    \end{lem}

    The proof of Lemma \ref{lem-explain-he} is quite elementary, and is therefore omitted. Lemma \ref{lem-explain-he} explains the unique configuration of the reducible edges of He. To be precise, for the operation of adding reducible edges to $G$ to obtain $\tilde{G}$, the added reducible edges (from a local perspective) can only appear in the configuration where $\Theta_{12}=\Theta_{23}=\Theta_{34}=\Theta_{41}=\pi/2$ and $\Theta_{24}=0$, as shown in Figure \ref{subgraphgivingextra}. 
    However, in case $\Theta\in[0,\pi)^E$, the following example shows that there is no such canonical construction to add reducible edges that is independent of the disk radii.
	
	\begin{ex}\label{ex1234}
	\label{ex-four-circle}
    We still consider the quadrilateral contact graph $G$ configured by four disks shown in Figure \ref{He-reduce-config}. Assume that the four disks have radii $r_1=r_3=x$, $r_2=y$, $r_4=z$, and have intersection angles 
    $\Theta_{12} = \Theta_{14} = \Theta_{24}=\Theta_{23}=\Theta_{34}=2\pi/3$. Obviously, the five intersection angles satisfy the inequality (\ref{angle-condition}). By Lemma \ref{sanyuangouxing_yinli} in the next section, the two Euclidean triangles $\Delta v_1v_2v_4$ and $\Delta v_3v_2v_4$ are uniquely determined for any $x, y, z >0$. We may calculate the side lengths $l_{ij}=\sqrt{r_i^2+r_j^2+2r_ir_j\cos\Theta_{ij}}$ and compare the relationship between $l_{13}$ and $r_1+r_3$ in the following three cases respectively:

   \medskip
   \emph{Case 1}. $x=y=z=1$, direct calculation shows that $l_{13}=\sqrt{3}<2=r_1+r_3$, which implies $D_1 \cap D_3\neq \emptyset$ and $\Theta_{13}=\pi/3$. This indicates that the edge $v_1v_3$ indeed exists, and therefore should be added to the graph $G$ with an assigned intersecting angle of $\pi/3$.
   
    \medskip
    \emph{Case 2}. $x=y = 1$, $z= 2$, direct calculation shows that $l_{13} = \sqrt{11/3} < r_1 + r_3$, which implies $D_1 \cap D_3 \neq \emptyset$ and $\Theta_{13} = \cos^{-1}(5/6)$. This indicates that the edge $v_1v_3$ also exists, and therefore should be added to the graph $G$ with an assigned intersecting angle of $\cos^{-1}(5/6)$.
    
    \medskip
    \emph{Case 3}. $x= 1$, $y=z=4$, direct calculation shows that $l_{13} = 6 > 2 = r_1 + r_3$, which implies $D_1 \cap D_3 = \emptyset$. This indicates that the edge $v_1v_3$ does not exist in this case, so there is no need to add it to graph $G$.
    \end{ex}
    
    \begin{figure}[h]
		\centering
		\includegraphics[width=0.32\textwidth]{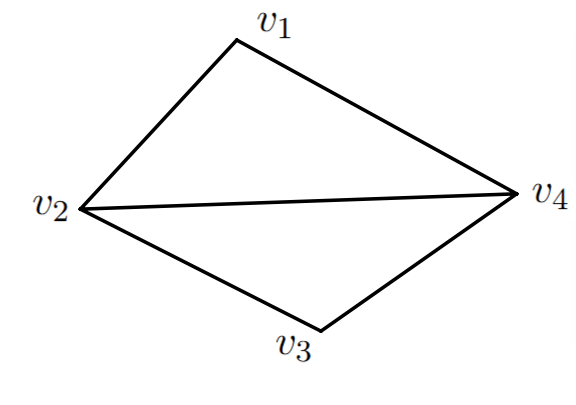}
		\caption{two adjacent triangle configuration}
		\label{He-reduce-config}
	\end{figure}
    
    Example \ref{ex-four-circle} shows that in the case of $\Theta\in[0,\pi)^E$, the intersection and inclusion relationships between disks depend not only on $(G, \Theta)$, but also on the specific radii. Due to Lemma \ref{lem-explain-he}, this phenomenon does not occur in the case $\Theta\in[0,\pi/2]^E$, where such relationships are determined solely by $(G, \Theta)$. This phenomenon makes the obtuse angle case significantly more difficult to handle, rendering it impossible to add reducible edges without considering the disk radii.
	
	Now, come back to Problem \ref{problem-realize}. It is well posed when further assuming $\Theta\leq\pi/2$. In fact, by Koebe-Andreev-Thurston's circle packing theorem, He-Schramm \cite{He1,He2}, and He's work (see \cite[Uniformization Theorem 1.3]{He}  and the explanation in the preceding paragraph), the existence problem has a complete answer, that is, there is a disk pattern that realizes the data $(G,\Theta)$ if and only if the following conditions hold:
	\begin{itemize}
		\item[($C_1$)] If a simple loop in $G$ formed by three edges $e_1$, $e_2$, $e_3$ separates the vertices of $G$ (i.e. there is a pair of vertices in the complement of the subset so that any path joining the vertices passes through the subset), then $\Theta(e_1)+\Theta(e_2)+\Theta(e_3)<\pi$;  
		\item[($C_2$)] If $v_1$, $v_2$, $v_3$, $v_4=v_0$ are distinct vertices in $G$ and if $[v_{i-1},v_i]\in E$ and $\Theta([v_{i-1},v_i])=\pi/2$, $i=1,\cdots,4$, then either $[v_1,v_3]$ or $[v_2,v_4]$ is an edge in $G$.
	\end{itemize}
	Moreover, the issue of uniqueness has also been understood quite clearly by Thurston \cite{Th76} and He-Liu \cite[Theorem 1.3]{He-Liu-2013} when $G$ is finite, and by He \cite[Rigidity Theorem 1.1, 1.2]{He} when $G$ is an infinite disk triangulation graph. It is worth highlighting that according to He-Liu's work \cite[Theorem 1.3]{He-Liu-2013}, the uniqueness is not valid whenever there is an interstice in $\pac$ with more than three sides. Compared to condition $(C_2)$, Example \ref{ex-four-circle} suggests that Problem \ref{problem-realize} is not well-posed in the general case $0\leq\Theta<\pi$. Moreover, there are no sufficient and necessary conditions as $(C_1)$ and $(C_2)$ to resolve the existence issue. Therefore, it is quite natural to assume that $G$ is a disk triangulation graph and to rule out the occurrence of such singular configurations as ``reducible edges'', to reach more accurate conclusions about Problem \ref{problem-realize}. This naturally gives rise to the following concept.

	
	\begin{defn}[RCP]
		\label{RCP_def}
		A circle pattern $\pac=\{C_v\}_{v\in V}$ is called a \textbf{regular circle pattern} (abbr. \textbf{RCP}) if its contact graph $G(\pac)=(V,E)$ is a \textbf{disk triangulation graph}, i.e. $G(\pac)$ is induced by some \textbf{disk triangulation} $\mathcal{T}$, where $\mathcal{T}= (V, E, F)$ is an infinite and locally finite triangulation of the open unit disk $\mathbb U$. Moreover, the geometric realization obtained by joining the centers of adjacent disks gives a planar embedding compatible with the triangulation.
Here $V$, $E$ and $F$ denote the set of vertices, edges, and faces in $\mathcal T$ respectively, and $\mathcal T$ is called locally finite if each vertex has a finite degree. 
	\end{defn}
	
	
	The definition of RCP excludes the kind of singular configuration shown in Figure \ref{reduce-edge}, in which four disks intersect each other in pairs, and the edges connecting their centers $v_1$, $v_2$, $v_3$, $v_4$, $v_1$ sequentially form a convex quadrilateral. Suppose $(G,\Theta)$ is weakly realized by a circle pattern $\pac$. Example \ref{ex-four-circle} shows that if $\Theta<\pi/2$, $\pac$ is always regular; whereas if $\Theta\leq\pi/2$, the only local configuration that makes $\pac$ irregular is precisely He’s reducible example, i.e., the local configuration shown in Figure \ref{subgraphgivingextra} (with $\Theta_{12}=\Theta_{23}=\Theta_{34}=\Theta_{41}=\pi/2$ and $\Theta_{24}=\Theta_{13}=0$, and this local configuration is also called extraneous tangencies by Bowers-Stephenson \cite{BS96}). 
	Furthermore, we have another motivation, which is to characterize the rigidity and existence of infinite hyperbolic polyhedra. According to the correspondence between hyperbolic polyhedra and circle patterns discussed in Section \ref{polyhedra_section}, in a hyperbolic polyhedron, it is impossible for a side to correspond to a ``reducible edge" in the circle pattern. From this perspective, the above definition is also natural, reasonable, and just right. The carrier of a circle pattern $\mathcal P$, denoted by $\operatorname{Carr}(\mathcal P)$,
    is defined as the union of all closed disks in $\mathcal P$ together with all
    closed geometric triangles $\Delta v_i v_j v_k$ corresponding to triangular
    faces of the underlying triangulation.
	If $\pac$ is regular, one can directly verify that there is a unique domain $\Omega \subset \hat{\mathbb{C}}$ such that $\pac$ is locally finite in $\Omega$. Moreover, $\partial \Omega$ coincides with the set of accumulation points of $\pac$ and $\Omega = \carrier(\pac)$.


	\emph{Notational conventions}. We provide some notations and conventions here, although some of them have appeared before. We use $v_i$, $e_{ij}$, and $\Delta v_iv_jv_k$, where $v_i$, $v_j$, $v_k \in V$, to present a special vertex, an edge, and a face, respectively. Most of the time, for convenience, we abbreviate  $v_i$, $e_{ij}$, and $\Delta v_iv_jv_k$ as $i$, $[i,j]$, and $[i,j,k]$ respectively. Sometimes we write $[i,j]\in E$ (or $i\sim j$) if $i$ and $j$ are connected by an edge in $E$, and write $[i,j]\notin E$ if there is no edge in $E$ connecting $i$ and $j$. Note that the notation $[i,j]\notin E$ only appears as a whole, indicating that there is no edge between $i$ and $j$ when written this way. If only $[i,j]$ (or $[i,j]\in E$, $i\sim j$) appears, it indicates that $i$ and $j$ are connected by an edge. In addition, regarding the notations for vertices, edges, and faces (or triangles) mentioned above, we do not distinguish whether they appear in a circle pattern $\pac$, in the contact graph $G(\pac)$ of a circle pattern $\pac$, or in an abstract triangulation graph $G$. One can always discern the specific meaning from the context. 
    A loop in a connected graph $G=(V, E)$ is a closed path in $G$, and we denote by $d\left(v_1, v_2\right)$ the combinatorial distance between $v_1$ and $v_2$, i.e., the length of the shortest path connecting $v_1$ and $v_2$ in $G$. We also denote by $d_{\mathbb{R}^2}(x, y)$ the Euclidean distance between $x$ and $y$ in $\mathbb{R}^2$.

	\subsection{Thurston's construction for CPs}
    \label{section-thurston-cp}
	In this subsection, we recall Thurston's geometric construction for circle patterns \cite[Chap. 13]{Th76}. Let $\mathcal T=(V,E,F)$ be an infinite and locally finite triangulation of the open unit disk $\mathbb{U}$ ($\mathcal{T}$ is also a triangulation of $\mathbb{C}$ since $\mathbb{U}$ and $\mathbb{C}$ are homeomorphic), where the set of vertices, edges, and faces is denoted by $V$, $E$ and $F$ respectively. For the convenience of expression, we write $V$ in the form of $V=\big\{(\cdots,v_i\,,\cdots)\big\}$. By definition, a \textbf{circle pattern metric} $r=(\cdots,r_i\,,\cdots)\in \mathbb R_{+}^{V}$ assigns each vertex $i$ a positive number $r_i$. Given an intersection angle function $\Theta\in[0,\pi)^E$ satisfying suitable conditions (such as ($Z_1$) or ($Z_4$)), each CP-metric $r$ produces a Euclidean (or hyperbolic resp.) cone metric on $\mathbb{C}$ (or $\mathbb{U}$ resp.) as follows.
	
	For each combinatorial triangle $\Delta v_iv_jv_k$ of $\mathcal T$, one associates it with a Euclidean (or hyperbolic resp.) triangle formed by linking the centers of three Euclidean (or hyperbolic resp.) disks of radii $r_{i}, r_{j}, r_{k}$ with intersection angles $\Theta([v_i,v_j]), \Theta([v_j,v_k]),\Theta([v_k,v_i])$ respectively. More precisely, let $l_{ij}$ be the length of the edge $[i,j]$. Then
	\begin{equation}
		l_{ij}\,=\,\sqrt{r_i^2+r_j^2+2r_ir_j\cos\Theta([v_i,v_j])}
	\end{equation}
	in Euclidean background geometry, or
	\begin{equation}
		l_{ij}\ =\ \cosh^{-1}\big(\cosh r_{i}\cosh r_{j}+\sinh r_{i}\sinh r_{j}\cos\Theta([v_i,v_j])\big)
	\end{equation}
	in the hyperbolic background geometry, see Figure \ref{two-disk-reduce-config}. Under the conditions ($Z_1$) or ($Z_4$), for any three positive numbers $r_{i}, r_{j}, r_{k}$, the corresponding $l_{ij}, l_{jk}, l_{ki}$
	satisfy the triangle inequalities, see Lemma \ref{sanyuangouxing_yinli} and Corollary \ref{cor-cos-three-angle-config} below. Thus up to isometry, there exists a unique Euclidean (or hyperbolic resp.) triangle of edge lengths $l_{ij}, l_{jk}, l_{ki}.$
	
	Gluing all these Euclidean (or hyperbolic resp.) triangles along the common edges produces a Euclidean (or hyperbolic resp.) cone metric on $\mathbb{C}$ (or $\mathbb{U}$ resp.) with possible cone singularities at the vertices of $\mathcal T$. For each $i\in V$, the discrete Gauss curvature (also called combinatorial Gauss curvature), is defined as follows:
	\begin{equation}
		K_i = 2\pi - \sum_{[i, j, k] \in F} \vartheta_i^{jk},
	\end{equation}      
	where \( \vartheta_i^{jk} \) is the angle at \( v_i \) between edges \([v_i, v_j]\) and \([v_i, v_k]\) in the triangle \(\Delta v_iv_jv_k\), and the sum is taken over all triangles incident to the vertex $v_i$.
	
	Clearly, the discrete Gauss curvature $K_i$ at the vertex $v_i$ is a smooth function of $r$ (in fact, it is a function of $r_i$ and those $r_j$, where $v_j$ are the vertices connected to $v_i$). This gives rise to a curvature map
	$Th(\cdot):\mathbb{R}_{+}^{V}\to\mathbb{R}^{V}$, $(\cdots,r_i\,,\cdots)\mapsto (\cdots, K_i\,,\cdots)$. 
	
	The major purpose is to show that the origin $(0,0,\cdots)$ belongs to the image of the curvature map $Th$. If there exists a CP-metric $r^\ast$ such that  $K_i(r^\ast)=0$ for $v_i\in V$, then it produces a smooth Euclidean (or hyperbolic resp.) metric on $\mathbb{C}$ (or $\mathbb{U}$ resp.). Drawing the circle centered at $v_i$ of radius $r^{\ast}_i$, one will obtain the desired circle pattern realizing $(\mathcal T,\Theta)$.

	\subsection{Some properties of three-circle configurations}
	\label{section-three-circle-config}
	We first provide some useful lemmas on three-circle configurations (see Figure \ref{tu_sanyuangouxing} for example). We remark that some relevant results have already appeared in the works of Thurston \cite{Th76}, Marden-Rodin \cite{MR90}, Bowers-Stephenson \cite{BS96}, He \cite{He}, Chow-Luo \cite{CL03}, Ge-Hua-Zhou \cite{GHZ21}, Xu \cite{Xu18}, Zhou \cite{Zhou19}-\cite{Zhou23} and others. Suppose that $\Theta_{i}, \Theta_{j}, \Theta_{k} \in[0, \pi )$ are three real numbers. Consider the following four properties:
	\begin{equation}
		\label{sum<pi}
		\Theta_{i}+\Theta_{j}+\Theta_{k} \leq \pi,
	\end{equation}
	\begin{equation}
		\label{angle-condition}
		\left\{
		\begin{aligned}
			\Theta_{i}+\Theta_{j} & < \pi+\Theta_{k}\\
			\Theta_{j}+\Theta_{k} & < \pi+\Theta_{i}\\
			\Theta_{k}+\Theta_{i} & < \pi+\Theta_{j},
		\end{aligned}
		\right.
	\end{equation}
	
	\begin{equation}
		\label{cos-condition}
		\left\{
		\begin{aligned}
			\cos \Theta_i+\cos \Theta_j \cos \Theta_k & \geq 0\\
			\cos \Theta_j+\cos \Theta_k \cos \Theta_i & \geq 0\\
			\cos \Theta_k+\cos \Theta_i \cos \Theta_j & \geq 0,
		\end{aligned}
		\right.
	\end{equation}
	and
	\begin{equation}
		\label{equ-sphere-triangle}
		\Theta_{i}, \Theta_{j}, \Theta_{k}~\text{are the angles of a spherical triangle with each side length} \leq\frac{\pi}{2}.
	\end{equation}	
	We have:
	\begin{lem}(\cite[Proposition 2.7]{Zhou19})\label{C1_C1}
		Given $\Theta_i, \Theta_j, \Theta_k \in[0, \pi)$. Then 
		(\ref{cos-condition}) $\iff$ (\ref{sum<pi}) or (\ref{equ-sphere-triangle}). To be precise, (\ref{cos-condition}) holds if and only if at least one of (\ref{sum<pi}) or (\ref{equ-sphere-triangle}) holds.
	\end{lem}

	From Lemma \ref{C1_C1}, we see that condition ($Z_4$) implies condition ($Z_1$). We note that (\ref{cos-condition}) was first introduced by Zhou in \cite{Zhou19}, which may seem a bit unintuitive. On the other hand, (\ref{sum<pi}) and (\ref{equ-sphere-triangle}) can be seen as a geometric interpretation of property (\ref{cos-condition}). Therefore, under the assumption of $\Theta_{i}+\Theta_{j}+\Theta_{k}>\pi$,  (\ref{cos-condition}) becomes quite intuitive. 
	
	\begin{lem}(Configuration of three-circles, \cite[Lemma 2.1]{JLZ20})\label{sanyuangouxing_yinli}
		Given $\Theta_{i}, \Theta_{j}, \Theta_{k} \in[0, \pi )$, if they satisfy either (\ref{sum<pi}) or (\ref{angle-condition}), then for any $r_{i}, r_{j}, r_{k} \in (0, \infty)$, there exists a configuration of three intersecting circles in both Euclidean and hyperbolic geometries (see Figure \ref{tu_sanyuangouxing}), unique up to isometry, having radii $r_{i}, r_{j}, r_{k}$ and meeting in exterior intersection angles $\Theta_{i}, \Theta_{j}, \Theta_{k}$. 
	\end{lem}

	\begin{lem}
		Given $\Theta_{i}, \Theta_{j}, \Theta_{k} \in[0, \pi )$, then (\ref{equ-sphere-triangle}) $\Rightarrow$ (\ref{angle-condition}). To be precise, if (\ref{equ-sphere-triangle}) holds, then both $\Theta_i, \Theta_j, \Theta_k >0$, $\Theta_{i}+\Theta_{j}+\Theta_{k}>\pi$ and (\ref{angle-condition}) hold.
	\end{lem}

	\begin{cor}
		\label{cor-cos-three-angle-config}
		Given $\Theta_{i}, \Theta_{j}, \Theta_{k} \in[0, \pi )$, then  (\ref{cos-condition}) $\Rightarrow$ (\ref{sum<pi}) or (\ref{angle-condition}). To be precise, if (\ref{cos-condition}) holds, then at least one of (\ref{sum<pi}) or (\ref{angle-condition}) holds. Therefore, by Lemma \ref{sanyuangouxing_yinli}, when assuming (\ref{cos-condition}), for any $r_{i}, r_{j}, r_{k} \in (0, \infty)$, there exists a configuration of three intersecting circles in both Euclidean and hyperbolic geometries, unique up to isometry, having radii $r_{i}, r_{j}, r_{k}$ and meeting in exterior intersection angles $\Theta_{i}, \Theta_{j}, \Theta_{k}$. 
	\end{cor}

	\begin{figure}[h]
		\centering
		\includegraphics[width=0.45\textwidth]{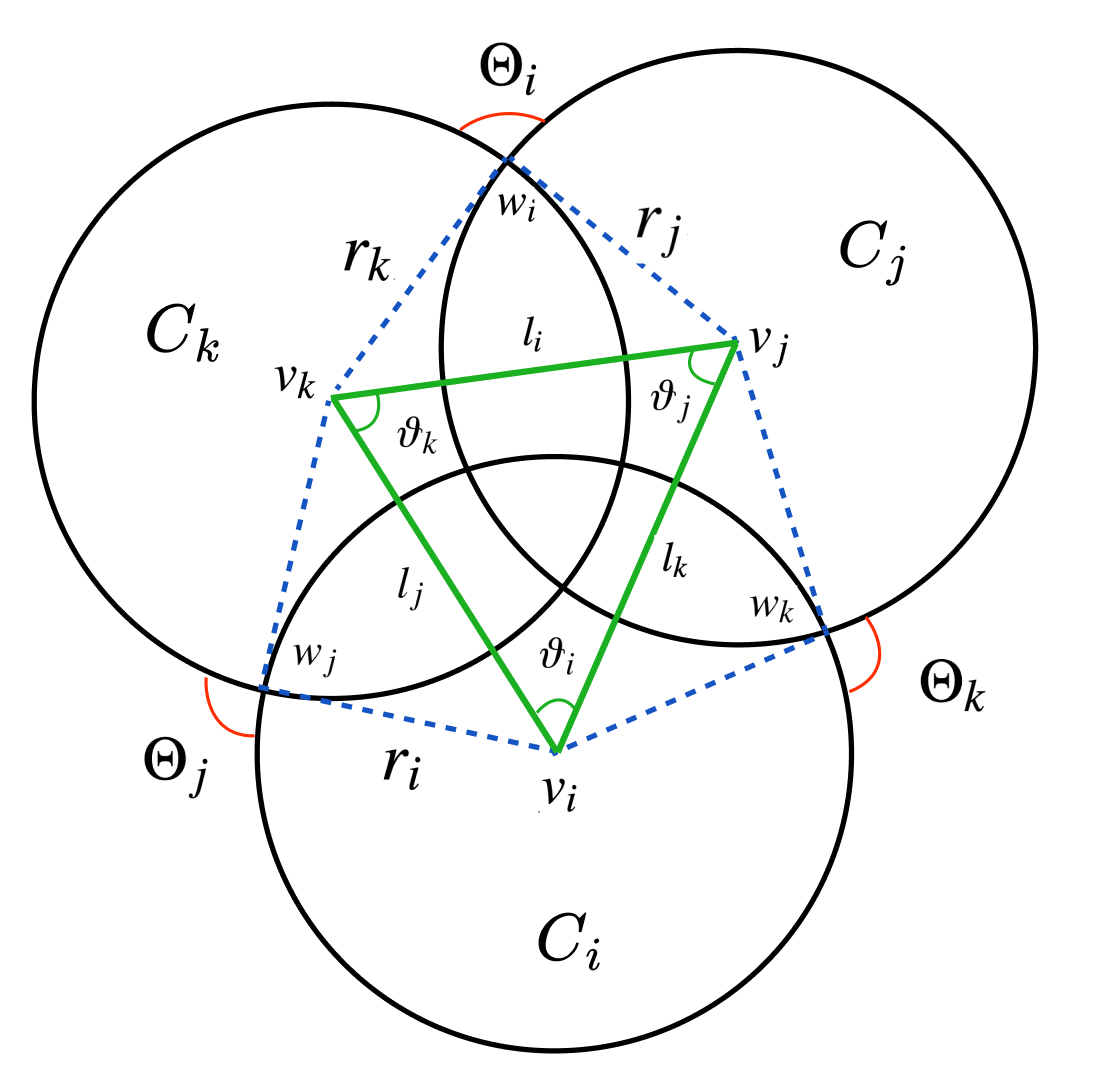}
		\caption{three circle configuration }
		\label{tu_sanyuangouxing}
	\end{figure}
	
	\begin{rem}
		Generally, one cannot derive (\ref{cos-condition}) from $\Theta_{i}+\Theta_{j}+\Theta_{k}>\pi$ and property (\ref{angle-condition}). For example, take one-eighth of $\mathbb{S}^2$ as the initial spherical triangle $\Delta_{ijk}$, and slightly move the third vertex $k$ outward on the sphere so that the three interior angles are $\Theta_i=\Theta_j=\frac{\pi}{2}+\epsilon$, $\Theta_k=\frac{\pi}{2}-\delta_{\epsilon}$ respectively. At this point, one of the three sides of the triangle is $\frac{\pi}{2}$, while the other two are slightly greater than $\frac{\pi}{2}$. It is easy to see $\Theta_{i}+\Theta_{j}+\Theta_{k}>\pi$ and that (\ref{angle-condition}) are true, but $\cos \Theta_i+\cos \Theta_j \cos \Theta_k<0$.
	\end{rem}

	Given three angles $\Theta_i, \Theta_j, \Theta_k \in[0, \pi)$ that satisfy (\ref{cos-condition}). For $r_i$, $r_j$, $r_k>0$, consider the three-circle configuration in Figure \ref{tu_sanyuangouxing}. By Corollary \ref{cor-cos-three-angle-config}, the three-circle configuration exists. Moreover, by \cite{Xu18,Zhou19}, the common intersection point of the radical axes of three mutually intersecting circles lies inside or on the boundary of the triangle $\Delta v_iv_jv_k $.

	\begin{lem}(\cite[Lemma 1.3]{GHZ21})\label{vari}
		Suppose that $\Theta_i, \Theta_j, \Theta_k \in[0, \pi)$ satisfies (\ref{cos-condition}). The three interior angles of the triangle \(\Delta v_iv_jv_k\) are denoted by \( \vartheta_i\), \( \vartheta_j\), and \( \vartheta_k\), and they are all elementary functions of $r_i$, $r_j$, and $r_k$. Then we have:
		\begin{itemize}
			\item[(1)]  in Euclidean geometry,
			$$
			\frac{\partial \vartheta_i}{\partial r_i}<0, \quad r_j \frac{\partial \vartheta_i}{\partial r_j}=r_i \frac{\partial \vartheta_j}{\partial r_i} \geq 0.
			$$
			\item[(2)]  in hyperbolic geometry,
			$$
			\frac{\partial \vartheta_i}{\partial r_i}<0, \quad \sinh r_j \frac{\partial \vartheta_i}{\partial r_j}=\sinh r_i \frac{\partial \vartheta_j}{\partial r_i} \geq 0, \quad \frac{\partial \operatorname{Area}\left(\Delta v_iv_jv_k\right)}{\partial r_i}>0,
			$$
			where Area $\left(\Delta v_iv_jv_k\right)$ denotes the area of $\Delta v_iv_jv_k$.
		\end{itemize}
	\end{lem}
   
	\qed
 
 In particular, set $u_i=\ln r_i$ (or $\ln\tanh{(r_i/2)}$ resp.) in the Euclidean (or hyperbolic resp.) background geometry, then we have
    \begin{align}\label{conformalfactor}
    \pp{\vartheta_i}{u_j}=\pp{\vartheta_j}{u_i},
    \end{align}	
where $u_i$ is called the discrete conformal factor at the vertex $i$ in literature.

	For brevity, write $C(x, r)$ for the circle centered at $x\in\mathbb{C}$ with radius $r$, and denote $C(r)$ by the circle centered at the origin with radius $r$. In addition, denote $D(x, r)$ and $D(r)$ by the closed disk bounded by the circle $C(x, r)$ and $C(r)$ respectively. The next two lemmas reveal more relationships between the positions of the three circles in the three-circle configuration and the aforementioned properties (\ref{sum<pi}) and (\ref{angle-condition}). 
	
	\begin{figure}[h]
		\centering
		\includegraphics[width=0.8\textwidth]{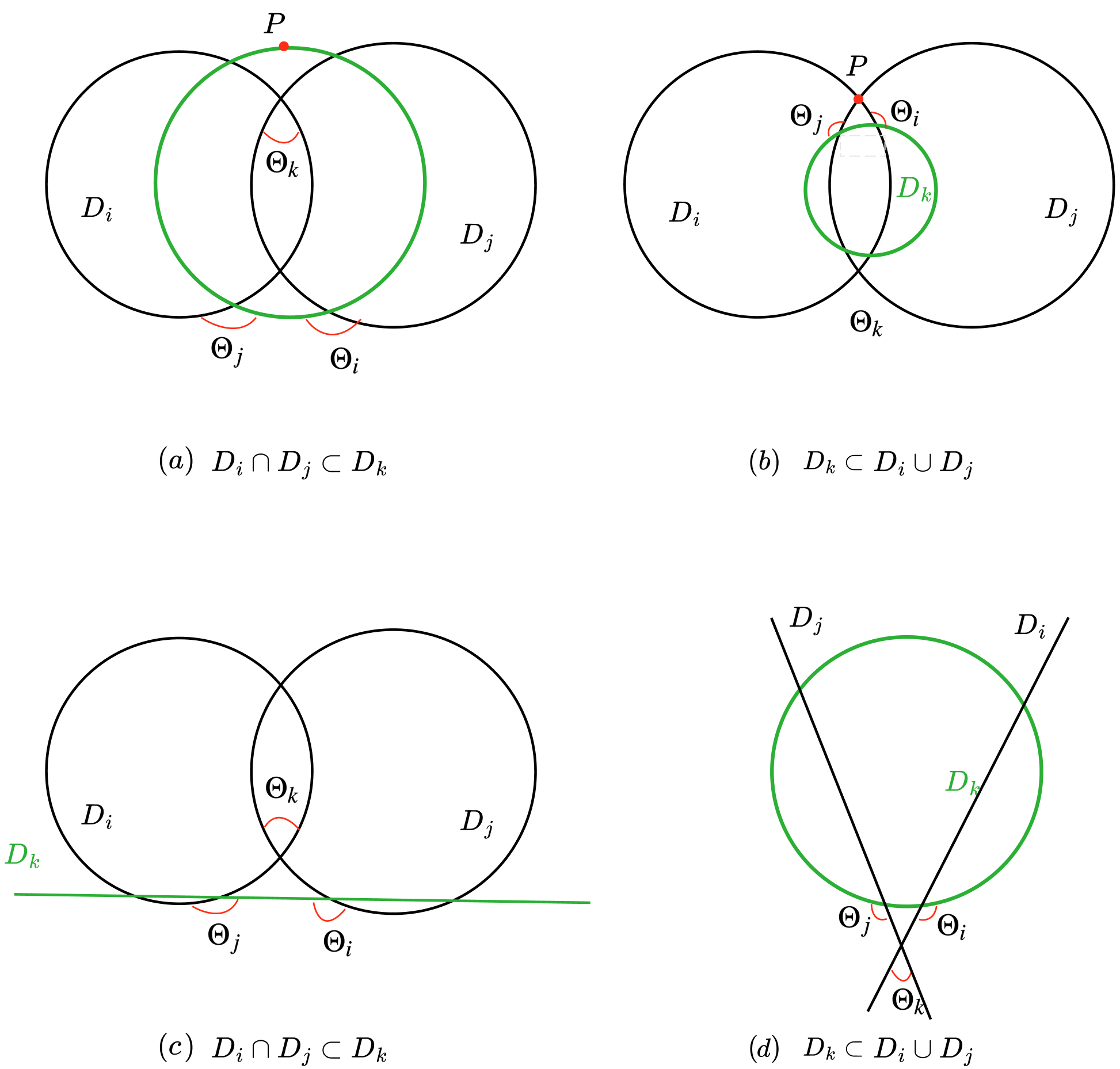}
		\caption{critical three-circle configurations.}
		\label{containlemmafig}
	\end{figure}
	
	\begin{lem}\label{containlemma}
		As shown in Figure \ref{tu_sanyuangouxing}, let $D_i, D_j, D_k$ be three intersecting disks with intersection angles $\Theta_i, \Theta_j, \Theta_k \in [0, \pi)$. Whether it is $D_i \cap D_j \subset D_k$ (see Figure \ref{containlemmafig}-(a)) or $D_k \subset D_i \cup D_j$ (see Figure \ref{containlemmafig}-(b)), the following conclusion holds:
		\[
		\Theta_i + \Theta_j \ge \Theta_k + \pi.
		\]
	\end{lem}
	
	\begin{proof}
		There exists a M\"obius transformation $M$ sending a point $p$ to infinity, such that the configuration in Figure \ref{containlemmafig}-(a) is transformed into that in Figure \ref{containlemmafig}-(c). It is clear that
		\[
		\Theta_i + \Theta_j \ge \Theta_k + \pi.
		\]
		Similarly, there exists a M\"obius transformation $M'$ sending a point $p$ to infinity, transforming the configuration in Figure \ref{containlemmafig}-(b) into that in Figure \ref{containlemmafig}-(d).  
		Then we obtain
		\[
		\Theta_i + \Theta_j \ge \Theta_k + \pi.
		\]
	\end{proof}
    By the same argument as in the proof of Lemma~\ref{containlemma}, we have
    \begin{cor}\label{containlemmacor}
        As shown in Figure \ref{tu_sanyuangouxing}, let $D_i, D_j, D_k$ be three intersecting disks with intersection angles $\Theta_i, \Theta_j, \Theta_k \in [0, \pi)$. If $C_i \cap C_j \subset D_k$, the following conclusion holds:
		\[
		\Theta_i + \Theta_j \ge \Theta_k + \pi.
		\]
    \end{cor}
	
	\begin{lem}\label{angleleqpilemma}
		As shown in Figure \ref{tu_sanyuangouxing}, let $D_i, D_j, D_k$ be three intersecting disks with intersection angles $\Theta_i, \Theta_j, \Theta_k \in [0, \pi)$. Then $\Theta_i+\Theta_j+\Theta_k<\pi$ if and only if
		$$
		D_i \cap D_j \cap D_k=\emptyset.
		$$
	\end{lem}
	\begin{proof}
		We divide the proof into two steps.
		
		\medskip
		\emph{Step 1. If $\Theta_i+\Theta_j+\Theta_k<\pi$, then $
			D_i \cap D_j \cap D_k=\emptyset$.}
		\medskip

		Fix $r_i, r_j, r_k>0$. Define
		\[
		\Theta_i(t)=t\Theta_i, \quad \Theta_j(t)=t\Theta_j, \quad \Theta_k(t)=t\Theta_k,
		\]
		where $t \in [0,t_0]$ and $t_0(\Theta_i+\Theta_j+\Theta_k)=\pi$. Clearly, $t_0>1$.
		
		By Lemma~\ref{sanyuangouxing_yinli}, for each $t \in [0,t_0]$ there exists a unique configuration of three circles with radii $r_i, r_j, r_k$ and intersection angles $\Theta_i(t), \Theta_j(t), \Theta_k(t)$. Since the configuration varies continuously with $t$, and
		\[
		D_i(0)\cap D_j(0)\cap D_k(0)=\emptyset,\]
		\[D_i(t_0)\cap D_j(t_0)\cap D_k(t_0)=\{p\},
		\]
		where $p$ is a single point in $\mathbb C$. For contradiction, suppose that the statement is false. Then there exists a smallest $t_1\le 1$ such that 
		\[
		D_i(t_1)\cap D_j(t_1)\cap D_k(t_1)
		\]
		consists of exactly one point. At this critical parameter, it follows that 
		$t_1(\Theta_i+\Theta_j+\Theta_k)=\pi$, implying $t_1=t_0$. However, this contradicts $t_1\le 1<t_0$. Hence, when $t=1$, we must have 
		$$D_i\cap D_j\cap D_k=\emptyset.$$
		
		\medskip
		\emph{Step 2. If $D_i \cap D_j \cap D_k=\emptyset$, then $\Theta_i+\Theta_j+\Theta_k<\pi$.}
		\medskip
		
		The conclusion follows directly from the geometric configuration illustrated in Figure \ref{localconfig}.
		\begin{figure}[h]
			\centering
			\includegraphics[width=0.3\textwidth]{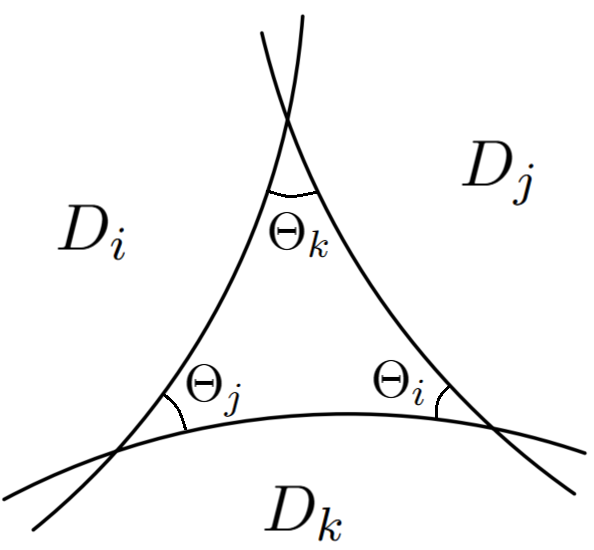}
			\caption{the local configuration of $D_i \cap D_j \cap D_k=\emptyset$.}
			\label{localconfig}
		\end{figure}	
	\end{proof}
	
	Consider the local configuration where three disks $D_i$, $D_j$, and $D_k$ intersect each other in pairs, the open set $\Delta v_iv_jv_k\setminus (D_i\cup D_j\cup D_k)$ (if it is non-empty) is called the \textbf{interstice} bounded by $D_i$, $D_j$, and $D_k$. By Lemmas~\ref{sanyuangouxing_yinli} and \ref{angleleqpilemma}, we have the following corollary.

    \begin{cor}
    \label{cor-three-one}
    Let $D_i, D_j, D_k$ be three disks with intersection angles $\Theta_i, \Theta_j, \Theta_k \in [0, \pi)$. Suppose that  $\Theta_i$, $\Theta_j$, and $\Theta_k$ satisfy at least one of (\ref{sum<pi}) or (\ref{angle-condition}). Then we have:
    		\begin{itemize}
			\item[(1)] if $\Theta_i+\Theta_j+\Theta_k < \pi$, then there is an interstice bounded by $D_i$, $D_j$, and $D_k$; 
			\item[(2)]  if $\Theta_i+\Theta_j+\Theta_k = \pi$, then $D_i \cap D_j \cap D_k = \{p\}$, where $p$ is a point in $\mathbb{C}$; 
			\item[(3)]  if $\Theta_i+\Theta_j+\Theta_k > \pi$, then there is no interstice bounded by $D_i$, $D_j$, and $D_k$. 
		\end{itemize}
    \end{cor}


	\subsection{Some lemmas for CPs and RCPs}\label{someusefullemmas}

	For any subset $A\subset V$, denote $S(A)$ by the union of all open simplexes connected to $A$ in $\mathcal{T}$. Specifically, denote $S\left(i\right)=S\left(\left\{i\right\}\right)$ and $N(i)= \{j\in V: j \sim i \}$ for any $i \in V$. 
	\begin{figure}[h]
		\centering
		\includegraphics[width=0.4\textwidth]{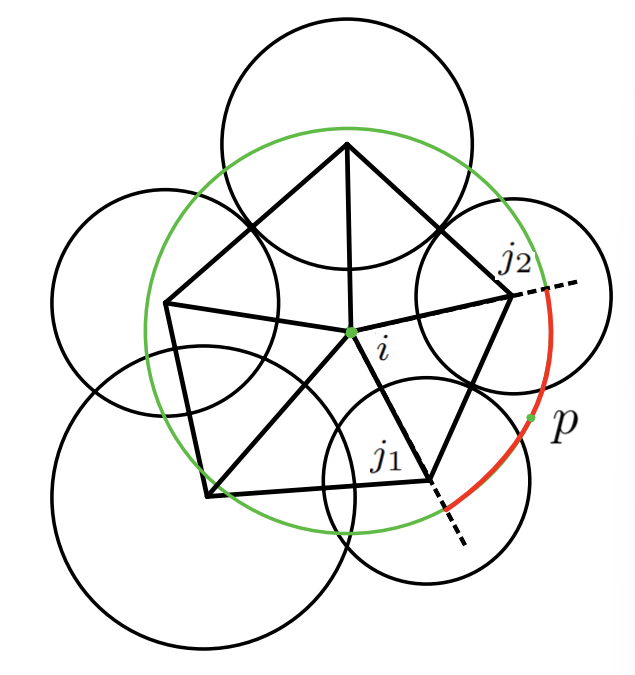}
		\caption{the configuration of $D_i$ and $\big(\bigcup_{j \in N(i)} D_j\big)\cup S(i).$}
		\label{figarea_of_a_point}
	\end{figure}
	\begin{lem}\label{containinFLlemma}
		Let $\mathcal{T}=(V,E,F)$ be a  disk triangulation and let $\Theta\in [0 ,\pi)^E$ be an intersection angle function. Suppose that $\pac = \{D_i\}_{i \in V}$ is a CP that weakly realizes $(\mT, \Theta)$. If further assume that condition ($Z_1$) holds, then
		$$
		D_i \subset \big(\bigcup_{j \in N(i)} D_j\big)\cup S(i).
		$$
	\end{lem}
	\begin{proof}
		We divide the proof into three steps.
		
		\medskip
		\emph{Step 1. We show $\partial D_{i} \subset \bigcup_{j \in N(i)} D_j$}.
		\medskip
		
		We argue by contradiction. Without loss of generality, denote by \( \langle j_1, j_2 \rangle \) the red circular arc, as shown in Figure~\ref{figarea_of_a_point}. 
		Assume that there exists a point \( p \in \langle j_1, j_2 \rangle \) such that \( p \notin \bigcup_{j \in N(i)} D_j \).
		Clearly,
		\[
		D_{j_1} \cap D_{j_2} \subset D_i.
		\]
		By Lemma~\ref{containlemma}, we have
		\[
		\Theta([i,j_1]) + \Theta([i,j_2]) \geq \pi + \Theta([j_1,j_2]),
		\]
		which contradicts condition~($Z_1$). Hence, $\partial D_{i} \subset \bigcup_{j \in N(i)} D_j$.
		
		\medskip
		\emph{Step 2. We show $\partial S(i) \subset \bigcup_{j \in N(i)} D_j$}.
		\medskip
		
		Since $\Theta \in [0,\pi)^E$, the geodesic segment $[j_1,j_2]$ lies within \( D_{j_1} \cup D_{j_2} \). 
		Therefore, we obtain 
		$$\partial S(i) \subset \bigcup_{j \in N(i)} D_j.$$
		
		\medskip
		\emph{Step 3. We prove $D_i \subset \big(\bigcup_{j \in N(i)} D_j\big)\cup S(i)$}.
		\medskip
		
		To prove this, we construct a family of closed Jordan regions. Express the boundaries \(\partial D_i\) and \(\partial S(i)\) in polar coordinates as \(A(\theta)\) and \(B(\theta)\), respectively. Define a family of regions by
		\[
		J_t = \{(\rho, \theta) \mid \rho \leq tA(\theta) + (1-t)B(\theta)\}.
		\]
		This family satisfies
		\[
		J_0 = \overline{S(i)}, \quad J_1 = D_i, \quad \partial J_t \subset \big(\bigcup_{j \in N(i)} D_j\big)\cup S(i).
		\]
		Define 
		\[
		X = \{t \in [0,1] \mid J_t \subset \big(\bigcup_{j \in N(i)} D_j\big)\cup S(i)\}.
		\]
		Clearly, \(X\) is non-empty, open, and closed in \([0,1]\). Hence \(X = [0,1]\). 
		Taking \(t = 1\), we obtain 
		$$D_i \subset \big(\bigcup_{j \in N(i)} D_j\big)\cup S(i),$$
		which completes the proof.
	\end{proof}
	
	\begin{lem}\label{onepointlemma}
		Let $\mathcal{T}=(V,E,F)$ be a disk triangulation and let $\Theta\in [0 ,\pi)^E$ be an intersection angle which satisfies ($Z_1$). Assume $\pac = \{C_i\}_{i \in V}$ is a CP that weakly realizes $(\mT, \Theta)$. Let $i,j \in V$ be not adjacent in $\mathcal{T}$, then for any point $p \in D_i \cap D_j$, there is a vertex $k\in V \backslash\left\{i, j\right\}$ such that $p \in D_k$.
	\end{lem}
	\begin{proof}
		If the lemma is not true, then there exists a point \( p \in D_i \cap D_j \), such that \( p \notin {D}_k \) for any \( k \in V \setminus \{i, j\} \). From Lemma \ref{containinFLlemma}, we know $p\in S(i)$ and $p\in S(j)$. Then 
		$$
		p\in S(i)\cap S(j).
		$$
		However, since \( i \) and \( j \) are not adjacent, we have \( S(i) \cap S(j) = \emptyset \), which is a contradiction.
	\end{proof}
	
	\begin{rem}
		It is remarkable that Lemma \ref{angleleqpilemma}, \ref{containinFLlemma}, and \ref{onepointlemma}, as well as part of Lemma \ref{containlemma}, were previously and essentially proved by Zhou \cite{Zhou21} and Ge-Yu \cite{GY24}.
		For the sake of logical fluency and convenience for readers, we include complete proofs here.
	\end{rem}

	\begin{lem}\label{mostthree}
		Let $\mathcal{T}=(V,E,F)$ be a disk triangulation, and let $\Theta\in [0 ,\pi)^E$ be an intersection angle function that satisfies the condition ($Z_2$). Suppose $\pac=\{C_i\}_{i \in V}$ is an RCP that realizes $(\mathcal{T}, \Theta)$. Then for any four vertices $u,v, w,x$, we have $D_u \cap D_v \cap D_w \cap D_x = \emptyset$.
	\end{lem}
	
	\begin{proof}
		We prove the lemma by contradiction. Suppose that
        \begin{equation}\label{DuDvDwDxnonemp}
            D_u\cap D_v\cap D_w\cap D_x\neq \emptyset .
        \end{equation}
        Since $\mathcal P$ is an RCP, we know $u,v, w,x$ are pairwise connected in $\mT$. Since  $\mT$ is a planar triangulation, then without loss of generality,  we can assume that the edges $[u,v]$, $[v,w]$ and $[u,w]$ constitute a simple closed curve that is not the boundary of a face. From ($Z_2$), we have $\Theta([u,v])+ \Theta([v,w])+ \Theta([u,w])<\pi$. Then, by Lemma \ref{angleleqpilemma}, we know $D_u \cap D_v \cap D_w = \emptyset$, which contradicts \eqref{DuDvDwDxnonemp}.
	\end{proof}
	The following lemma plays an extremely important role in this article. It shows that when conditions ($Z_1$) and ($Z_2$) are assumed, all disks in an RCP are dispensable.
	\begin{lem}\label{indispensable}
		Let $\mathcal{T}=(V,E,F)$ be a disk triangulation, and let $\Theta\in [0 ,\pi)^E$ be an intersection angle function that satisfies conditions ($Z_1$) and ($Z_2$). Suppose $\pac=\{C_i\}_{i \in V}$ is an RCP that realizes $(\mathcal{T}, \Theta)$, then for any $i \in V$, we have
		$$
		D_i \not \subset \bigcup_{j \neq i} D_j.
		$$
		
	\end{lem}
	\begin{proof}
		We prove this lemma by contradiction. Suppose that there is a vertex $i \in V$ such that 
		$$
		D_i\subset \bigcup_{j \neq i} D_j.
		$$
		Let $N(i)= \{j\in V: j \sim i \}$, since the circle pattern  $\pac$ is an RCP, by the properties of RCP, we know 
		\begin{equation}\label{containeq}
			D_i\subset \bigcup_{j \in N(i)} D_j.
		\end{equation}
		In the following, we will prove $D_i\subset D_{j_1}\cup D_{j_2}$ for some $j_1,j_2\in N(i)$ with $j_1\sim j_2$.

        \begin{figure}[h]
			\centering
			\includegraphics[width=0.3\textwidth]{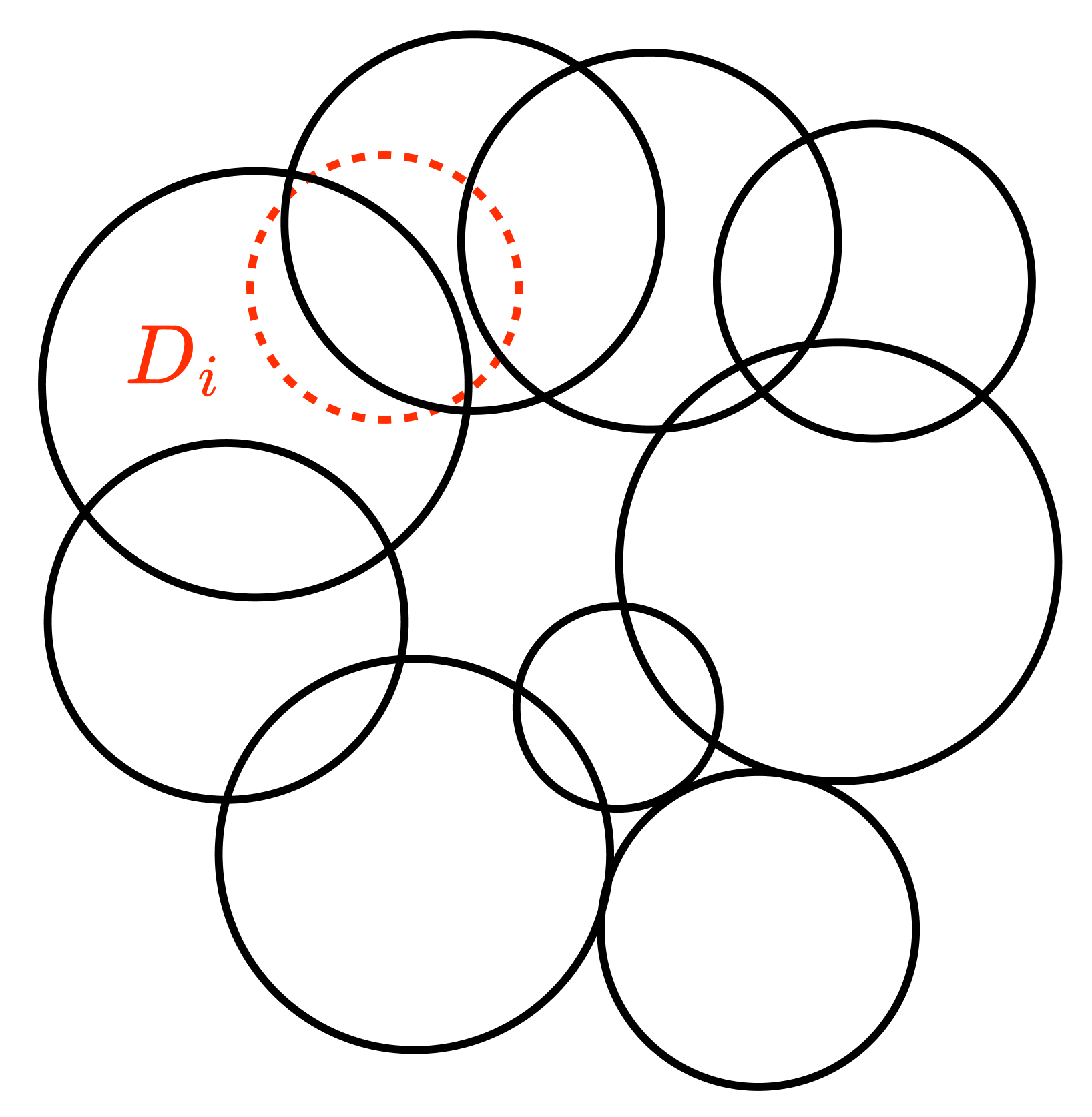}
			\caption{the configuration of $D_i\subset \bigcup_{j \in N(i)} D_j$}
			\label{n(i)}
		\end{figure}
        
		For each edge $[l,m]$ with ends $l,m\in N(i)$, let $K_{l,m}=D_l\cap D_m\backslash \cup_{s\in N(i)\backslash\{l,m\}}D_s$ and let $E_l^m=\partial K_{l,m}\cap D_l$. We claim that $D_i\cap \partial K_{l,m}\subset E_l^m\cup E_m^l$ for each edge $[l,m]$ with two ends $l,m\in N(i)$. Otherwise, we may assume that $p\in D_i\cap \partial K_{l,m}\backslash (E_l^m\cup E_m^l)$. If $p\notin D_l\cup D_m$, by \eqref{containeq} we see that $p\in D_s$ for some $s\in N(i)\backslash\{l,m\}$, which contradicts the definition of $K_{l,m}$. Therefore, we have $p\in D_l\cup D_m$. Since $p\notin E_l^m\cup E_m^l$, we have $p\in D_l\cap D_m$. Hence, there is a small neighborhood $N_p$ of $p$ such that $N_p\subset D_i\cap D_m\cap D_l$. By the definition of $K_{m,l}$ there exists some $s\in N(i)\backslash\{l,m\}$ such that $N_p\cap D_s\neq\emptyset$. Therefore, $D_i\cap D_l\cap D_m\cap D_s\neq\emptyset$, which contradicts Lemma \ref{mostthree}. It follows that the above claim is valid.
		
		Without loss of generality, we may well assume that $D_i\cap D_j\neq\emptyset$ for
        some vertex $j\in N(i)$. Since $D_i$ is not contained in $D_j$, and $D_i=\cup_{t\in N(i)}D_i\cap D_t$, by the connectivity of $D_j$, there exists a vertex $s\in N(i)\backslash\{j\}$ such that $D_i\cap D_j\cap D_s\neq \emptyset$. Otherwise, $D_i$ can be divided into two closed sets
        $D_i\cap D_j$ and $D_i\cap( \cup_{t\in N(i)\backslash\{j\}}D_t )$.
        Therefore, $(D_i\cap D_j\cap D_s)\backslash (\cup_{h\in N(i)\backslash\{j,s\}}D_h)\neq\emptyset$.
		Next, we consider the following two cases:

        \medskip

        \emph{Case 1:} Both $D_i\cap E_s^j\neq\emptyset$ and $D_i\cap E_j^s\neq\emptyset$, then $C_i\cap E_s^j\subset C_j\neq\emptyset$ and $C_i\cap E_j^s\subset C_s\neq\emptyset$. Since three non-collinear points determine a circle, it is obvious that $D_i\subset D_j\cup D_s$ in this case.

       \medskip
       
        \emph{Case 2:} $D_i\cap E_s^j=\emptyset$ or $D_i\cap E_j^s=\emptyset$. Without loss of generality, assume $D_i\cap E_s^j= \emptyset$. Therefore, $D_i\cap E_j^s\neq\emptyset$. Since $D_i$ is not contained in $D_j$, we have $D_i\cap (C_j\backslash D_s)\neq\emptyset$. By \eqref{containeq}, there exists a vertex $k\in N(i)$ and a point $p\in  C_j\backslash D_s$ such that $p\in E_k^j\subset D_k$. Since $K_{j,k}$ and $K_{j,s}$ are disjoint and $D_i$ is connected, we have $D_i\cap E_j^k\neq\emptyset$. Therefore, for reasons similar to Case 1, we see that $D_i\subset D_j\cup D_k$.
	 
        In conclusion, we have shown that $D_i\subset D_{j_1}\cup D_{j_2}$ for some $j_1,j_2\in N(i)$ with $j_1\sim j_2$. According to Lemma \ref{containlemma}, we know that this contradicts condition ($Z_1$). 
		\end{proof}

	The following lemma introduces and provides a description of $V_{C(r)}(\mathcal{P})$, which will be used in the following sections. It states that $V_{C(r)}(\mathcal{P})$ is a set of vertices in those triangles or edges that are adjacent to a simple closed finite path in $\mathcal{T}$.

    \begin{lem}\label{sectionlemma} Let $\mathcal{T}=(V,E,F)$ be a disk triangulation, and let $\Theta\in [0 ,\pi)^E$ be an intersection angle that satisfies ($Z_1$) and ($Z_2$). Suppose $\pac=\{C_i\}_{i \in V}$ is an RCP that realizes the data $(\mathcal{T}, \Theta)$ and $\carrier(\pac)= \Omega$. Fix a vertex $v \in V$ and assume that the center $o_v$ of $C_v$ is located at the origin in $\mathbb{C}$. Assume that $D(r)\subset\Omega$ and that $C(r)$ is not contained in any disk
    of $\mathcal P$. Define
    \[
    V_{C(r)}(\mathcal P)=\{i\in V: D_i\cap C(r)\neq\emptyset\}.
    \]
    Then there exists a simple loop $\gamma \subset V_{C(r)}(\mathcal{P})$ such that $V_{C(r)}(\mathcal{P})$ consists of three types of vertices: (1) the vertices in $\gamma$; (2) the vertices of those edges that intersect $\gamma$; (3) the vertices of some triangle that has a boundary edge in $\gamma$. \end{lem}
    
    \begin{figure}[h] \centering \includegraphics[width=0.4\textwidth]{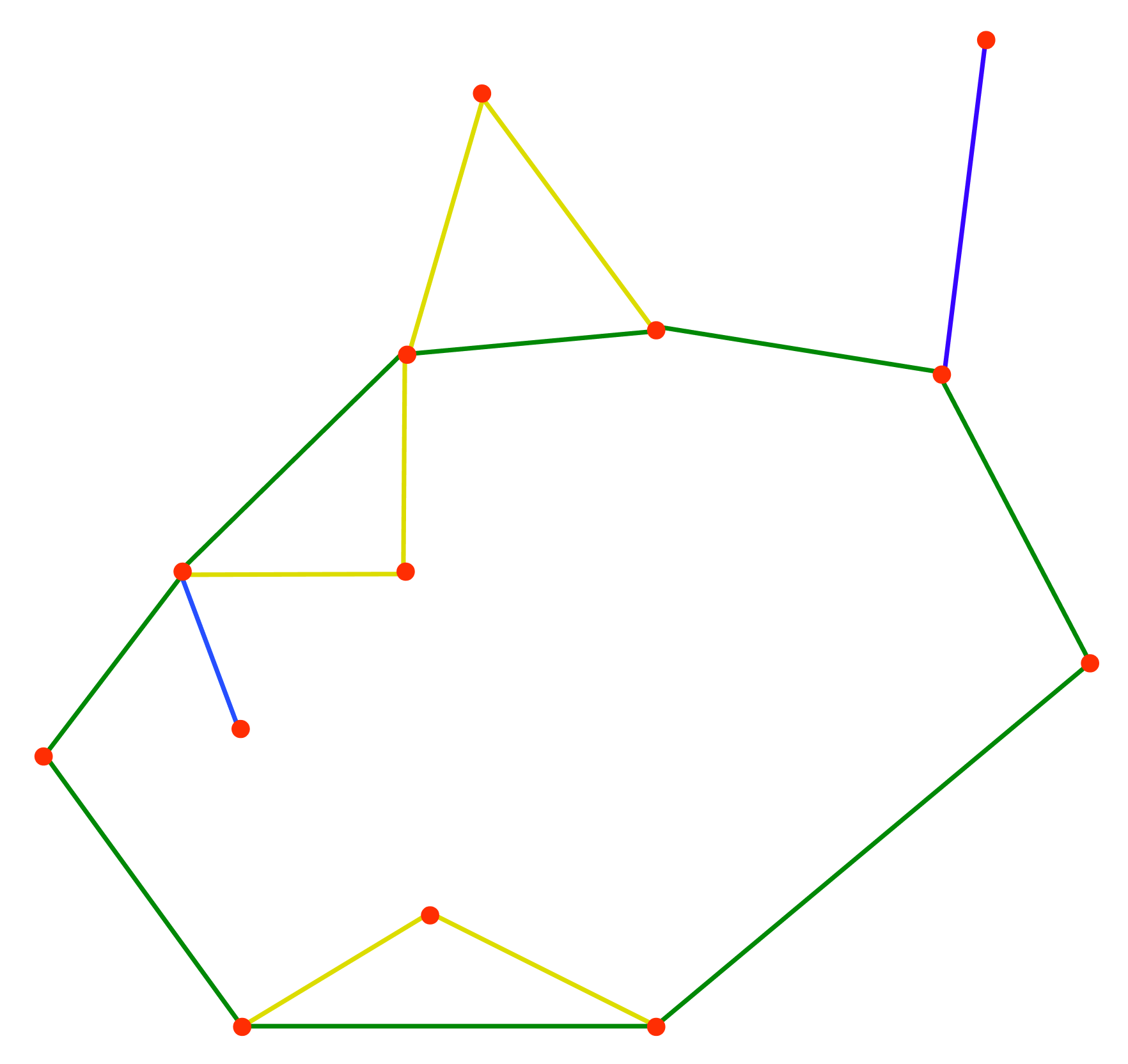} \caption{the structure of $\mathcal{T}_{C(r)}(\mathcal{P})$ associated with $V_{C(r)}(\mathcal{P})$} \label{Vcr_structure} 
    \end{figure}
    
    \begin{proof}
    By Lemma~\ref{mostthree} and Lemma~\ref{indispensable}, under the RCP assumption on 
    $\mathcal P$, the carrier $\carrier(\mathcal P)=\Omega$ admits a decomposition into the 
    following mutually disjoint types of regions:
    \begin{itemize}
    \item 
    \[
    A_i=D_i\setminus \bigcup_{j\neq i}D_j, \qquad i\in V;
    \]
    \item 
    \[
    B_{[i,j]}=(D_i\cap D_j)\setminus \bigcup_{k\neq i,j}D_k, 
    \qquad [i,j]\in E;
    \]
    \item 
    \[
    C_{[i,j,k]}
    =
    \bigl[\Delta_{ijk}\setminus (D_i\cup D_j\cup D_k)\bigr]
    \cup
    (D_i\cap D_j\cap D_k),
    \qquad [i,j,k]\in F.
    \]
\end{itemize}
It is clear that sets $\cup_{i\in V}A_i,\cup_{[i,j]\in E}B_{[i,j]},\cup  _{[i,j,k]\in F}C_{[i,j,k]}$ are pairwise disjoint, and together they decompose 
$\Omega$.



We now construct a closed edge path $\gamma$ in $\mathcal T$. 
Choose a point $p_0\in C(r)$, and let $p$ travel once around $C(r)$ 
in the clockwise direction. If $p_0$ lies in a disk $D_{v_0}$ such that
\[
C(r)\cap D_{v_0} \not\subset C(r)\cap D_v
\qquad 
\text{for every } v\in V \text{ with } v\neq v_0 \text{ and } p_0\in D_v,
\]
then we take $v_0$ to be the initial vertex of $\gamma$. Otherwise, we continue 
moving $p$ along $C(r)$ clockwise until it first enters a disk $D_{v_0}$ such that
\[
C(r)\cap D_{v_0} \not\subset C(r)\cap D_v
\qquad 
\text{for every } v\in V \text{ with } v\neq v_0 \text{ and } p\in D_v,
\]
and take this $v_0$ as the initial vertex. By Lemma~\ref{mostthree} every point is contained in at most three circles. Therefore such a choice 
of $v_0$ is always possible.

Suppose that the vertices $v_0,\ldots,v_n$ have already been chosen. Continue moving 
$p$ clockwise along $C(r)$ until it leaves $D_{v_n}$ and next enters a disk 
$D_{v_{n+1}}$. If the current $p \in D(v_{0})$, the construction stops. And we link the final vertex to $v_0$. Otherwise, 
we add $v_{n+1}$ to the sequence and continue the same procedure.

This produces a cyclic sequence of vertices
\[
v_0,v_1,\ldots,v_m,v_0.
\]
By the RCP condition, whenever the moving point passes from $D_{v_\ell}$ to 
$D_{v_{\ell+1}}$, the corresponding vertices are joined by an edge in $\mathcal T$. 
Thus
\[
[v_0,v_1],\ [v_1,v_2],\ \ldots,\ [v_m,v_0]
\]
are edges of $\mathcal{T}$. We define $\gamma$ to be the resulting 
closed edge path.

Moreover, for each disk $D_i$ intersecting $C(r)$, the set $C(r)\cap D_i$ is connected, 
since both $C(r)$ and $D_i$ are Euclidean circles/disks. Hence, as $p$ travels once 
around $C(r)$, it enters and exits $D_i$ at most once. Consequently, no vertex appears 
more than once in the above cyclic sequence, except for the initial vertex at the end. 
Therefore $\gamma$ is a simple closed edge path.

It remains to describe the vertices in $V_{C(r)}(\mathcal P)$. The construction of 
$\gamma$ divides $C(r)$ into subarcs, each of which is associated with one of the 
vertices appearing in $\gamma$. Since the regions
\[
A_i,\qquad B_{[i,j]},\qquad C_{[i,j,k]}
\]
form a decomposition of $\Omega$, every point of $C(r)\subset \Omega$ lies in exactly 
one of these regions. Therefore, if a disk $D_i$ meets $C(r)$, then the corresponding 
vertex $i$ must occur in one of the following ways:
\begin{enumerate}
    \item $i$ is a vertex of the simple loop $\gamma$;
    \item $i$ is an endpoint of an edge of $\mathcal T$ whose geometric realization 
    intersects $\gamma$;
    \item $i$ is a vertex of a triangle in $\mathcal T$ which has a boundary edge 
    contained in $\gamma$.
\end{enumerate}
Hence $V_{C(r)}(\mathcal P)$ consists precisely of vertices of the above three types. 
This proves the lemma.
\end{proof}

	\section{The existence of  CPs and RCPs}\label{section:existence}
	\subsection{Ring Lemma for embedded CPs}
    Given a circle pattern that weakly realizes the data$(\mathcal{T},\Theta)$, connecting the centers of circles in a face, we can obtain a triangle in Euclidean or hyperbolic background geometry. Gluing all those triangles together via the developing map (see \cite{Martelli} or \cite{Th76}), which is denoted by $\eta$, one can obtain a flat or hyperbolic surface. Even if all these triangles are embedded in $\mathbb{C}$ or $\mathbb{U}$, the surface obtained by the developing map $\eta$ may not be embedded in the plane. Therefore, we have the following natural definition.
        \begin{defn}
            Given a circle pattern $\pac$, it is called an \textbf{embedded circle pattern} in Euclidean background geometry (hyperbolic background geometry resp.) if and only if the corresponding developing map is an isometric embedding.
        \end{defn}
	\begin{rem}
    It is clear that \emph{all RCPs are embedded}.  
    Therefore, each theorem for embedded CPs also holds for RCPs. 
    \end{rem}
    
    To obtain existence results for CPs and RCPs, we need to establish a Ring Lemma for CPs that allows circles with obtuse intersection angles. 
    The Ring Lemma was first obtained by Rodin and Sullivan \cite{Rodin-Sullivan} for CPs with $\Theta = 0$. Subsequently, He \cite{He} extended the Ring lemma for CPs with $\Theta\in[0,\frac{\pi}{2}]^E$. However, the proof of the Ring Lemma for obtuse intersection angles is much more complicated, and the methods of He and Rodin-Sullivan will no longer be applicable in our setting. For any connected graph $G=(V, E)$, let $d\left(v_1, v_2\right)$ denote the combinatorial distance between $v_1$ and $v_2$, i.e., the length of the shortest path connecting $v_1$ and $v_2$ in $G$.
    For a vertex $i \in V$ and a nonnegative $\delta$, define
    $$
    B(i, \delta)=\{v \in V: d(v, i) \leq \delta\}, \quad S(i, \delta)=\{v \in V: d(v, i)=\delta\}.
    $$
    We have the following Ring Lemma:

	\begin{lem}\label{ringlemma}(Ring lemma)
		Let $\mathcal{T}= (V, E, F)$ be a \emph{finite} triangulation of the closed disk $\overline{\mathbb{U}}$, and let $\Theta\in[0 ,\pi)^E$ be an angle function. Assume conditions ($Z_1$) and ($Z_2$) hold, and there exists an embedded circle pattern $\pac = \{C_i\}_{i \in V}$ realizing $(\mT, \Theta)$ weakly. If there is a constant $\epsilon>0$, so that $\Theta([v_i,v_j])\leq\pi-\epsilon$ for all $[i,j]\in E$ (since $\mathcal{T}$ is finite, such $\epsilon$ can always be taken). Then, for each vertex $i\in V$ with $B(i, \frac{2\pi}{\epsilon})\cap \partial V= \emptyset$, there exists a constant \[
C=C\bigl(\epsilon,\Theta|_{B(i,\lceil 2\pi/\epsilon\rceil)}, 
\mathcal T|_{B(i,\lceil 2\pi/\epsilon\rceil)}\bigr)>0
\] such that
		\begin{equation}
			\frac{r_j}{r_i}\geq C, \quad \forall j \sim i.
		\end{equation}
	\end{lem}
	\begin{proof}
        We prove it by contradiction. Choose a vertex $i \in V$ such that $B(i, \frac{2\pi}{\epsilon})\cap \partial V= \emptyset$ and a vertex $j \in V$ such that $j \sim i$. If this lemma is not true, then there is a sequence of finite circle patterns $\{\pac^{[n]}\}$, which weakly realize $(\mT, \Theta)$, such that $r_i^{[n]}=1$ and $r_j^{[n]}\rightarrow 0$ as $n \rightarrow + \infty$. By subtracting a subsequence, we may assume that, for any $v \in V$, the limit $r_v^{[\infty]}= \lim\limits_{n\rightarrow + \infty}r_v^{[n]}$ exists in $[0, + \infty]$. It is obvious that $r_i^{[\infty]}= 1$ and $r_j^{[\infty]}= 0$.
		
		Let $A = \{v\in V : r_v^{[\infty]}= 0\}$, and let $A_0$ be the connected component of $A$ containing $j$. It is obvious that $i \in \tilde \partial A_0$. Let $V_\gamma$ be the connected component of $\partial A_0$ that contains $i$. Denote 
		$$\gamma=V_\gamma \cup \{e \in E: \partial e \subset V_\gamma\}.$$ 
		
		The following proof is divided into two steps.
		
		\medskip
		\emph{Step 1. We claim that $\gamma$ contains a closed curve $\gamma_0$.}
		\medskip
		
		If the claim is not true, then $\gamma$ is a path. For convenience, if $\gamma$ contains only one vertex, we still call $\gamma$ a path. Let $\alpha, \beta \in V$ be the end vertices of $\gamma$ (if $\gamma=\{pt\}$, then $\alpha$ and $\beta$ are the same). We derive contradictions by considering whether $\alpha$ and $\beta$ are interior points of $V$ separately.

		\medskip
		\emph{Case 1: $\alpha$ or $\beta$ is an interior vertex of $V$}. 
		\medskip

        Assume that $\alpha$ is an interior point of $V$. The case where $\beta$ is an interior point is similar. In this case, there are at least two vertices $v^\prime, v^{\prime \prime} \sim \alpha$ such that $r_{v^\prime}^{[\infty]}>0$ and  $r_{v^{\prime \prime}}^{[\infty]}>0$. Consider the set neighboring vertices $S(\alpha, 1)$ of $\alpha$. Denote $\tilde N = \tilde \partial A_0 \cap S(\alpha, 1)$. Since the set $S(\alpha, 1)$ is circle-like, $\tilde N$ contains at least two vertices. For any $v \in \tilde N$, it follows that $r_v^{[\infty]}>0$, otherwise $v \in A_0$ by the definition of $A_0$, which is a contradiction. By the definition of $\gamma$, we know that $\tilde N \subset \gamma\cap S(\alpha, 1)$. This contradicts the fact that there is at most one vertex in $\gamma\cap S(\alpha, 1)$, since $\alpha$ is an end vertex of $\gamma$. 

		\medskip
		\emph{Case 2: Both $\alpha$ and $\beta$ are boundary points of $V$.} 
		\medskip
        
		  Write $\gamma= (e_1, e_2,\cdots, e_l)$. Let $(\alpha_0= \alpha, \alpha_1,, \alpha_2 \cdots, \alpha_n= \beta)$ be the corresponding vertices of $\gamma $. By the definition of $A_0$, the radius $r_v^{[\infty]}=0$ for every $v \in A_0$. It follows that $D^{[\infty]}_{\alpha_k}$ ($k = 0, 1, \cdots, l$), which are the disks in the limit circle pattern $\pac^{[\infty]} = \lim\limits_{n\rightarrow + \infty}\pac^{[n]}$, intersect at a common point. Since the nondegenerate triangles of $\pac^{[\infty]}$ is also embedded, we have 
		$$
		\sum_{k=1}^l (\pi- \Theta(e_k))\leq 2\pi.
		$$
		By $\Theta\in[0,\pi-\epsilon]^E$, it follows that  $l \leq \frac{2\pi}{\epsilon}$, which contradicts $B(i, \frac{2\pi}{\epsilon})\cap \partial V= \emptyset$.
		
		\medskip
		\emph{Step 2. We finish the proof.}
		\medskip
		
		From the discussion in Step 1, we actually gather that there are no endpoints in $\gamma$ and $\gamma \cap \partial V = \emptyset$.  In this assertion, a vertex in $\gamma$ is called an endpoint if there is at most one neighboring vertex in $\gamma$. 
		
		Since $\gamma$ contains a closed curve, there are two cases:
		
		\medskip
		\emph{Case 1: Every simple loop in $\gamma$ is the boundary of a face}. 
		\medskip
		
		Since there are no end vertices in $\gamma$, there must exist an end triangle $\Delta v_1v_2v_3$ in $\gamma$. In this assertion, a face $f$ is called an end triangle if all its vertices are in $\gamma$ and there is at most one vertex $v \in \gamma \cap \tilde \partial f$. Now, we just need to deal with the end triangle just like the end vertex. For convenience, we adopt a notation that slightly abuses the symbols but causes no ambiguity. Since $v_1,v_2,v_3$ are all interior vertices, it is easy to see that there are at least two vertices $v^\prime, v^{\prime \prime} \in \tilde \partial \Delta_{v_1v_2v_3} $ such that $r_{v^\prime}^{[\infty]}>0, r_{v^{\prime \prime}}^{[\infty]}>0$.  Otherwise, $\sum_{k=1}^3 \sigma(v_k) < 6\pi$.  Let $\tilde N = \tilde \partial A_0 \cap \tilde \partial \Delta v_1v_2v_3$. Since $\tilde \partial \Delta v_1v_2v_3$ is circle-like, $\tilde N$ contains at least two vertices. For any $v \in \tilde N$, it follows that $r_v^{[\infty]}>0$. Otherwise $v \in A_0$ by the definition of $A_0$, which is a contradiction. By the definition of $\gamma$, we know $\tilde N \subset \gamma\cap S(\alpha, 1)$. It contradicts  that $\Delta v_1v_2v_3$ is an end triangle in $\gamma$.
		Therefore, there is no possible that $\gamma$ only contain the simple closed curves that are the boundary of a face.
		
		\medskip
		\emph{Case 2: $\gamma$ contains a simple loop $\gamma_0$ that is not a boundary of a face}.
		\medskip
		
		Note the radius $r_v^{[\infty]}=0$ for any $v \in A_0$, it follows that $D^{[\infty]}_{\alpha_k}$ ($k = 0, 1, \cdots, l$), which are the disks in the limit circle pattern $\pac^{[\infty]} = \lim\limits_{n\rightarrow + \infty}\pac^{[n]}$, intersect at a common point. Hence we know
		$$
		\sum_{e \in \gamma_0}(\pi- \Theta(e))=2\pi,
		$$
		which contradicts  ($Z_2$).
		
	\end{proof}

	\subsection{The existence of CP  for infinite triangulation}
	We adopt the approach of He \cite{He} and Ge-Yu-Zhou \cite{GYZ} to prove the existence of circle patterns.
	To prove the existence of CPs realizing infinite triangulations, we first provide an existence result of CPs realizing finite triangulations, with the help of exhaustion and approximation methods. Then we obtain the existence of infinite CPs.
	\begin{thm}\label{exist_boundary_radius}
		Let $\mathcal{T}= (V, E, F)$ be a \emph{finite} triangulation of $\overline{\mathbb{U}}$, and let $\Theta\in[0 ,\pi)^E$ be an angle function. Assume conditions ($Z_2$) and ($Z_4$) hold. 
		Then there exists a unique embedded circle pattern $\mathcal{P}=\{C_i\}_{i\in V}$ weakly realizing $(\mathcal{T},\Theta)$, and moreover, every boundary circle $C_i$ (for $i\in\partial V$) is tangent to $\partial \mathbb{U}$.  
	\end{thm}
	\begin{proof}
		The proof is almost identical to the argument of \cite[Lemma 6.2]{GYZ}, which primarily relies on Lemma \ref{vari}. For the completeness of the paper, we present a sketch of the proof.
        First, using the variational principle Lemma $\ref{vari}$, there exists a hyperbolic CP weakly realizing $(\mathcal{T},\Theta)$ whose boundary circles are horocycles. This is just a Dirichlet boundary problem when assuming the radii of circles on boundary vertices are all equal to $+\infty$ in hyperbolic background geometry. Therefore, we obtain a circle pattern with each boundary circle internally tangent to $\partial\mathbb{U}$. It is also possible to use Chow-Luo's combinatorial Ricci flow method to prove it; for specific ideas, we refer to \cite{CL03, GHZ-advance, GHZ21}. We omit the details here.

        Secondly, we show that this CP is actually an embedded CP in hyperbolic space. The method is similar to that of Stephenson \cite[Chapter 6]{st}. Notice that the developing map we obtained is an immersion from the closed topological disk to $\overline{\mathbb{U}}$.
        Furthermore, by pushing the boundary of $\eta(\mathbb{U})$ to $\partial\mathbb{U}$, the developing map can be modified to $\tilde\eta$, which is an immersion that maps from $\partial\overline{\mathbb{U}}$ to $\partial\overline{\mathbb{U}}$. Therefore, by the domain invariance theorem, it is an embedding. Since the modification only changes the image of the boundary of $\eta(\mathbb{U})$, we can prove that $\eta$ is also an embedding.
	\end{proof}
	
	Our proof of Theorem~\ref{exist_boundary_radius} requires the stronger hypothesis ($Z_4$) and cannot be 
    weakened to the assumption ($Z_1$). The topological degree method used in Zhou \cite{Zhou21} is also inapplicable here, since a key step in Zhou's approach is valid only for good circle patterns. Nevertheless, we conjecture that the above theorem remains true under the weaker condition ($Z_1$) (and ($Z_2$)).

	\begin{thm}[Existence of CPs]\label{exist_cp}
		Let $\mathcal{T}= (V, E, F)$  be a disk triangulation graph with an angle function $\Theta\in[0,\pi)$ and $\sup_{e\in E}\Theta(e)<\pi$. Suppose that conditions ($Z_2$) and ($Z_4$) hold, then there exists an embedded circle pattern $\mathcal{P}$ weakly realizing the data ($\mathcal{T}$, $\Theta$). 
	\end{thm}
	\begin{proof}
		Pick an arbitrary vertex $v_0\in V$, recall that $d(v, v_0)$ is the combinatorial distance in $\mT$. Since $\mathcal T=(V,E,F)$ is a locally finite triangulation of the open disk,
        we choose an increasing exhaustion by finite triangulated topological disks
        \[
        \left\{\mathcal T^{[n]}=\left(V^{[n]},E^{[n]},F^{[n]}\right)\right\}_{n=1}^{\infty}
        \]
        such that each $|\mathcal T^{[n]}|$ is homeomorphic to a closed disk,
        \[
        |\mathcal T^{[n]}|\subset \operatorname{int} |\mathcal T^{[n+1]}|,
        \qquad
        \bigcup_{n=1}^{\infty}|\mathcal T^{[n]}|=|\mathcal T|,
        \]
        and every compact subset of $|\mathcal T|$ is contained in $|\mathcal T^{[n]}|$
        for all sufficiently large $n$.
        Here $|\mathcal T^{[n]}|$ denotes the underlying polyhedron of the finite
        subcomplex $\mathcal T^{[n]}$.
        
		By Theorem \ref{exist_boundary_radius}, we know that there exist circle patterns $\pac^{[n]} = \{C^{[n]}_i\}_{i \in V^{[n]}}$, which weakly realize $(\mT^{[n]}, \Theta)$ with boundary circles tangent to the unit circle. By M\"obius transformation, we can obtain a circle pattern that weakly realizes $(\mT^{[n]}, \Theta)$ with $r_{v_0}^{[n]}= 1$.
		
		For any $i, j \in V$ such that $i \sim j$, by Lemma \ref{ringlemma}, we know there is a constant $C (i,j)>0$, independent of $n$,  such that 
		$$
		C(i,j)^{-1}\leq \frac{r_i^{[n]}}{r_j^{[n]}} \leq C(i,j)=C(\mathcal{T}^{[n]},\Theta),
		$$
		for some sufficiently large $n$. Since $r_{v_0}^{[n]}= 1$, and $\mathcal{T}$ is connected, then for any $i \in V$, there is a constant $\tilde C (i)>0$ such that 
		$$
		\tilde C(i)^{-1}\leq r_i^{[n]} \leq \tilde C(i),
		$$
		for any sufficiently large $n$. Hence, for any $i \in V$, the center of $C^{[n]}_i$ is also located in a compact set, for any sufficiently large $n$. 
		
		Thus, by a diagonal argument, there exists a subsequence such that the radii and centers of all vertices converge. The limit circle pattern  $\pac^{[\infty]} = \{C^{[\infty]}_i\}_{i \in V}$ of this subsequence is exactly the CP we need.
	\end{proof}

	\subsection{The existence of RCP for infinite triangulation}
	\begin{thm}(Zhou \cite[Theorem 1.4]{Zhou21})\label{RCP_finite}
		Let $\mathcal{T}$ be a \emph{finite} triangulation of the sphere with more than four vertices. Assume that $\Theta: E \rightarrow[0, \pi)$ is a function satisfying the following conditions:
		\begin{itemize}
			\item[(1)] If $e_1, e_2, e_3$ forms the boundary of a triangle of $\mathcal{T}$, then $\Theta\left(e_1\right)+\Theta\left(e_2\right)<\Theta\left(e_3\right)+\pi, \Theta\left(e_2\right)+$ $\Theta\left(e_3\right)<\Theta\left(e_1\right)+\pi, \Theta\left(e_3\right)+\Theta\left(e_1\right)<\Theta\left(e_2\right)+\pi$.
			\item[(2)] If $e_1, e_2$ forms a homologically non-adjacent arc, then $\Theta\left(e_1\right)+\Theta\left(e_2\right) \leq \pi$, and one of the inequalities is strict when $\mathcal{T}$ is the boundary of a triangular bipyramid.
			\item[(3)] If $e_1, e_2, e_3$ forms a simple loop that separates the vertices of $\mathcal{T}$, then $\sum_{i=1}^3 \Theta\left(e_i\right)<\pi$.
			\item[(4)] If $e_1, \cdots, e_4$ forms a simple loop that separates the vertices of $\mathcal{T}$, then $\sum_{i=1}^4 \Theta\left(e_i\right)<2 \pi$.
		\end{itemize}
		Then there exists a regular circle pattern $\mathcal{P}$ on the Riemann sphere $\hat{\mathbb{C}}$, whose exterior intersection angles are given by $\Theta$.
	\end{thm}
	
	\begin{thm}[Existence of RCPs]\label{RCPe}
		Let $\mathcal{T}= (V, E, F)$  be a disk triangulation with $\Theta\in[0,\pi)^E$ and $\sup_{e\in E}\Theta(e)<\pi$. Suppose that conditions ($Z_1$), ($Z_2$) and ($Z_3$) hold, then there exists a regular circle pattern $\mathcal{P}$ realizing the data ($\mathcal{T}$, $\Theta$). 
	\end{thm}
	\begin{proof}
		Take an arbitrary vertex $v_0\in V$, recall that $d(v, v_0)$ is the combinatorial distance in $\mT$. Since $\mathcal T=(V,E,F)$ is a locally finite triangulation of the open disk,
we choose an increasing exhaustion by finite triangulated topological disks
\[
\left\{\mathcal T^{[n]}=\left(V^{[n]},E^{[n]},F^{[n]}\right)\right\}_{n=1}^{\infty}
\]
such that each $|\mathcal T^{[n]}|$ is homeomorphic to a closed disk,
\[
|\mathcal T^{[n]}|\subset \operatorname{int} |\mathcal T^{[n+1]}|,
\qquad
\bigcup_{n=1}^{\infty}|\mathcal T^{[n]}|=|\mathcal T|,
\]
and every compact subset of $|\mathcal T|$ is contained in $|\mathcal T^{[n]}|$
for all sufficiently large $n$.
Here $|\mathcal T^{[n]}|$ denotes the underlying polyhedron of the finite
subcomplex $\mathcal T^{[n]}$.
		
		For each $\mathcal{T}^{[n]}$, we may add a point $v_\infty$ outside its carrier. For any boundary vertex $v\in \partial V^{[n]}$, connect it with $v_\infty$  by an edge $e=[v, v_\infty]$. Such construction gives a triangulation  $ \mathcal{T}_s^{[n]}$, which is a finite triangulation of the sphere. Set $\Theta_s^{[n]}([v, v_\infty])=0$ for any boundary vertex $v\in \partial V^{[n]}$, and set $\Theta_s^{[n]}(e)=\Theta(e)$ for any $e\in E^{[n]}$. 
		
		  By the definition of $V^{[n]}$ and using the conditions ($Z_1$), ($Z_2$), 
        and ($Z_3$), it is easy to check that the intersection angle $\Theta_s(e)$ in $ \mathcal{T}_s^{[n]}$ satisfies all conditions (1)-(4) in Theorem \ref{RCP_finite}. Therefore, by Theorem \ref{RCP_finite}, there exists an RCP $\mathcal{P}_s^{[n]}$ on the Riemann sphere $\hat{\mathbb{C}}$, whose exterior intersection angles are given by $\Theta_s^{[n]}$. By stereographic projection, if we fix $r_{v_0}=1$, we know that there exist RCPs $\pac^{[n]} = \{C^{[n]}_i\}_{i \in V}$ that realize $(\mT^{[n]}, \Theta)$ and $r_{v_0}^{[n]}= 1$.
		
		By a similar argument as in the proof of Theorem \ref{exist_cp}, we know that there exists a subsequence such that the radii and centers of all vertices converge. Without loss of generality, we may assume $\{\pac^{[n]} \}_{n=1}^{\infty}$ is exactly the convergent subsequence. Denote $\pac^{[\infty]} = \{D^{[\infty]}_i\}_{i \in V}$ by the limit circle pattern of $\{\pac^{[n]} \}_{n=1}^{\infty}$. To prove $\pac^{[\infty]}$ is an RCP, we only need to check that $D^{[\infty]}_u \cap D^{[\infty]}_v = \emptyset$ for any $[u,v] \notin E$. We prove this fact by contradiction. Assume that there are two vertices $u, v \in V$ such that $[u,v] \notin E$ and $D^{[\infty]}_u \cap D^{[\infty]}_v \neq \emptyset$. There are two cases to consider:

        \medskip
        \emph{Case 1: $D^{[\infty]}_u \cap D^{[\infty]}_v$  contains an interior point $p$.}
        \medskip
        
        In this case, we have $d_{\mathbb R^2}(p, c^{[\infty]}_u)< r^{[\infty]}_u$ and $d_{\mathbb R^2}(p, c^{[\infty]}_v)< r^{[\infty]}_v$, where $d_{\mathbb R^2}(x,y)$ is the Euclidean distance between $x$ and $y$,  and $c^{[\infty]}_u, c^{[\infty]}_v$, $r^{[\infty]}_u, r^{[\infty]}_v$ are the centers and radii of $C^{[\infty]}_u , C^{[\infty]}_v$ respectively. Thus, from the properties of $d_{\mathbb R^2}(x,y)$, we know $d_{\mathbb R^2}(p, c^{[n]}_u)< r^{[n]}_u$ and $d_{\mathbb R^2}(p, c^{[n]}_v)< r^{[n]}_v$, for any sufficiently large $n$. Hence, $p\in D^{[n]}_u \cap D^{[n]}_v$ holds for any sufficiently large $n$, which leads to a contradiction.
		
        \medskip
        \emph{Case 2: $D^{[\infty]}_u \cap D^{[\infty]}_v$  contains a single point $p$.} 
        \medskip
        
		Due to Lemma \ref{onepointlemma}, we know that there is a vertex $w \in V$ so that $p \in D^{[\infty]}_w$. Then it follows that both $D^{[\infty]}_u \cap D^{[\infty]}_w$ and $D^{[\infty]}_v \cap D^{[\infty]}_w$ contain interior points. For the same reasons as previously argued, $D^{[n]}_u \cap D^{[n]}_w$ and $D^{[n]}_v \cap D^{[n]}_w$ contain interior points for any sufficiently large $n$. Since $\pac^{[n]}$ are RCPs, $u$ and $v$ are connected to $w$. 
        Since $[u,v]\notin E$, $[u,w]$ and $[v,w]$ form a homologically non-adjacent arc.
        According to Lemma \ref{containlemma}, we see
		$$\Theta (u,w)+\Theta (v,w)\ge\pi.$$ 
		It contradicts the condition ($Z_3$).
		
	\end{proof}

	\section{Uniformization of RCPs}\label{sec:uniform}
	\subsection{Vertex extremal length and recurrence of networks}
	
	  To deal with infinite CPs, we will make extensive use of \textbf{vertex extremal length} theory on graphs. Specifically, to obtain the rigidity of CPs, we require some theory that comes from \textbf{infinite networks}. In this section, we will provide a brief review of the relevant background \cite{ahlfors, Ca94, He}.

	Given an undirected infinite graph $G=(V, E)$, where $V$ and $E$ are the vertex and edge sets of $G$ respectively, we write $v_i \sim v_j$ if $v_i$ and $v_j$ are connected by an edge in $E$. For a function $f: V \rightarrow \mathbb{R}$, the discrete Laplacian operator $\Delta_G$ is given by
	$$
	\Delta_G f_i=\sum_{j \sim i} \omega_{i j}\left(f_j-f_i\right), ~i\in V,
	$$
	where $\omega_{ij}$ is the weight on the edge $E$. For an edge $e=[i,j]\in E$, we also write $\omega_{ij}$ as $\omega_{e}$ or $\omega_{[i,j]}$ without ambiguity.
	
	\begin{defn}
		Let $i\in V$ be a vertex. A function $f: V\rightarrow \mathbb{R}$ is called \textbf{harmonic} (\textbf{subharmonic} resp.) at $i$ if $\Delta_G f_i=0$ ($\ge0$ resp.). Given a subset $W\subset V$, $f$ is called harmonic (subharmonic resp.) on $W$ if $\Delta_G f_i=0$ ($\geq 0$ resp.) for any $i\in W$.
	\end{defn}
	For subharmonic functions on finite graphs, the following maximum principle holds.
	\begin{lem}(\cite[Proposition 1.4.1]{anandam2011harmonic})\label{maxharmonic}
		Let $G=(V,E)$ be a finite connected graph. Suppose that $f$ is a subharmonic function on $\mathrm{int}(V)$ and attains its maximum at a vertex in $\mathrm{int}(V)$, then $f$ is constant.
	\end{lem}
	We call a function $\nu: V\rightarrow [0,+\infty)$ a vertex metric. For a vertex metric $\nu$, its area is defined by
	$$
	\area (\nu) = \sum_{v\in V}  \nu_v^2.
	$$
	For a subset $\alpha \subset V$, its length with respect to the vertex metric $\nu$ is defined by 
	$$
	l(\alpha) = \int_\alpha d\nu=  \sum_{v\in \alpha} \nu_v.
	$$
	Let $A$ be a non-empty set of subsets of $V$. A vertex metric $\nu$ is called $A$-admissible, if $\int_\alpha d\nu \geq 1$ for any $\alpha \in A$. 
	\begin{defn}
		The \textbf{vertex modulus} of $A$ is defined by 
		$$
		\modc (A) = \inf \{\area (\nu)| \nu \text{ is }  A\text{-admissible}  \}.
		$$
	\end{defn}
	
	\begin{defn}
		The \textbf{vertex extremal length} of $A$ is defined by 
		$$
		\vel (A)= \frac{1}{ \modc (A)}.
		$$
	\end{defn}

	We call a function $\eta: E\rightarrow [0,+\infty)$ an edge metric. For an edge metric $\eta$, its area is defined as the following:
	$$
	\area (\eta) = \sum_{e\in E} \omega_e \eta_e^2.
	$$
	For a path $\gamma=(v_0,v_1,\cdots,v_n,\cdots)$, the length of $\gamma$ in the edge metric $\mu$ is given by 
	$$
	\int_\gamma d\eta= \sum_{i \geq 0}\eta_{[v_i, v_{i+1}]}.
	$$
	Let $\Gamma$ be a set of paths. An edge metric $\eta$ is called \textbf{$\Gamma$-admissible}, if $\int_\gamma d\eta \geq 1$ for any $\gamma \in \Gamma$. 
	\begin{defn}
		The \textbf{conductance} of $\Gamma$ is defined by
		$$
		\cond (\Gamma) = \inf \{\area_\omega (\eta)| \eta \text{ is }  \Gamma\text{-admissible}  \}.
		$$
	\end{defn}
	\begin{defn}
		The  \textbf{resistance} of $\Gamma$ is defined by
		$$
		\res (\Gamma)= \frac{1}{ \cond (\Gamma)}.
		$$
	\end{defn}

	We denote $\Gamma\left(V_1, V_2\right)$ be the set of paths joining $V_1$ and $V_2$ in $G$. Since we can see the paths as a subset of $V$, for any two nonempty sets $V_1, V_2 \subset V$, the vertex extremal length between $V_1$ and $V_2$ is defined by
	$$
	\vel (V_1, V_2) = \vel (\Gamma(V_1, V_2)) .
	$$
	For a pair of mutually disjoint, nonempty sets $V_1, V_2 \subset V$, the resistance between $V_1$ and $V_2$ is defined by
	$$
	\res (V_1, V_2) = \res (\Gamma(V_1, V_2)) .
	$$
For vertex extremal length and resistance, we have the following properties.
	\begin{prop}\cite[Lemma 5.1]{He}\label{addlemma}
		Let $V_1, V_2, \ldots, V_{2 m}$ be mutually disjoint, nonvoid subsets of vertices such that for $i_1<i_2<i_3, V_{i_2}$ separates $V_{i_1}$ from $V_{i_3}$. We allow $V_{2 m}=\infty$. Then 
		$$
		\operatorname{VEL}\left(V_1, V_{2 m}\right) \geq \sum_{i=1}^m \operatorname{VEL}\left(V_{2 k-1}, V_{2 k}\right).
		$$
	\end{prop}

    \begin{prop}\label{addlemma2}
        	Let $V_1, V_2, \ldots, V_{2 m}$ be mutually disjoint, non-empty subsets of vertices such that for any $i_1<i_2<i_3, V_{i_2}$ separates $V_{i_1}$ from $V_{i_3}$. We allow $V_{2 m}=\infty$. Then 
		$$
		\operatorname{RES}\left(V_1, V_{2 m}\right) \geq \sum_{i=1}^m \operatorname{RES}\left(V_{2 k-1}, V_{2 k}\right).
		$$
    \end{prop}
	By way of contrast, the continuous counterparts of the above propositions are also valid and can be found in Ahlfors's well-known book \cite{ahlfors}.
    
	\begin{defn}
		A connected graph $G$ is called \textbf{VEL-parabolic (or VEL-hyperbolic, resp.)} if $\vel (V_1, \infty) =+ \infty$ (or $\vel (V_1, \infty) <+ \infty$, resp.) for any  finite nonempty set $V_1$.
	\end{defn}
	
	\begin{defn}
		A connected weighted graph $(G, \omega)$ is called \textbf{recurrent (or transient, resp.)} if $\res (V_1, \infty) =+ \infty$ (or $\res (V_1, \infty) <+ \infty$, resp.) for any  finite nonempty set $V_1$.
	\end{defn}
	
	It is worth noting that the ``any" in the above two definitions can be replaced with ``some". Moreover, there is a relationship between $\vel (V_1, V_2) $ and $\res (V_1, V_2)$.

	In order to prove the rigidity of RCPs, we need the following Lemma.
	\begin{lem}\cite[Lemma 5.5]{He}\label{harmonic_current}
		For a recurrent connected weighted graph $(G, \omega)$, if $f: V \rightarrow \mathbb{R}$ is a bounded harmonic function, then $f$ is a constant function.
	\end{lem}

		\begin{figure}[h]
			\centering
			\includegraphics[width=0.5\textwidth]{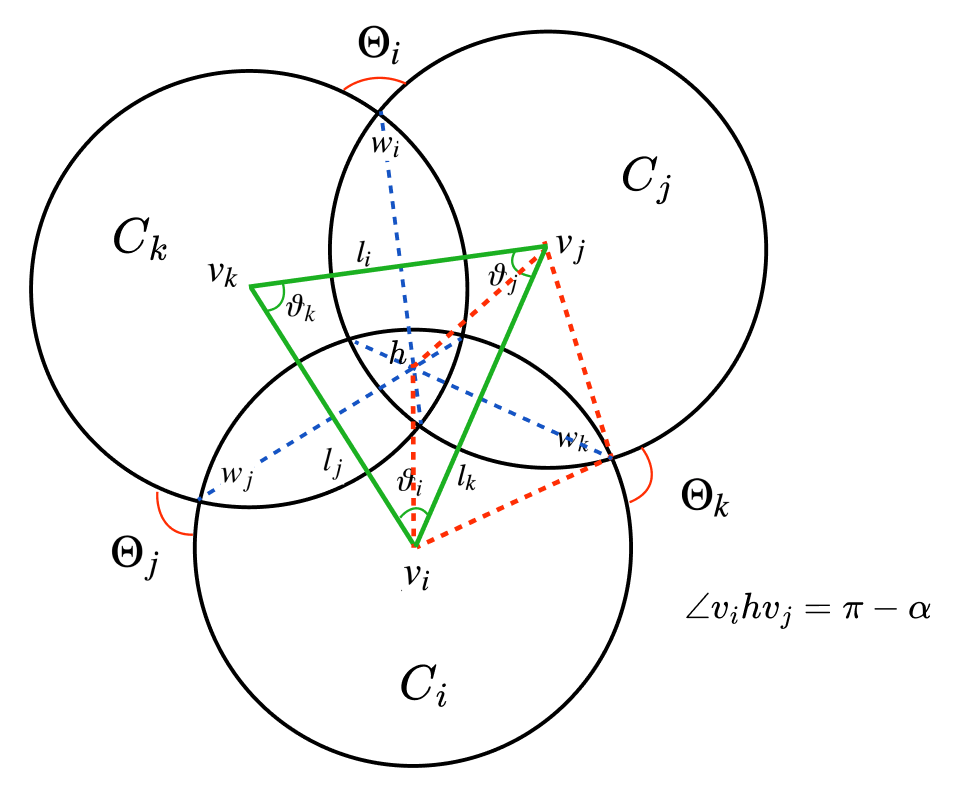}
			\caption{the geometric meaning of $r_j\frac{\partial \vartheta_i}{\partial r_j}$.}
			\label{partialconstrlemma1}
		\end{figure}	
		Now we introduce the weight associated to a circle pattern.
		Given a circle pattern of a triangulation on the plane, let $\{\Delta_{v_iv_jv_k}\}$ be a triangular face.
        Let $L_i,L_j$ and $L_k$ be straight lines passing through the intersection points of the pairs $\{C_j,C_k\},\{C_i,C_k\}$ and $\{C_i,C_j\}$ respectively. By \cite[Lemma A-1]{CL03}, as shown in Figure \ref{partialconstrlemma1}, we know $w_i,w_j$ and $w_k$ intersect at a common point $h$. By $h_i,h_j$ and $h_k$ we denote the distance from $h$ to the straight lines $v_jv_k$, $v_iv_k$ and $v_iv_j$. Let $H_k$ be $d_k$ if $h$ is inside the triangle and $-d_k$ otherwise. Then we have  $$r_j\frac{\partial \vartheta_i}{\partial r_j}=\frac{H_k}{l_k}.$$  Therefore, by Lemma \ref{vari},  if the angles $\Theta_i,\Theta_j$ and $\Theta_k$ satisfy $\ref{cos-condition}$, $h$ is contained in the triangle $\Delta_{ijk}$. Now for an edge $[i,j]$ of the face, we denote by $h_{ij}^k$ the quantity $\frac{H_k}{l_k}$. Then one can define weight on edges
        \[
        \omega_{ij}=\sum_{k:[i,j,k]~\text{is a triangle}}h_{ij}^k.
        \]
There is a geometric interpretation of the weight $\omega$, see \cite{Glickenstein15}. For any three circles in general position, there exists a circle that is perpendicular to each of them. For a given circle pattern of a triangulation $\mathcal{T}$ and each face $f=[i,j,k]$ of $\mathcal{T}$, we denote by $h_{ijk}$ the center of the circle that is perpendicular to $C_i, C_j$ and $C_k$. We call $h_{ijk}$ the center of the triangle $\Delta_{ijk}$. Connecting the centers of the triangles, we obtain a dual graph of the triangulation. The dual edge $e^*=[i,j]^*$ is obtained by connecting two centers of faces that are adjacent to $e=[i,j]$. 

\textbf{For an edge $e$ of the triangulation, we denote by $Q_e$ the quadrilateral whose diagonals are $e$ and $e^*$, as shown in Figure \ref{partialconstrlemma1}.} This quadrilateral is obtained by connecting the endpoints of edge $e$ and its dual, which is not necessarily convex.
Moreover, the edge and its dual are perpendicular. The weight is the ratio of the dual edge length to the edge length. Now we have the following lemma. Noticing that a similar lemma appeared in \cite[Lemma 4.4]{Ulrike08}.
\begin{lem}\label{recurrent_lemma}
    Let $\mathcal{T}=(V,E,F)$ be a disk triangulation, $\Theta\in [0,\pi)^E$ be an angle function. Assume that $(Z_3)$ and $(Z_4)$ hold. Let $\mathcal{P}$ be a parabolic RCP, and $\omega$ be the weight associated with $\mathcal{P}$ defined above, then $(\mathcal{T},\omega)$ is recurrence.
\end{lem}
\begin{proof}
    Since $\mathcal{P}$ is RCP-parabolic, it is locally finite. Let $v_0\in V$, be a fixed vertex. Without loss of generality, we assume that $C_{v_0}$ is the unit circle.  We denote by $V_0$ the set of vertices whose circles intersect with $C_{v_0}$. Let $R$ be a number large enough such that for all $v\in V_0$, $C_{v_0}\subset D(0,R)$. Now let $\eta$ be a function on edges defined as
    \[
    \eta(e)=\begin{cases}
    |e\cap D(0,2R)|& e~~\text{has a vertex inside }D(0,2R),
    \\
    0& \text{otherwise.}
    \end{cases}
    \]
    Here $e\cap D(0,2R)$ is the length of the segment $|e\cap D(0,2R)|$.
    Since $\pac$ is locally finite, we can find a finite vertex set $V_1$ that separates $V_0$ and $\infty$. Then clearly, $\frac{\eta}{R}$ is $\Gamma(V_0,V_1)$ admissible. Now assume that an edge $e$ has a vertex $v$ inside $D(0,2R)$. If that vertex belongs to $V_0$, by the definition of $D(0,R)$, the circle is contained in $D(0,R)$, and its radius is smaller than $R$. If $v$ does not belong to $D(0,R)$, by the definition of $V_0$,  $C_v\cap C_{v_0}=\emptyset$. Clearly, we have $r_{v}\le R$. In both two cases, we see that $C_{v}\subset D(0,3R)$. Therefore, the quadrilateral $\frac{|e\cap D(0,2R)|}{|e|}Q_e$ is contained in $B(0,3R)$. Here, $\frac{|e\cap D(0,2R)|}{|e|}Q_e$ means a quadrilateral inside $Q_e$ that shares the vertex $v$ with $Q_e$ and is scaled by a factor of $\frac{|e\cap D(0,2R)|}{|e|}$. Now, computing the conductance, we have
    \[
    \mathrm{area(\eta)}\le\sum_{e\in E}\omega(e)\eta^2(e)=\frac{1}{R^2}\sum_{e:e~\text{has a vertex in}~ D(0,2R)}\omega(e)|e\cap D(0,2R)|^2.
    \]
    It is easy to see that the quantity $\omega(e)|e\cap D(0,2R)|^2$ is the area of $\frac{|e\cap D(0,2R)|}{|e|}Q_e$. Therefore, we have 
    $$\mathrm{COND}(V_0,V_1)\le 9\pi.$$ This means 
    $\mathrm{RES(V_0,V_1)}\ge \frac{1}{9\pi}$.
    By induction, we can construct infinitely many vertex sets $V_0,V_1,V_2,\cdots$ such that 
    $V_i$ separates $V_{i-1}$ and $V_{i+1}$ with 
    $\mathrm{RES}(V_{2i},V_{2i+1})\ge \frac{1}{9\pi}$. By Proposition \ref{addlemma2}, $\mathcal{T}$ is recurrent.
\end{proof}

Also, we have the following estimation of the weight $\omega$.
\begin{lem}\label{partialthetabound}
		Assume $\Theta_i, \Theta_j, \Theta_k \in[0, \pi-\epsilon]$, where $0<\epsilon<\pi$. Suppose that they satisfy (\ref{cos-condition}). Then in the Euclidean background geometry, there is a constant $C=C(\epsilon)>0$ so that
		$$
		0\le \omega_{ij}=r_j\frac{\partial \vartheta_i}{\partial r_j}<C\vartheta_i.
		$$
	\end{lem}
	\begin{proof}
 Just as the proof of \cite[Lemma~3.3]{He}, the center of the face $h$ is contained in the convex hull of circles $C_i$ and $C_j$. Therefore, $0\le H_k\le r_i+r_j$. Hence, 
		\[
		r_j\frac{\partial \vartheta_i}{\partial r_j}\le \frac{r_i+r_j}{\sqrt{r_i^2+r_j^2+2r_ir_j\cos\Theta_k }}\le \frac{2}{\sin{\epsilon}}.
		\]
		Therefore, when $\vartheta_i\le 1$, it is easy to see that 
		\[
		\frac{H_k}{l_k}\le \tan\vartheta_i\le C\vartheta_i,
		\]
		with some universal constant $C$. When $\vartheta_i>1$, we have
		\[\frac{H_k}{l_k}\le C(\epsilon)\vartheta_i,\]
		where $C(\epsilon)=\frac{2}{\sin\epsilon}.$
	\end{proof}
    The following Lemma links the resistance to the vertex extremal length. 
    \begin{lem}\cite[Lemma 5.4]{He}\label{VEL_current}
		If there is a uniform constant $C>0$, such that 
		$$
		\sum_{j \sim i}\omega_{ij} \leq C
		$$
		for any $i\in V$, then for any pair of mutually disjoint, nonempty sets $V_1, V_2$, we have
		$$
		\vel (V_1, V_2)  \leq 2C \cdot \res (V_1, V_2).
		$$
		Moreover, if $G$ is connected and VEL-parabolic, then  $(G, \omega)$ is recurrent.
	\end{lem}

In the end, we introduce the vertex extremal length of the cuts on graphs, which is the dual of that of the path set.	
	\begin{defn}
		We say the set $V_3$ \textbf{separates} $V_1$ from $V_2$, if $V_3\cap \gamma \neq \emptyset$, for any $\gamma \in \Gamma(V_1, V_2)$.
	\end{defn}
	
	Let $\Gamma^*(V_1, V_2)= \{\alpha \subset V: \alpha \cap \gamma \neq \emptyset, \forall \gamma \in \Gamma(V_1, V_2)\}$, there is an important argument concerning the duality of vertex extremal length:
	\begin{thm}(\cite{He2})\label{veldual}
		For any two nonempty sets $V_1, V_2 \subset V$ (we allow $V_2= \infty$), we have
		$$
		\vel (\Gamma^*(V_1, V_2))= 	\vel (\Gamma(V_1, V_2))^{-1}.
		$$ 
	\end{thm}
	By the definition of $\Gamma$ and $\Gamma^*$, we have the following  properties.
	\begin{prop}\label{nestlemma}
		Let $V_0, V_1, V_2, \cdots, V_n$ be mutually disjoint, finite, connected subsets of vertices. We allow $V_n = \infty$. Assume that for any $0 \leq i_1 < i_2 < i_3 \leq n$, the set $V_{i_2}$ separates $V_{i_1}$ from $V_{i_3}$. Then $\Gamma^*(V_a, V_b) \subset \Gamma^*(V_a^*, V_b^*)$, if $a^*<a<b<b^*$.
	\end{prop}
	\begin{proof}
		Since the sets $V_a, V_b, V_a^*, V_b^*$ satisfy the nested separation condition, any path $\gamma$ joining $V_a^*$ and $V_b^*$ must contain a sub-path $\gamma_0$ joining $V_a$ and $V_b$. For any $\alpha \in \Gamma^*(V_a, V_b) $, we have $\emptyset \neq \alpha \cap \gamma_0 \subset \alpha \cap \gamma$, which means $\alpha \in \Gamma^*(V_a^*, V_b^*)$.
	\end{proof}

	\subsection{The uniformization theorem for RCPs}
	
	In this section, we establish the uniformization theorem for RCPs, which is stated as follows.
	
	\begin{thm}\label{uniformization_RCP}
		Let $\mathcal{T}= (V, E, F)$ be a disk triangulation with an angle function $\Theta\in[0,\pi)^E$ and $\sup_{e\in E}\Theta(e)<\pi$. The graph $G=(V,E)$ is induced from $\mathcal{T}$. Then the following two statements are equivalent:
		\begin{enumerate}
			\item [(\textbf{U1})] $G$ is VEL-parabolic (VEL-hyperbolic resp.).
			\item [(\textbf{U2})] $(G,\Theta)$ is RCP parabolic (RCP hyperbolic resp.).
		\end{enumerate}
	\end{thm}
	
	To prove the uniformization theorem for RCPs, we need to establish the estimate of vertex extremal length with the help of the geometry of an RCP.

	Let $\mathcal{P} = \{C_i\}_{i \in V}$ be a circle pattern realizing $(\mathcal{T}, \Theta)$. Let $D_i$ denote the closed disk bounded by the circle $C_i$, and let $o_i$ be the center of $C_i$.
	For a circle $C$, we denote by $V_C(\mathcal{P})$ the set of vertices $i \in V$ such that $D_i \cap C \neq \emptyset$.
	For any subset $A \subset V$, we define $\mathcal{P}(A) := \bigcup_{v \in A} D_v$.
	We write $C(o, r)$ for the circle centered at $o$ with radius $r$, and denote by $C(r)$ the circle centered at the origin with radius $r$, for brevity.
	\begin{lem}\label{es}
		Let $\mathcal{T}=(V,E,F)$ be a disk triangulation, and $\Theta\in [0 ,\pi)^E$ be the prescribed intersection angle that satisfies ($Z_2$). Suppose $\pac=\{C_i\}_{i \in V}$ is an RCP that realizes $(\mathcal{T}, \Theta)$. Then for any $0<r_1<r_2$, we have
		$$
		\vel (V_{C(r_1)}(\pac), V_{C(r_2)}(\pac)) \geq \frac{(r_2-r_1)^2}{(24+36\pi^2)r_2^2}.
		$$
	\end{lem}
	\begin{proof}
		Let $d_v = \mathrm{diam}(D_v \cap D(r_2))$.  We construct a vertex metric $\nu$, which is defined by
		$$
		\nu_v = \frac{d_v }{r_2-r_1}, \quad \forall v \in V.
		$$
		It is obvious that $\int_\gamma \nu_v\geq 1$, for any $\gamma \in \Gamma\left(V_{C(r_1)}(\pac), V_{C(r_2)}(\pac)\right)$. Thus, by definition, 
		\begin{equation}\label{es0}
			\begin{aligned}
				\modc( \Gamma\left(V_{C(r_1)}(\pac), V_{C(r_2)}(\pac)\right)) &\leq \area (\nu)\\
                &= \sum_{v\in V}  \nu_v^2= \frac{1}{(r_2-r_1)^2} \sum_{v\in V}  d_v^2.
			\end{aligned}
		\end{equation}
		
		Let $A= \{v\in V:  \area (D_v \cap D(r_2)) \geq \frac{1}{2}\area ( D(r_v))\}$, then
		\begin{equation}\label{es1}
			\sum_{v\in A}  d_v^2 \leq  \sum_{v\in A}  4 r_v^2 \leq \sum_{v\in A} \frac{8}{\pi}\area (D_v \cap D(r_2)) .
		\end{equation}
		By Lemma \ref{mostthree}, we know that for any point $z \in D(r_2)$, there are at most three disks covering it. Thus,
		\begin{equation}\label{es2}
			\sum_{v\in A} \area (D_v \cap D(r_2)) \leq 3 \area (D(r_2)) = 3\pi r_2^2.
		\end{equation}

		Let $B = \{v\in V: 0< \area (D_v \cap D(r_2)) \leq \frac{1}{2}\area ( D(r_v)) \}$. Then, for any vertex $v \in B$, it follows that $d_v \leq l(D_v \cap C(r_2))$. Thus, by Lemma \ref{mostthree}, we have 
		$$
		\sum_{v\in B} d_v \leq \sum_{v\in B} l(D_v \cap C(r_2)\leq 3  l(C(r_2))\leq 6\pi r_2.
		$$
		From $(x+y)^2\geq x^2+y^2, \forall x, y \geq 0$, we know 
		\begin{equation}\label{es3}
			\sum_{v\in B} d_v^2 \leq (6\pi r_2)^2
		\end{equation}
		
		Combining \eqref{es0}, \eqref{es1}, \eqref{es2}, \eqref{es3}, we have 
		$$
		\vel (V_{C(r_1)}(\pac), V_{C(r_2)}(\pac)) = \frac{1}{\modc( \Gamma\left(V_{C(r_1)}(\pac), V_{C(r_2)}(\pac)\right)) }\geq \frac{(r_2-r_1)^2}{(24+36\pi^2)r_2^2}.
		$$
	\end{proof}

	\begin{lem}\label{annuilemma}
		Let $\mathcal{T}=(V,E,F)$ be a disk triangulation, and let $\Theta\in [0 ,\pi)^E$ be an intersection angle function satisfying condition ($Z_2$). Suppose $\pac=\{C_i\}_{i \in V}$ is an RCP that realizes $(\mathcal{T}, \Theta)$. Let $V_0=\{v_0\}$, and let $V_1$, $V_2$ be two mutually disjoint, finite, connected subsets of vertices. Let $V_3 = \infty$. Assume that for any $0 \leq i_1 < i_2 < i_3 \leq 3$, the set $V_{i_2}$ separates $V_{i_1}$ from $V_{i_3}$. 
		If
		\begin{equation}\label{annuli_P}
			\operatorname{VEL}(V_1, V_2) > 9(24 + 36\pi^2),
		\end{equation}
		then there exists $\hat r > 0$ such that for any circle $C(o_{v_0}, r)$ with $r \in [\hat r, 2\hat r]$, where $o_{v_0}$ is the center of $C_{v_0}$, the vertex set $V_{C(o_{v_0}, r)}(\pac) := \{ i \in V : D_i \cap C(o_{v_0}, r) \neq \emptyset \}
		$ separates $V_1$ from $V_2$ in $\mT$
	\end{lem}
	
	\begin{proof}
		Without loss of generality, assume that the center $o_{v_0}$ of $C_{v_0}$ is located at the origin. Hence, we denote  $C(o_{v_0}, r)$ by $C(r)$ for convenience.  Let
		\begin{equation}\label{hatr}
			\hat r := \min \left\{ r > 0 : D(r) \cap D_i \neq \emptyset, \forall i \in V_1 \right\}.
		\end{equation}
		We claim that $\hat r$ is what we need.
		
		Consider an $\alpha \in \Gamma^*(V_1, V_2)$. From Proposition \ref{nestlemma}, we have $\alpha\in \Gamma^*(V_0, \infty)$. According to \eqref{hatr}, it follows that
		\begin{equation}\label{annuli_P0}
			\operatorname{diam}(\mathcal{P}(\alpha)) \geq \hat r.
		\end{equation}
		
		For the subsequent proof of this lemma, we argue by contradiction. If there is a $r \in [\hat r, 2\hat r]$ such that $V_{C(r)}(\pac) \notin \Gamma^*(V_1, V_2)$. Then there exists a $\gamma \in \Gamma(V_1, V_2)$ such that
		\[
		\gamma \cap V_{C(r)}(\pac) = \emptyset.
		\]
		This is equivalent to
		\[
		D_i \cap C(r) = \emptyset \quad \text{for all } i \in \gamma,
		\]
		which implies
		\begin{equation}\label{annuli_P1}
			\mathcal{P}(\gamma) \cap C(r)  = \emptyset.
		\end{equation}
		By \eqref{hatr} and $\gamma \in \Gamma(V_1, V_2)$, we know 
		\begin{equation}\label{annuli_P2}
			\mathcal{P}(\gamma) \cap D(0, \hat r) \neq \emptyset.
		\end{equation}
		Combining \eqref{annuli_P1} and \eqref{annuli_P2}, and noting that $r \in [\hat r, 2\hat r]$, it follows that
		\[
		\mathcal{P}(\gamma) \subset D(0, r).
		\]
		Since $\alpha \in \Gamma^*(V_1, V_2)$ and $\gamma \in \Gamma(V_1, V_2)$, we have
		\begin{equation}\label{annuli_P3}
			\mathcal{P}(\alpha) \cap D(0, r) \supset	\mathcal{P}(\alpha) \cap \mathcal{P}(\gamma) \supset \mathcal{P}(\alpha \cap \gamma) \neq \emptyset.
		\end{equation}
		Let $d_v= \operatorname{diam}(D_v \cap D(0, 3\hat r))$.
		Define a vertex metric $\nu$ on $\mathcal{T}$ by
		\[
		\nu_v:= \frac{d_v}{\hat r}.
		\]
		We claim $\nu$ is $\Gamma^*(V_1, V_2)$-admissible. We need to verify the following two cases:
        
		\medskip 
        \emph{Case 1}: if $\mathcal{P}(\alpha) \subset D(0, 3\hat r)$, then from \eqref{annuli_P0},
			\[
			\sum_{v\in \alpha} \nu_v= \sum_{v\in \alpha} \ \frac{d_v}{\hat r}= \sum_{v \in \alpha} \frac{r_v}{\hat r} \geq \frac{	\operatorname{diam}(\mathcal{P}(\alpha)) }{\hat r} \geq 1.
			\]
            
		\medskip
        \emph{Case 2}: if $\mathcal{P}(\alpha) \setminus D(0, 3\hat r) \neq \emptyset$, then $\mathcal{P}(\alpha) $ intersects both $D(0, r)$ and the complement of the disk $D(0, 3\hat r)$, which implies
			\[
			\sum_{v\in \alpha} \nu_v= \sum_{v\in \alpha} \ \frac{d_v}{\hat r}\geq \frac{3\hat r - r}{\hat r} \geq 1.
			\]
        Combining the two cases above, we know that $\nu$ is $\Gamma^*(V_1, V_2)$-admissible. Then it follows 
		\[
		\operatorname{VEL}(V_1, V_2) = \operatorname{MOD}(\Gamma^*(V_1, V_2)) \leq \operatorname{area}(\nu).
		\]
		For the same reason as in Lemma~\ref{es}, we have the following estimate
		\[
		\operatorname{VEL}(V_1, V_2)  \leq  \operatorname{area}(\nu) \leq 9(24 + 36\pi^2),
		\]
		which contradicts the assumption \eqref{annuli_P}. This completes the proof.
	\end{proof}

	\begin{lem}\label{parabolic_equivalent}
		Let $\mathcal{T}=(V,E,F)$ be a disk triangulation, and let $\Theta\in[0,\pi)^E$ be an intersection angle function. Suppose $\pac=\{C_i\}_{i \in V}$ is an RCP that realizes $(\mathcal{T}, \Theta)$. Then $\pac$ is locally finite in the plane if and only if $\mathcal{T}$ is VEL-parabolic.
	\end{lem}
	\begin{proof}
		
		The proof is divided in two parts.
		
		\medskip
		\emph{Part 1. We claim that if $\pac$ is locally finite in the plane, then $\mathcal{T}$ is VEL-parabolic.}
		\medskip
		
		Let $v_0$ be a fixed vertex of $V$. For a sufficiently large $r_1>0$, from Lemma \ref{sectionlemma}, we know that $V_{C(r_1)}(\pac)$ is a set of vertices that contains a loop as a  finite vertex cut separating $\pac(v_0)$ from infinity.
		
		Since $V_{C(r_1)}(\pac)$ is finite, there exists a radius $r_2> 2r_1$ such that $D_v\cap C(r_2)= \emptyset$, which means
		$V_{C(r_1)}(\pac) \cap V_{C(r_2)}(\pac)= \emptyset$. Repeating the same procedure, we can get a sequence $\{r_n\}_{n=0}^\infty$ such that $r_{n+1}> 2r_n$ and $V_{C(r_n)}(\pac) \cap V_{C(r_n+1)}(\pac)= \emptyset$. 
		
		According to Lemma \ref{sectionlemma}, we know $V_{C(r_n)}(\pac)$ is a set of vertices lying around a simple closed finite path in $\mathcal{T}$. Thus, for any $1\leq n_i < n_2<n_3$, the set $V_{C(r_{n_2})}(\pac)$ separates $V_{C(r_{n_1})}(\pac)$ and $V_{C(r_{n_3})}(\pac)$. From Proposition \ref{addlemma} and Lemma \ref{es}, we have
		$$
		\begin{aligned}
			\vel(V_{C(r_1)}(\pac), \infty) &\geq \sum_{n=1}^{\infty} \vel(V_{C(r_n)}(\pac), V_{C(r_{n+1})}(\pac))\\
			&\geq  \sum_{n=1}^{\infty}  \frac{(r_{n+1}-r_n)^2}{(24+36\pi^2)r_{n+1}^2}\geq  \sum_{n=1}^{\infty}  \frac{1}{4(24+36\pi^2)}\\
			&= + \infty.
		\end{aligned}
		$$
		
		\medskip
		\emph{Part 2. We claim that if $\mathcal{T}$ is VEL-parabolic, then $\pac$ is locally finite in the plane.}
        \medskip
		
		Choose an increasing exhaustion by finite triangulated closed disks
        $\mathcal T^{[n]}$ with simple closed boundary such that each $|\mathcal T^{[n]}|$ is homeomorphic to a closed disk, $|\mathcal T^{[n]}|\subset \operatorname{int} |\mathcal T^{[n+1]}|$ and $\bigcup_{n=1}^{\infty}|\mathcal T^{[n]}|=|\mathcal T|$. Let $\partial V^{[n]}$ be the vertices on $\partial|\mathcal T^{[n]}|$. It is easy to see that $\{\partial V^{[n]}\}_{n=0}^\infty$ is a set of mutually disjoint, finite, and connected subsets of vertices such that $\partial V^{[n]} \in \Gamma^*(\partial V^{[n_1]}, \partial V^{[n_2]})$ for any $n_1\leq n\leq n_2$. Since $\mathcal{T}$ is VEL-parabolic, there is a subsequence $\{V_i=\partial  V^{[n_i]}\}_{i=0}^\infty$ such that 
        		$$\vel (V_i, V_{i+1}) \geq 9(24 + 36\pi^2).$$
		
		By Lemma \ref{annuilemma}, there are two numbers $\hat r_k$ and $\hat r_{k+1}$ such that $V_{C(2\hat r_k)}(\pac) \in \Gamma^*(V_{2k-1}, V_{2k})$ and $V_{C(\hat r_{k+1})}(\pac) \in \Gamma^*(V_{2k+1}, V_{2k+2})$. From Lemma \ref{sectionlemma}, we know $V_{C(2\hat r_k)}(\pac) \cap V_{C(\hat r_{k+1})}(\pac) \subset V_{2k}\cap V_{2k+1} = \emptyset$. Hence, $\hat r_{k+1}> 2 \hat{r_k}$ and $\lim_{k\rightarrow \infty} \hat r_k = + \infty.$
		
		Using the method of contradiction, it is easy to see that $V_{C(\hat r_k)}(\pac) \subset \mathrm{Interior}(V_{2k})$. Hence $V_{C(r)}(\pac) $ is a finite set for any $r\leq \hat r_k$. Since $\lim_{k\rightarrow \infty} r_k = + \infty$, it follows that $V_{C(r)}(\pac)$ is a finite set for any $r < + \infty$. 
	\end{proof}
	
	\begin{lem}\label{sequenceaaq}
		Let $\mathcal{T}=(V,E,F)$ be a disk triangulation, and $\Theta\in [0 ,\pi)^E$ be the prescribed intersection angle which satisfies ($Z_2$). Suppose $\pac=\{C_i\}_{i \in V}$ is an RCP that realizes $(\mathcal{T}, \Theta)$. Then $\mathcal{T}$ is VEL-hyperbolic, if and only if $\lim _{k \rightarrow \infty} \mathrm{VEL}\left(V_k, \infty\right)=0$, for any increasing sequence of finite, connected subsets of vertices $V_k$ in $\mT$ such that $V_{\infty}=\bigcup_{k=1}^{\infty} V_k$ is infinite.
	\end{lem}
	\begin{proof}
		
		The proof is divided into two parts.
		
		\medskip
		\emph{Part 1. Suppose that $\lim _{k \rightarrow \infty} \mathrm{VEL}\left(V_k, \infty\right)=0$, then $\mathcal{T}$ is VEL-hyperbolic.}
        \medskip
		
		Since $\lim _{k \rightarrow \infty} \mathrm{VEL}\left(V_k, \infty\right)=0$, for sufficiently large $k$, we have $\vel (V_k, \infty)<+ \infty$, which implies that $\mathcal{T}$ is VEL-hyperbolic.
		
		\medskip
		\emph{Part 2. Suppose that $\mathcal{T}$ is VEL-hyperbolic, then $\lim _{k \rightarrow \infty} \mathrm{VEL}\left(V_k, \infty\right)=0$.}
        \medskip
		
		Let $\Omega = \carrier(\pac)$. According to Lemma \ref{parabolic_equivalent}, we know $\Omega \neq \mathbb C$. Let $z_\infty \in \partial \Omega$ be a limit point of $\pac(V_\infty)$. Without loss of generality, we can assume $z_\infty= \infty$. Otherwise we can transform $\pac$ by an inversion with respect to a circle centered at $z_\infty$. Let $r_0>0$ be a real number large enough such that $\pac(V_1)\subset D(r_0)$ and $ D(r_0)\cap \partial \Omega \neq \emptyset$. Let $r_k$ ($k\geq1$) be an increasing sequence such that $r_{k+1}>2r_k$ and $V_{C( r_k)}(\pac) \cap V_{C( r_{k+1})}(\pac) = \emptyset$.
		
		Denote $\tilde r_k = \inf \{r \geq 0:  D(r)\cap \pac(V_k)\neq \emptyset\}$. Since $\infty$ is the limit point of $\pac(V_\infty)$, we know $\lim_{k\rightarrow \infty} \tilde r_k = + \infty.$ Without loss of generality, we can assume $\tilde r_k >r_k$, otherwise, we can subtract a subsequence from sequence $\{\tilde r_k\}$. As $ D(r_0)\cap \partial \Omega \neq \emptyset$, any path $\alpha \in \Gamma^*(V_{2n},\infty)$ contains a disjoint union of paths $\gamma_i \in \Gamma(V_{C(r_{2i-2})}(\pac), V_{C(r_{2i})}(\pac) )$, $i=1,2,\cdots,n$. Thus, by Lemma \ref{es}, we have
		$$
		\vel (\Gamma^*(V_{2n},\infty))\geq \sum_{i=1}^n \vel(V_{C(r_{2i-2})}(\pac), V_{C(r_{2i})}(\pac) )\geq \frac{n}{96+144\pi^2}.
		$$
		It follows that 
		$$
		\vel  (\Gamma(V_{2n},\infty))= \frac{1}{\vel (\Gamma^*(V_{2n},\infty))} \leq\frac{96+144\pi^2}{n}\rightarrow 0
		$$
		as $n \rightarrow + \infty$.
	\end{proof}

	\begin{lem}\label{hyperbolic_equivalent}
		Let $\mathcal{T}=(V,E,F)$ be a disk triangulation with an angle function $\Theta\in[0,\pi)$ and $\sup_{e\in E}\Theta(e)<\pi$. If ($Z_1$), ($Z_2$) and ($Z_3$) hold, then the following two properties are equivalent:
		\begin{itemize}
			\item[(1)] There exists an RCP $\pac=\{C_i\}_{i \in V}$ and $\carrier(\pac)= \mathbb{U}$.
			\item[(2)] $\mathcal{T}$ is VEL-hyperbolic.
		\end{itemize}
	\end{lem}
	\begin{proof}
		
		The proof is divided into two parts.
		
		\medskip
		\emph{Part 1. If $\pac$ is an RCP and $\carrier(\pac)= \mathbb{U}$, then $\mathcal{T}$ is VEL-hyperbolic.}
        \medskip

		If not, assume that $\mathcal{T}$ is VEL-parabolic. Then Lemma~\ref{parabolic_equivalent} implies that $\pac$ is locally finite in the plane, contradicting the fact that $\carrier(\pac)= \mathbb{U}$.

		\medskip
		\emph{Part 2. If $\mathcal{T}$ is VEL-hyperbolic, then there exist an RCP $\pac=\{C_i\}_{i \in V}$ and $\carrier(\pac)= \mathbb{U}$.}
        \medskip
		
		Choose an arbitrary vertex $v_0\in V$. Since $\mathcal T=(V,E,F)$ is a locally finite triangulation of the open disk,
        we choose an increasing exhaustion by finite triangulated topological disks
        \[
        \left\{\mathcal T^{[n]}=\left(V^{[n]},E^{[n]},F^{[n]}\right)\right\}_{n=1}^{\infty}
        \]
        such that each $|\mathcal T^{[n]}|$ is homeomorphic to a closed disk,
        \[
        |\mathcal T^{[n]}|\subset \operatorname{int} |\mathcal T^{[n+1]}|,
        \qquad
        \bigcup_{n=1}^{\infty}|\mathcal T^{[n]}|=|\mathcal T|,
        \]
        and every compact subset of $|\mathcal T|$ is contained in $|\mathcal T^{[n]}|$
        for all sufficiently large $n$.
        Here $|\mathcal T^{[n]}|$ denotes the underlying polyhedron of the finite
        subcomplex $\mathcal T^{[n]}$.
		Following the proof in Theorem \ref{RCPe}, notice that we have assumed $\Theta_s^{[n]}([v, v_\infty])=0$, which implies that for any $v \in \partial V^{[n]}$, the $C_v^{[n]}\in \pac^{[n]}$ is tangent to a unit circle if we give a suitable scale for $\pac^{[n]}$. Without loss of generality, we assume $\pac^{[n]}$ is just with such suitable scale.
		
		By Maximum Principle Theorem \ref{MaximumPrincipleHyper} in Hyperbolic plane, we have $r_{v_0}^{[i]}> r_{v_0}^{[j]}$, if $i<j$. 
		
		Let $$R^{[n]}= \frac{1}{r_{v_0}^{[n]}}$$ and $$\pac^{\prime[n]}= R^{[n]}\pac^{[n]}.$$
		It is obvious that $r_{v_0}^{\prime[n]}=1$, and for any $v \in \partial V^{[n]}$, the $C_v^{\prime[n]}\in \pac^{\prime[n]}$ is tangent to circle $C(R^{[n]})$. Since $R^{[n]}$ is increasing , it has a limit  $R^{[\infty]}$. By Lemma \ref{parabolic_equivalent}, we know $R^{[\infty]}< +\infty$. Otherwise, $\pac_\infty$ is a locally finite RCP, which means $\mathcal{T}$ is VEL-parabolic.
        Thus, for the same reason as in the proof in Theorem \ref{RCPe}, we know there exists an RCP $\pac=\{C_i\}_{i \in V}$ that is contained in $D(R^{[\infty]})$.

		Assume that $\pac$ is not locally finite in $U$. Then there exists a vertex $x\in \mathrm{Int}(D(R^{[\infty]})) \cap \partial\carrier(\pac)$. Since $\partial\carrier(\mathcal{P}_\infty)$ is a closed set, without loss of generality, we can assume that $x$ is a point in $\partial\carrier(\mathcal{P}_\infty)$ closest to the origin $o_{v_0}$. It is obvious that 
		\begin{align*}
			\# \{v\in V:r_v\ge \delta\}<+\infty,~\forall \delta>0.
		\end{align*}
		Then there exists a connected infinite vertex set $W=\{v_i\}_{i\in\mathbb{N}_+}$ such that $\lim_{i\rightarrow\infty}c(v_i)=x$, and $W_i=\{v_1,\cdots,v_i\}$ is connected for each $i$. Then $W=\cup_{i=1}^\infty W_i$. 
		
		Let $R_0= d_{\mathbb{R}^2}(o_{v_0}, x)$, and $R_1, R_2$ be two constants such that $R_0< R_1<R_2<R^{[\infty]}$. Then, there is a $n_k$ such that  $R^{[n_k]}>R_2$ and $C_{v_i}^{[n_k]} \cap D(R_1) \neq \emptyset$, $\forall i\leq k$. From Lemma \ref{es}, we have 
		$$
		\vel (W_i, \partial V^{[n_k]}) \geq  \vel(V_{C(R_1)}(\pac^{[n_k]}), V_{C(R_2)}(\pac^{[n_k]}) )\geq \frac{(R_2-R_1)^2}{(24+36\pi^2)R_2^2},
		$$
		which contradicts Lemma \ref{sequenceaaq}.
	\end{proof}
	\begin{proof}[Proof of Theorem \ref{uniformization_RCP}]
		This is directly deduced from Lemma \ref{parabolic_equivalent} and \ref{hyperbolic_equivalent}.
	\end{proof}

	\section{The rigidity of RCPs}\label{sec:rigid}

	\subsection{The rigidity of hyperbolic RCPs}
    The proof of the rigidity of infinite RCPs is along the same lines as He's work \cite{He}. 
	In order to prove the rigidity of RCPs in the hyperbolic case, we need to develop a maximum principle for $\Theta \in [0,\pi)^E$ in the hyperbolic background geometry. Let $\mathcal{T}=(V,E,F)$  be a finite cellular decomposition of a disk, and   $\mathcal{P}= \{C_i\}_{i\in V}$ be a finite ideal disk pattern realizing $(\mathcal{T}, \Theta)$ on the complex plane $\mathbb{C}$. Recall that $D_i$ is the closed disk bounded by the circle  $C_i$. 
	
	Let $\mathbb{U}= \{z\in \mathbb C: |z|<1\}$ be the unit disk with hyperbolic metric $ds= 2|dz|/(1-|z|^2)$. We consider a special RCP  $\mathcal{P}= \{C_i\}_{i\in V}$ such that  $D_i\cap \mathbb{U} \neq \emptyset$, for any $i \in V$.  If $C_i$ is entirely lying in $\mathbb{U}$, it is called a \textbf{hyperbolic circle}; if $C_i$ is internally tangent to $\partial \mathbb{U}$, it is called a \textbf{horocycle}; if $C_i$ intersects $\partial \mathbb{U}$ in two points, it is called a \textbf{hypercycle} (see Figure \ref{generalized cycles}). We refer to the above three types of cycles as \textbf{generalized cycles}. We call the tangent point of a horocycle and the unit circle $\partial \mathbb{U}$ the \textbf{center} of the horocycle. For each hypercycle $C$, there exists a geodesic line that is equidistant to $C$, which is also called the ``center'' of $C$. 
	\begin{figure}[h]
		\centering
		\includegraphics[width=0.42\textwidth]{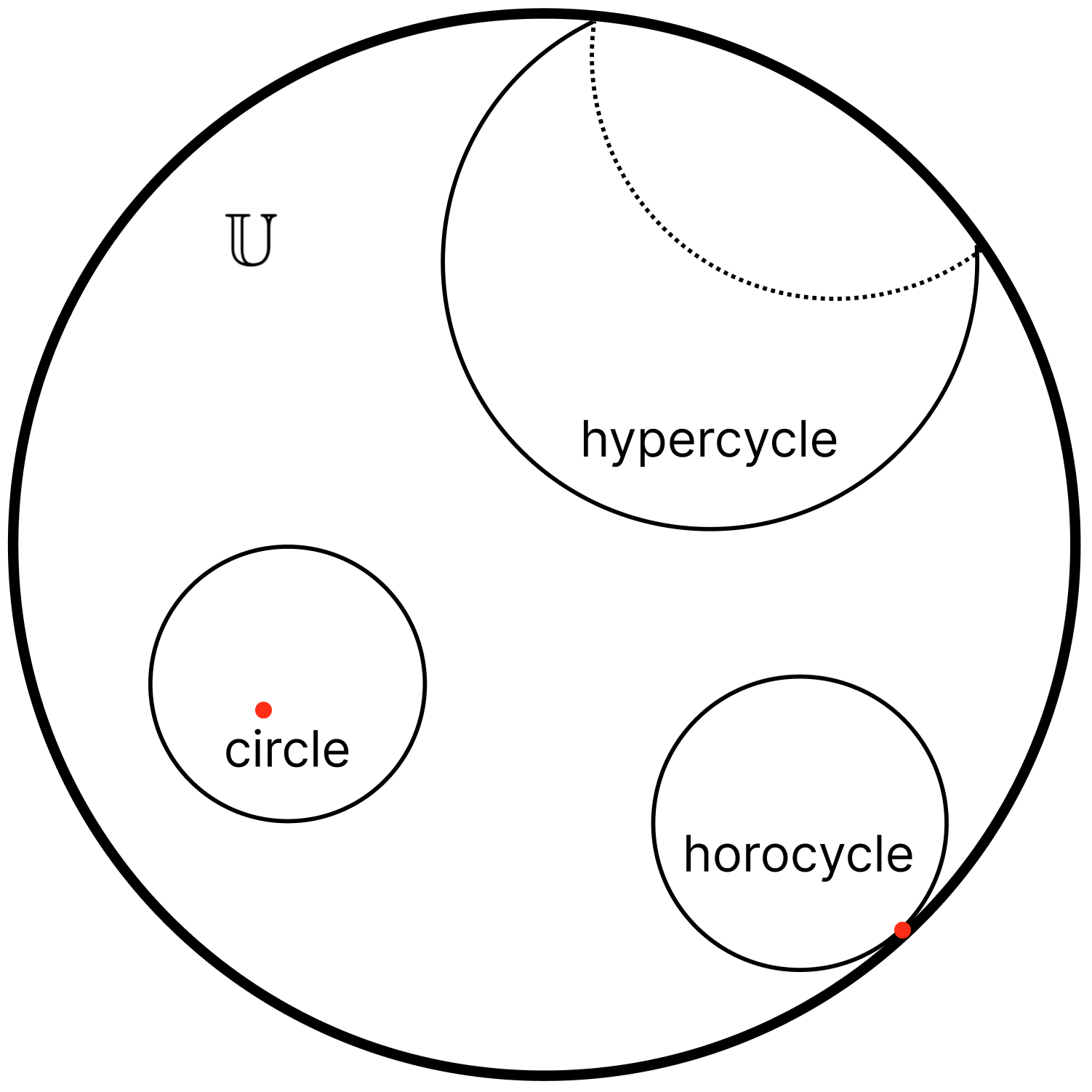}
		\caption{three types of generalized cycles}
		\label{generalized cycles}
	\end{figure}
	
	The following \textbf{geodesic curvature} provides a convenient parameter for generalized cycles. 
    \begin{defn}[Geodesic curvature]
    Let $r_i$ be the Euclidean radius of $C_i$, and $o_i$ be the center of $C_i$ in $\mathbb C$. If $C_i$ is a hyperbolic circle, we denote by $\rho_i$ the hyperbolic radius of $C_i$ with respect to $\mathbb{U}$. If $C_i$ is a hypercycle, we denote by $\alpha_i$ the intersection angle of $C_i$ and $\partial \mathbb{U}$. The geodesic curvature $g_i$ of $C_i$ is given by
	\begin{equation}\label{g_if}
		g_i= \begin{cases}\coth \rho_i & \text {for hyperbolic circle;}  \\ 1 & \text {for horocycle;} \\ -\cos \alpha_i & \text {for hypercycle} .\end{cases}
	\end{equation}
    \end{defn}
	It is obvious that $g_i(r_i, o_i)$ is a function on $o_i$ and $r_i$ such that $g_i$ is strictly decreasing with respect to $d_{\mathbb{R}^2}(0, o_i)$ or $r_i$. In fact, we can see the horocycle as a hypercycle with $\alpha_i =\pi$ or a hyperbolic circle with $\rho_i = + \infty$.
	
	\begin{lem}\label{dgdr}
		Let $\Theta \in [0,\pi)$ be fixed. Given a hyperbolic circle $C_0$ with Euclidean center $o_0$ at the origin of $\mathbb{C}$ and Euclidean radius $r_0<1$, consider any generalized cycle $C=C(o,r)$ that intersects $C_0$ at an angle $\Theta$. Denote by $g$ the geodesic curvature of $C$. Then
		\[
		\frac{dg}{dr}<0.
		\]
	\end{lem}
	
	\begin{proof}
		Let $d=d_{\mathbb{R}^2}(0,o)$. From elementary geometry,
		\[
		d^2=r_0^2+r^2+2r_0r\cos\Theta .
		\]
		
		It is clear that, $C$ is a hyperbolic circle if $d<1-r$, a horocycle if $d=1-r$, and a hypercycle if $1-r<d<1+r$.
		
		From \eqref{g_if}, we obtain
		\[
		g=\begin{cases}
			\coth\!\left(\operatorname{arctanh}(d+r)-\operatorname{arctanh}(d-r)\right), & d<1-r, \\[4pt]
			1, & d=1-r, \\[4pt]
			\dfrac{1+r^2-d^2}{2r}, & 1-r<d<1+r.
		\end{cases}
		\]
		
		We now check $\dfrac{dg}{dr}$ in each case.

        \medskip

        \emph{Case 1:}  $d<1-r$. 
		Let $\dot{d}$ and $\rho$ denote the derivative of $d_{\mathbb{R}^2}(r)$ with respect to $r$ and the hyperbolic radius of $C$. Then
		\[
		\begin{aligned}
			\frac{dg}{dr}
			&=- \frac{1}{\sinh^2 \rho}\left[\frac{\dot{d}+1}{1-(d+r)^2}-\frac{\dot{d}-1}{1-(d-r)^2}\right] \\[4pt]
			&=- \frac{4\dot{d}dr+2(1-d^2-r^2)}{\sinh^2\rho(1-(d+r)^2)(1-(d-r)^2)} \\[4pt]
			&=- \frac{2-2r_0^2}{\sinh^2\rho(1-(d+r)^2)(1-(d-r)^2)} <0 .
		\end{aligned}
		\]

        \medskip

        \emph{Case 2:}  $1-r<d<1+r$. 
		Here
		\[
		g=\frac{1+r^2-d^2}{2r}
		=\frac{1-r_0^2}{2r}-r_0\cos\Theta ,
		\]
		which immediately implies $\dfrac{dg}{dr}<0$.
		
		\medskip
		Thus, in all cases, $\dfrac{dg}{dr}<0$ as claimed.
	\end{proof}

	\begin{lem}\label{gthreeconfig}
		Let $\Theta_i, \Theta_j, \Theta_k \in [0,\pi)$ be three angles satisfying \eqref{sum<pi}
		or \eqref{angle-condition}.
		For any $g_i > 1$ and $g_j, g_k > -1$, there exists a configuration of three intersecting generalized cycles 
		$C_i, C_j, C_k$, unique up to conformal transformation, having geodesic curvatures 
		$g_i, g_j, g_k$ and meeting in intersection angles $\Theta_i, \Theta_j, \Theta_k$.
	\end{lem}
	
	\begin{proof}
		Without loss of generality, assume that $o_i$ lies at the origin in $\mathbb{C}$, 
		$o_j$ lies on the positive real axis, and $o_k$ lies in the upper half-plane, 
		as shown in Figure~\ref{hyperbolic_three_config}.
		
		It is clear that $C_i$ and $C_j$ are completely determined by $g_i$, $g_j$ and $\Theta_k$. 
		Thus, the radii $r_i$, $r_j$ and the location of $o_j$ are also determined. 
		We now let $r_k$ vary.
		
		From Lemmas~\ref{sanyuangouxing_yinli}, we know that for each $r_k > 0$ there exists 
		a unique circle $C_k$ such that $C_i, C_j, C_k$ meet with intersection angles 
		$\Theta_i, \Theta_j, \Theta_k$. 
		By Lemma \ref{dgdr}, when $r_k$ increases, $g_k$ strictly decreases. 
		Hence, by the intermediate value theorem, there exists a unique $r_k$ such that 
		$C_k$ has geodesic curvature $g_k$. 
		This establishes the existence of the configuration.
		
		The uniqueness follows directly from Lemma~\ref{sanyuangouxing_yinli}.
	\end{proof}
	
	\begin{figure}[h]
		\centering
		\includegraphics[width=0.5\textwidth]{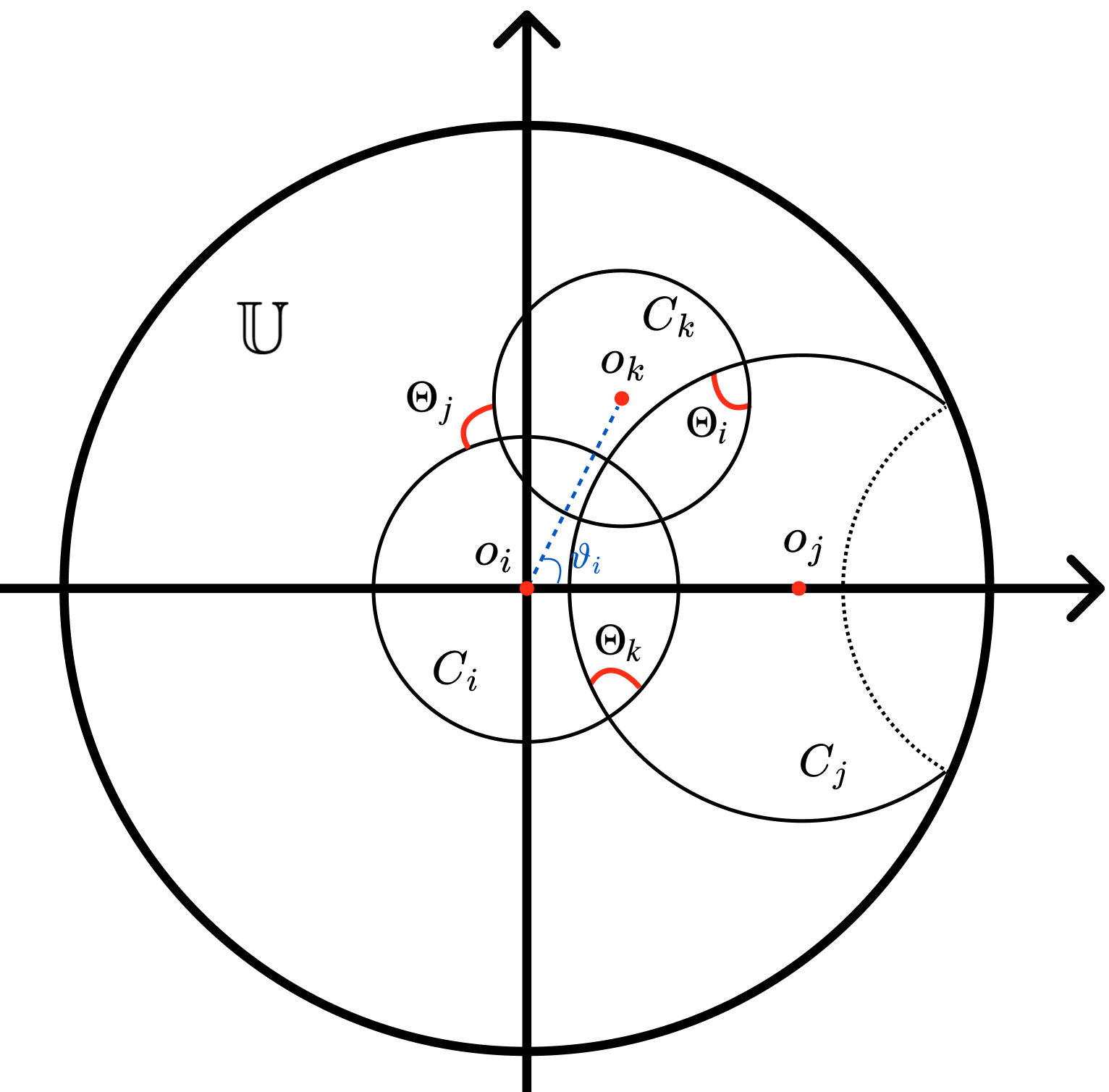}
		\caption{the configuration of three generalized cycles}
		\label{hyperbolic_three_config}
	\end{figure}
	
	\begin{lem}\label{gvari}
		Let $\Theta_i, \Theta_j, \Theta_k \in [0,\pi)$ be three angles satisfying \eqref{cos-condition}.
		For any $g_i > 1$ and $g_j, g_k > -1$, in hyperbolic geometry, by the previous lemma, there exists a three circle configuration formed by circles $C_i,C_j$ and $C_k$ with geodesic curvature equal to $g_i,g_j$ and $g_k$. Let $c_i,c_j$ and $c_k$ be the centers of $C_i, C_j$ and $C_k$. We denote by $\nu_i$ the inner angle between two geodesic segments connecting $\{c_i,c_j\}$ and $\{c_i,c_k\}$ (when $c_j$ or $c_k$ is a geodesic line that is the center of a hypercycle, then the geodesic segment connecting it with $c_i$ refers to the geodesic segment starting from $c_i$ that is perpendicular to $c_j$ or $c_k$, respectively).  Then we have
		\[
		\frac{\partial \vartheta_i}{\partial g_i} > 0, 
		\qquad 
		\frac{\partial \vartheta_i}{\partial g_j} \leq 0.
		\]
	\end{lem}
	
	\begin{proof}
		Without loss of generality, assume that $c_i$ lies at the origin in $\mathbb{C}$. Let $o_i,o_j$ and $o_k$ be the Euclidean centers of $C_i,C_j$ and $C_k$. 
		Then the angle $\vartheta_i$ in hyperbolic geometry coincides with the Euclidean angle 
		$\angle o_j o_i o_k$, as shown in Figure~\ref{hyperbolic_three_config}.
		From Lemma \ref{dgdr}, we know if $g_i$ is fixed, then
		\[
		\frac{d g_j}{d r_j} < 0.
		\]
		Since for $C_i$, we have $\rho_i = 2\operatorname{arctanh}(r_i)$ by elementary hyperbolic geometry. From~\eqref{g_if}, we also have $d g_i / d r_i < 0$ when $g_j, g_k$ are fixed. 
		Thus, the desired inequalities follow directly from Lemma~\ref{vari}.
	\end{proof}

	Also, we have the following lemma for three hyperbolic circles.

	Then we  establish the following maximum principle for hyperbolic background geometry:

	\begin{thm}\label{MaximumPrincipleHyper}(Maximum principle in hyperbolic plane)
		Let $G=(V,E)$ be a connected triangulation, and $\Theta\in [0 ,\pi)^E$ be the prescribed intersection angle such that (Z4) holds. Let $V=\partial V\cup \mathrm{int}(V)$, where $\partial V, \mathrm{int}(V)$ are two vertex sets.
		Assume $\mathcal{P}= \{C_i\}_{i\in V}$ and $\mathcal{P^*}= \{C^*_i\}_{i\in V}$ are two RCPs with circles intersecting the unit disk $U$ realizing $(\mathcal{T}, \Theta)$. Suppose $\mathcal{P} \subset \mathbb{U}$, $\{C_v^*\}_{v\in \mathrm{int(V)}} \subset \mathbb{U}$, and $C_v^*$ is not contained in $\mathbb{U}$ for all $v\in\partial V$. 
		Viewing all circles as generalized cycles in hyperbolic space.
		Let $g_v$ and $g^*_v$ respectively denote the geodesic curvatures of $C_v$ and $C_v^*$ respectively. Then $$g_v\ge g^*_v,~\forall v\in V $$
	\end{thm}
	\begin{proof}
		The statement is true for all vertices in $\partial V$, so we only need to prove it for the vertices in $\mathrm{int}(V)$.
		Assume, for contradiction, that the maximum of $g_v<g^*_v$ is attained at an interior vertex $v$. For each vertex $w\in\mathrm{int}(V)$ and edge $[s,l]$ such that $[w,s,l]$ is a face. We can define the inner angle $$\nu_w^{s,l}=\nu_w^{s,l}(g_w,g_s,g_l)$$ as in Lemma \ref{gvari}. Denote 
        $$\tilde E(i): = \{[j,k]\in E: [i,j,k]\in F \}, \quad \forall i\in V$$
        and
        $$\tilde V([j,k]): = \{i \in V: [i,j,k]\in F \}, \quad \forall [j,k]\in E.$$
        Since $\pac$ and $\pac^*$ are two circle patterns immersed in the Euclidean plane, we have 
		\begin{align}\label{equation_sum}
			\sum_{[j,k]\in \tilde E(i)}\theta_i^{jk}(g_i,g_j,g_k)=\sum_{[j,k]\in \tilde E(i)}\theta_i^{jk}(g_i^*,g_j^*,g_k^*)=2\pi,~\forall i\in\mathrm{int}(V).
		\end{align}
		Let $$\tilde{g}_i=\begin{cases}
			g^*_i, & i\in \mathrm{int}(V), \\[4pt]
			g_i, & i\in\partial V.
		\end{cases}$$
		It is easy to see that $\tilde{g}_i>1,~\forall i\in V.$
		By Lemma \ref{gvari}, we have 
		\begin{align}\label{neq_sum}
			\sum_{[j,k]\in \tilde E(i)}\theta_i^{jk}(\tilde{g}_i,\tilde g_j,\tilde g_k)<2\pi,~\forall i\in \mathrm{int}(V).
		\end{align}
		Let $u_i,u_i^*,\tilde{u}_i$ be the discrete conformal factor related to $g_i,g^*_i,\tilde{g}_i$ respectively, see Lemma \ref{conformalfactor}. It is easy to see that $\frac{\mathrm{d}u}{\mathrm{d}g}<0$. By \eqref{equation_sum} and \eqref{neq_sum}, we have 
		\begin{align*}
			\sum_{[j,k]\in \tilde E(i)}\theta_i^{jk}(u_i,u_j,u_k)-\theta_i^{jk}(\tilde{u}_i,\tilde u_j,\tilde u_k)\ge0,~\forall i\in\mathrm{int}(V).
		\end{align*}
		Therefore, we have
		\begin{align}\label{harmonic1}
			\sum_{j:j\sim i}A_{ij}[(u_j-\tilde{u}_j)-(u_i-\tilde{u}_i)]+B_{i}(u_i-\tilde{u}_i)\ge0,
		\end{align}
		where
		\[
		A_{ij}=\sum_{k \in \tilde V([i,j])}\int_0^1 \pp{\theta_i^{jk}}{u_j}((1-t)\tilde u_i+tu_i,(1-t)\tilde u_j+tu_j,(1-t)\tilde u_k+tu_k)\mathrm{d}t>0,
		\]
		and 
		\begin{align*}
			& B_{i}=\sum_{[j,k]\in \tilde E(i)}\pp{\theta_i^{jk}}{u_i}((1-t)\tilde u_i+tu_i,(1-t)\tilde u_j+tu_j,(1-t)\tilde u_k+tu_k)\mathrm{d}t
			\\&-\sum_{[j,k]\in \tilde E(i)}\pp{\theta_i^{jk}}{u_j}((1-t)\tilde u_i+tu_i,(1-t)\tilde u_j+tu_j,(1-t)\tilde u_k+tu_k)\mathrm{d}t\\&-\sum_{[j,k]\in \tilde E(i)}\pp{\theta_i^{jk}}{u_j}((1-t)\tilde u_i+tu_i,(1-t)\tilde u_j+tu_j,(1-t)\tilde u_k+tu_k)\mathrm{d}t<0.
		\end{align*}
		Let $w_i=\max\{u_i-\tilde u_i,0\}$. By \ref{conformalfactor}, we have $A_{ij}=A_{ji}$, and $w$ is a subharmonic function with respect to the weight $A$. Since $w_i=0,~\forall i\in \partial V$, by Lemma \ref{maxharmonic} we see that $$w_i\le 0,~\forall i\in V.$$ Therefore, $u_i\le \tilde u_i,~\forall i\in V.$ Hence, $g_i^*\le g_i,~\forall i\in V$.
	\end{proof}
	With the above maximum principle, we have the following rigidity theorem. 
	\begin{thm}\label{rigidity_hyperbolic}
		Let $\mathcal{T} = (V, E, F)$ be a disk triangulation of a disk $U$, and let $\Theta: E \rightarrow [0, \pi)$ be a function defined on $E$ that satisfies (Z4). Let $P$ and $P^*$ be disk patterns in $\mathbb{U}$ realizing $(G, \Theta)$. Assume that both $P$ and $P^*$ are locally finite in $\mathbb{U}$. Then there exists a hyperbolic isometry $f: \mathbb{U} \rightarrow \mathbb{U}$ such that $P^* = f(P)$.
	\end{thm}
	
	\begin{proof}
		Let $P$ and $P^*$ be as given. Choose an arbitrary vertex $v_0$ of $G$. We will show that 
		\[
		\rho_{\text{hyp}}\left(D(v_0)\right) = \rho_{\text{hyp}}\left(D^*(v_0)\right),
		\]
		which implies the theorem.
		
		Assume, for contradiction, that $\rho_{\text{hyp}}\left(D(v_0)\right) < \rho_{\text{hyp}}\left(D^*(v_0)\right)$. Then there exists $\delta = 1 + \varepsilon > 1$ such that 
		\[
		\rho_{\text{hyp}}\left(\delta D(v_0)\right) < \rho_{\text{hyp}}\left(D^*(v_0)\right).
		\]
		
		Consider the scaled pattern $\delta P = \{ \delta C_{v} \mid C_{v} \in P \}$. Let $G_1$ be the subgraph of $G$ corresponding to the subpattern $\pac_1 \subset \delta \pac$ consisting of disks that intersect $\partial \mathbb{U}$, and let $P_1^* \subset P^*$ be the corresponding subpattern of $P^*$.
		
		Since $P$ is locally finite in $\mathbb{U}$, $G_1$ is finite. For each boundary vertex $v$ of $G_1$, the disk $D_1(v)$ intersects the boundary of $\mathbb{U}$, so 
		\[
		\rho_{\text{hyp}}\left(D_1(v)\right) > \rho_{\text{hyp}}\left(C^*_{v}\right).
		\]
		
		Applying Lemma~\ref{MaximumPrincipleHyper}, we conclude that
		\[
		\rho_{\text{hyp}}\left(D_1(v_0)\right) \geq \rho_{\text{hyp}}\left(D_1^*(v_0)\right),
		\]
		which contradicts the fact that $D_1(v_0) = \delta D(v_0)$ and $D_1^*(v_0) = D^*(v_0)$. Therefore, the assumption must be false, and we conclude that 
		\[
		\rho_{\text{hyp}}\left(D(v_0)\right) = \rho_{\text{hyp}}\left(D^*(v_0)\right).
		\]
		
		Since $v_0$ was arbitrary, the result follows from the rigidity of hyperbolic isometries.
	\end{proof}

	\subsection{The rigidity of parabolic RCPs}
	Denote the 1-skeleton of $\mathcal{T}$ by $G$.
	We  establish the following maximum principle for generalized hyperbolic circle patterns:
	\begin{lem}\label{maximum_principle_euc}(Maximum Principle in Euclidean plane)
		Let $\mathcal{T}=(V,E,F)$ be a disk triangulation, and $\Theta\in [0 ,\pi)^E$ be the prescribed intersection angle. If (Z4) holds with two RCPs $\mathcal{P}= \{C_i\}_{i\in V}$ and $\mathcal{P^*}= \{C^*_i\}_{i\in V}$  realizing $(\mathcal{T}, \Theta)$. Then the maximum (or minimum) of $r^*_i/ r_i$ is attained at a boundary vertex.
	\end{lem}
	\begin{proof}
		If the maximum is attained at an interior vertex $v_0$. Let $v_1, v_2, \ldots, v_n, v_{n+1}=v_1$ be the vertices connected to $v_0$, whose subscripts are labeled clockwise. 
		
		Denote $\vartheta_{0,j}$ be the angle $\angle v_f v_0 v_j$ to $\mathcal{P}$, and $\vartheta^*_{0,j}$ be the angle $\angle v_f v_0 v_j$ to $\mathcal{P^*}$. Since $\mathcal{P}$ and $\mathcal{P^*}$ realize $(\mathcal{T}, \Theta)$, we have
		\begin{equation}\label{max_proof_1}
			\sum_{j=1}^n 2\vartheta_{0,j}= 2\pi,
		\end{equation}
		and
		\begin{equation}\label{max_proof_2}
			\sum_{j=1}^n 2 \vartheta^*_{0,j}= 2\pi.
		\end{equation}
		Since $r^*_i/ r_i$ attained its maximum at $v_0$, from Lemma \ref{vari}, we know
		\begin{equation}\label{max_proof_3}
			\vartheta^*_{0,j}\leq  \vartheta_{0,j}, \quad j = 1, 2,\cdots, n.
		\end{equation}
		It follows from \eqref{max_proof_1} and \eqref{max_proof_2} that all the 
        inequalities involved must in fact be equalities. By Lemma~\ref{vari}, 
        equality can occur only if
        \[
        \frac{r_j^*}{r_j}=\frac{r_0^*}{r_0}
        \]
        for every vertex \(v_j\sim v_0\). Since \(G\) is connected, this equality 
        propagates along paths in \(G\), and hence
        \[
        \frac{r_i^*}{r_i}=\frac{r_0^*}{r_0}
        \]
        for every \(v_i\in V\).
	\end{proof}

Before proving the rigidity, let us first introduce a combinatorial lemma for counting circles, which is essential in our proof.
\begin{lem}[Circle counting lemma]\label{circle_counting}
    Let $D(0,\rho)$ be the ball with radius $\rho$ whose center is the origin. Then for each $k>0$, there exists a constant $C=C(k)$ such that for any RCP $\pac=\{D(v)\}_{v\in V}$ with (Z2) holds, there are at most $C(k)$ disks in $\pac$ with radius greater than $k\rho$ which intersect $D(0,\rho)$.

\end{lem}
\begin{proof}
    Let     \[
    I(k):=\sup_{\pac=\{D(v)\}_{v\in V}}\#\{v:D(v)\cap D(0,\rho)\neq\emptyset, ~r_v\ge k\rho\}.
    \]
    It is easy to see that $I(k)$ is independent of $\rho$ since it is a scaling invariant, so it is well-defined. Without loss of generality, we assume that $\rho=1$. It is easy to see that $I(k)$ is non-increasing with respect to $k$. Let
    \[
    I(\infty)=\lim_{k\rightarrow\infty}I(s).
    \]
    We first show that $I(\infty)<\infty$. If not, we have $I(s)=\infty,~\forall s>0$. Then there exists a sequence of RCPs $$\pac_i=\{D^i(v)\}_{v\in V_i},~i\in\mathbb{Z}_+$$ such that there exist (at least) $9$ disks $\{D^i(v_j^i)\}_{j=1}^{9}$ whose radii are all greater than $4^i$ such that 
    \[
    D^i(v_j^i)\cap D(0,1)\neq\emptyset,~\forall~i\in\mathbb{Z}_+,~\forall j=1,\cdots, 9.
    \]
    Now consider the packing $2^{-i}\pac_i$. It is easy to see that, by subtracting a subsequence, those selected disks will converge to $9$ half spaces whose boundaries pass through the origin. 
    However, by the property of RCP, every point in the plane is contained in at most $3$ half-planes in the limit figure. By the pigeonhole principle, however, there must be a point that is contained in at least four half-planes, which leads to a contradiction.

    Therefore, there exists $s>0$ such that $I(s)<\infty$. For $0<k<s$, we see that for every disk with radius in $[k,s]$, which intersects $D(0,1)$, is contained in the ball $D(0,1+2s)$. Then for a fixed RCP, the number of such disks is less than or equal to 
    $
    {3(1+2s)^2}/{s^2}.$

    Therefore, 
    \[
    I(k)\le I(s)+\frac{3(1+2s)^2}{s^2}<\infty.
    \]
    This is because for every RCP satisfying (Z2), every point is contained in at most three circles.
    
\end{proof}

	For any circle \( C \subset \mathbb{C} \), we define its \textbf{normalized radius} (or distortion) by
	\[
	\tau(C) = \frac{r(C)}{d_{\mathbb{R}^2}(0, C)},
	\]
	where \( r(C) \) denotes the radius of the circle \( C \), and \( d_{\mathbb{R}^2}(0, C) \) is the Euclidean distance from the origin to \( C \).
	
	Given an infinite circle pattern \( \mathcal{P} = \{ C_i \}_{i \in V} \), we define its \textbf{distortion coefficient} by
	\[
	\tau(\mathcal{P}) = \limsup_{i \rightarrow +\infty} \tau(C_{v_i}),
	\]
	where \( \{v_i\}_{i \in \mathbb{N}} \) is an arbitrary enumeration of the vertex set \( V \). It is clear that \( \tau(\mathcal{P}) \) is independent of the choice of enumeration, and moreover, it is invariant under Euclidean similarities \( f: \mathbb{C} \to \mathbb{C} \), that is,
	\[
	\tau(f(\mathcal{P})) = \tau(\mathcal{P}).
	\]
	
	The quantity \( \tau(\mathcal{P}) \) measures how uniformly the radii of the circles are controlled relative to their Euclidean distance from the origin.

	\begin{lem}\label{rigidity_bound_tau_finite}
		Let $\mathcal{P}$ and $\mathcal{P}^*$ be as in Lemma~\ref{uniform_bound_infinite}. If $\tau(\mathcal{P}) < +\infty$, then there exists a uniform constant $C \geq 1$ such that for any vertex $v$,
		\[
		\frac{1}{C} \leq \frac{r^*_{v}}{r_{v}}  \leq C.
		\]
	\end{lem}
	
	\begin{proof}
		Without loss of generality, we assume $D^*(v_0) = D(v_0)$ and the center $o_0$ of $C_0$ lies at the origin.
		
		Let $v_1$ be an arbitrary vertex different from $v_0$, and let $V_0 \subset V$ be a finite set of vertices containing both $v_0$ and $v_1$. Since $\tau(\mathcal{P}) < +\infty$, we can choose $\delta > 3$ such that
		\begin{equation} \label{taufiniteP}
			\tau(C_i) = \frac{r_i}{d_{\mathbb{R}^2}(0, C_i)} \leq \frac{\delta}{3}, \quad \forall i \in V \setminus V_0.
		\end{equation}
		
		Let $\tilde{r} > 0$ be such that $\mathcal{P}(V_0) = \bigcup_{i \in V_0} D_i \subset D(\tilde{r})$. Then for any $r > \tilde{r}$, there is no disk $D_i$ in $\mathcal{P}$ intersecting both $C(r)$ and $C(\delta r)$. Otherwise, for some $i \in V_0$, which contradicts $\mathcal{P}(V_0) = \bigcup_{i \in V_0} D_i \subset D(\tilde{r})$; for some $i \in V \setminus V_0$, we would have $d_{\mathbb{R}^2}(0, C_i) \leq r$ and $\delta r \leq d_{\mathbb{R}^2}(0, C_i) + 2r_i$, which contradicts $\delta > 3$. Hence, for any $r > \tilde{r}$,
		\begin{equation} \label{taufiniteP1}
			V_{C(r)}(\mathcal{P}) \cap V_{C(\delta r)}(\mathcal{P}) = \emptyset.
		\end{equation}
		
		Let $\tilde{r}_n := \delta^n \tilde{r}$ and $V_n := V_{C(\tilde{r}_n)}(\mathcal{P})$ for $n \geq 1$. From \eqref{taufiniteP1}, the sets $\{V_n\}_{n \in \mathbb{N}}$ are mutually disjoint and satisfy the property that $V_{n_2}$ separates $V_{n_1}$ from $V_{n_3}$ for $0 \leq n_1 < n_2 < n_3$.
		By Proposition \eqref{addlemma} and Lemma~\ref{es}, we have
		\begin{align*}
			\mathrm{VEL}(V_2, V_{2k-1}) &\geq \sum_{n=1}^{k-1} \mathrm{VEL}(V_{2n}, V_{2n+1}) \\
			&\geq \frac{k-1}{4(24+36\pi^2)}.
		\end{align*}
		Let $k$ be an integer larger than
		\[
		144\pi^2(24 + 36\pi^2)^2 + 1.
		\]
		Then,
		\[
		\mathrm{VEL}(V_2, V_{2k-1}) \geq 36\pi^2(24 + 36\pi^2).
		\]

        Let $\hat{r}=\mathrm{diam}(\pac^*(V_2))$.
		Similar to the proof in Lemma~\ref{annuilemma}, we can prove $\forall\rho \in [\hat{r}, 2\hat{r}]$, the vertex set $V_{C(0, \rho)}(\mathcal{P}^*)$ separates $V_2$ from $V_{2k-1}$, and hence separates $V_1$ from $V_{2k}$. By the definition of $\hat{r}$, it is clear that
		\begin{align}
			P^*(V_1) &\subset D(0, \hat{r}), \label{taufiniteP3} 
            \end{align}

		
		Let $G_1$ be the subgraph of $G$ consisting of vertices $V_{D(0, \tilde{r}_1)}(\mathcal{P})$. Let $W$ be the connected component of $V\backslash V_{C(0,2\hat{r})}$ that containing $V_2$. We denote by $W'$ the set of vertices that belong to or are connected with $W$. It is easy to see that $\partial W'\subset  V_{C(0,2\hat{r})}$. Let $G_2$ be the induced subgraph of $G$. It is clear that $\partial W'$ separates $V_2$ and $V_{2k-1}$. Therefore, by the definition of $\partial W'$, we see that $D(v)\cap D(0,\tilde r_{2k-1})\neq\emptyset,~\forall v\in \partial W'$. Therefore, we have $\partial P(W')\subset D(0,\tilde{r}_{2k})$.

        Since $V_{C(0,2\hat{r})}$ separates $V_2$ and $V_{2k-1}$,
        Considering the sub-pattern decided by $G_1$, by Lemma~\ref{MaximumPrincipleHyper}, for each $v \in V_{D(0, \tilde{r}_1)}(\mathcal{P})$,
		\begin{equation} \label{taufiniteP5}
			\rho_{\mathrm{hyp}}\left(\frac{1}{2\hat{r}} C^*_{v}\right) \leq \rho_{\mathrm{hyp}}\left(\frac{1}{\tilde{r}_1} C_{v}\right).
		\end{equation}
		Similarly, considering the sub-pattern decided by $G_2$, for $v \in W'$,
		\begin{equation} \label{taufiniteP6}
			\rho_{\mathrm{hyp}}\left(\frac{1}{2\hat{r}} C^*_{v}\right) \geq \rho_{\mathrm{hyp}}\left(\frac{1}{\tilde{r}_{2k}} C_{v}\right).
		\end{equation}
		
		Since the disks involved in \eqref{taufiniteP5} and \eqref{taufiniteP6} are all contained in $(1/2)D(0,1)$, the hyperbolic and Euclidean radii are uniformly comparable. Thus, there exists a universal constant $\tilde{C} > 0$ such that for $v = v_0, v_1$,
		\begin{align}
			\frac{1}{2\hat{r}} C^*_{v} &\leq \tilde{C} \cdot \frac{1}{\tilde{r}_1} C_{v}, \label{taufiniteP7} \\
			\tilde{C} \cdot \frac{1}{2\hat{r}} C^*_{v} &\geq \frac{1}{\tilde{r}_{2k}} C_{v}. \label{taufiniteP8}
		\end{align}
		Combining \eqref{taufiniteP7} and \eqref{taufiniteP8}, and using $P^*(v_0) = D(v_0)$ and $\tilde{r}_n = \delta^n \tilde{r}$, we obtain
		\[
		\frac{1}{\tilde{C}^2 \delta^{2k-1}} \leq  \frac{r^*_{v_1}}{r_{v_1}} \leq \tilde{C}^2 \delta^{2k-1}.
		\]
		
		Letting $C := \tilde{C}^2 \delta^{2k-1}$ completes the proof, as this constant does not depend on the choice of $v_1$.
	\end{proof}
	
	Given an infinite circle pattern \( \mathcal{P} = \{ C_i \}_{i \in V} \), we define $W\subset V$ be the set consisting of vertex $v$ such that $\tau(C_v) \geq 1$. 
	
	\begin{lem}\label{tau_P_bound}
		Let $\mathcal{P}$ and $\mathcal{P}^*$ be as in Lemma~\ref{uniform_bound_infinite}. If $\tau(\mathcal{P}) = +\infty$, then
		$$\limsup_{i \rightarrow +\infty}\tau(C^*_{v_i})>0,$$
		where \( \{v_i\}_{i \in \mathbb{N}} \) is an arbitrary enumeration of the vertex set \( W \).
	\end{lem}
	\begin{proof}
		Without loss of generality, we assume $D^*(v_0) = D(v_0)$ and the center $c_0$ of $D(v_0)$ is located at the origin.
		
		We prove this lemma by contradiction. If $\limsup_{i \rightarrow +\infty}\tau(C^*_{v_i})=0$, then for any $0<\epsilon < 1/2$, there is a finite vertex set $V_0\subset V$ such that $v_0\in V_0$ and 
		\begin{equation} \label{tauinfinite_P}
			\tau(C_i^*) = \frac{r_i}{d_{\mathbb{R}^2}(0, C_i^*)} < \epsilon \quad \forall i \in W \setminus V_0.
		\end{equation}

		Let $W^{\prime}=W-V_0$, and let $\tilde{G}$ be the subgraph of $G$ obtained by removing the vertices in $W^{\prime}$ as well as the edges with an end in $W^{\prime}$.
		Let $\tilde{r} > 0$ be such that $\mathcal{P}(V_0) = \bigcup_{i \in V_0} D_i \subset D(0, \tilde{r})$. Let $\tilde{r}_n := 4^n \tilde{r}$ and $V_n := V_{C(0, \tilde{r}_n)}(\mathcal{P})$ for $n \geq 1$. For $i< j$, we have 
		\begin{equation} \label{tauinfinite_P1}
			V_i \cap V_j \subset W^{\prime}=W-V_0
		\end{equation}
		and for any $0 \leq n_1<n_2<n_3$, $V_{n_2}$ separates $V_{n_1}$ from $V_{n_3}$ in $G$. Here \eqref{tauinfinite_P1} holds due to the fact that for any $v\in V_i\cap V_j$ we have 
        \[
        d_{\mathbb{R}^2}(0,C^*_v)\le 4^i\tilde r;~ d_{\mathbb{R}^2}(0,C^*_v)+2r_v\ge 4^j\tilde r,
        \]
        which implies
        \[
        r_v\ge \frac{1}{2}(4^{j-i}-4^{-1})d_{\mathbb{R}^2}(0,C_v)>d_{\mathbb{R}^2}(0,C_v).
        \]
        	The following proof is divided into two steps.
		
		\medskip
		\emph{Step 1. We first prove that  there exists a  $\hat{r} > 0$ such that for any $\rho \in [\hat{r}, 2\hat{r}]$, the vertex set $V_{C(0, \rho)}(\mathcal{P}^*)$ separates $V_2$ from $V_{2k-1}$ in $G$.}
        \medskip
		Let $k$ be an integer such that
		\begin{align}\label{k_assumption}
		k\ge 72(24 + 36\pi^2)^2 + 1.
		\end{align}
		Let $U_{2,2 k-1}$ be the set of vertices in $G$ which is separated from $V_0$ by $V_2$, and separated from $\infty$ by $V_{2 k-1}$. In fact, let $A(0,a, b)= \{z\in \mathbb C: a\leq |z|\leq b \}$, then 
		\begin{equation}\label{tauinfinite_P3}
			U_{2,2 k-1}= \{v\in V : D_v \cap A(0, \tilde r_2, \tilde r_{2k-1})\neq\emptyset\}.
		\end{equation}
		Let $W^{\prime \prime}=U_{2,2 k-1} \cap W$. 
        
        Since for every $W''$, we have $$3r_v>d_{\mathbb{R}^2}(0,C_v)+2r_v\ge \tilde r_2=16\tilde r,$$
by Lemma \ref{circle_counting}, the cardinality of $W^{\prime \prime}$ is bounded by a universal constant $\tilde C_1$.
		
		Now, we use the method similar to the proof of Lemma \ref{annuilemma}.
		
		Let
		\begin{equation}\label{tauinfinite_P4}
		\tilde r^* = \min \{r: D(0, r) \cap D_i \neq \emptyset, \forall i \in V_2 \}.
		\end{equation}
		
		Let $\gamma^* \in \Gamma^*(V_2, V_{2k-1})$.  Thus, by the minimality of $\tilde r$, we have 
		\begin{equation}\label{tauinfinite_P5}
			\mathrm{diam} (\mathcal P^*(\gamma^*))\geq \tilde r^*.
		\end{equation}
		
		By contradiction, we assume there is a $\rho \in [\tilde r^*, 2\tilde r^*]$ such that $V_{C(0, \rho)}(\mathcal{P}^*)$ does not  separate $V_2$ from $V_{2k-1}$ in $G$. Thus, there is a $\gamma_0 \in \Gamma(V_2, V_{2k-1})$ such that
		\begin{equation}
			\gamma_0 \cap V_{C(0, \rho)}(\mathcal{P}^*) = \emptyset.
		\end{equation}
		By the definition of $V_{C(0, \rho)}(\mathcal{P}^*)$, we have
		\begin{equation}
			D_i^* \cap C(0, \rho)   = \emptyset, \quad \forall i \in \gamma_0.
		\end{equation}
		Hence 
		\begin{equation}\label{tauinfinite_P6}
			\mathcal{P}^*(\gamma_0) \cap C(0, \rho)     = \emptyset.
		\end{equation}
		Since $\gamma_0 \in \Gamma(V_2, V_{2k-1})$, we know $\gamma_0$ contains a vertex in $V_2$. By the definition of $	\tilde r^*$ and $\mathcal{P}^*(\gamma_0)$, we know
		\begin{equation}\label{tauinfinite_P7}
			\mathcal{P}^*(\gamma_0) \cap C(0, \tilde r^* )   \neq  \emptyset.
		\end{equation}
		Combining \eqref{tauinfinite_P6} and \eqref{tauinfinite_P7}, with $\tilde r^* \leq \rho$ we know 
		\begin{equation}
			\mathcal{P}^*(\gamma_0) \subset \mathbb D(0, \rho).
		\end{equation}
		Hence, since $\gamma^* \in \Gamma^*(V_2, V_{2k-1})$ and $\gamma_0 \in \Gamma (V_2, V_{2k-1})$, we have
		\begin{equation}\label{tauinfinite_P8}
			\mathcal{P}^*(\gamma^*) \cap D(0, \rho) \supset	\mathcal{P}^*(\gamma^*) \cap 	\mathcal{P}^*(\gamma_0) \supset \mathcal{P}^*(\gamma^* \cap \gamma_0) \neq \emptyset.
		\end{equation}
		
		Define $m(i)={ \mathrm{diam} (D_i^* \cap D(0, 3\tilde r^*))}/{\tilde r^*}$, which is a vertex metric in the cellular decomposition $\mathcal{T}$. 
		If $\gamma^*$ is connected, \eqref{tauinfinite_P8} implies two cases:
        
		\medskip
        \emph{Case 1: $\mathcal{P}^*(\gamma^*) \subset D(0, 3\tilde r^*)$}. From \eqref{tauinfinite_P5}, we have $$\int_{\gamma^*} d m=\sum_{i \in \gamma^*} \mathrm{diam} (D_i^* \cap D(0, 3\tilde r^*)) /{\tilde r^*}=\sum_{i \in \gamma^*} \mathrm{diam} (D_i^*) /{\tilde r^*} \geq \mathrm{diam} (\mathcal P^*(\gamma^*))  /{\tilde r^*} \geq 1.$$
			
		\medskip
        \emph{Case 2: $\mathcal{P}^*(\gamma^*)$ is not completely contained $D(0, 3\tilde r^*)$}. From \eqref{tauinfinite_P8}, we have $\mathcal{P}^*(\gamma^*)$ joining $D(0, \rho)$ and $D(0, 3\tilde r^*))$. Then we have
			$$\int_{\gamma^*} d m=\sum_{i \in \gamma^*} \mathrm{diam} (D_i^* \cap D(0, 3\tilde r^*)) /{\tilde r^*}=(3\tilde r^*- \rho) /{\tilde r^*} \geq 1.$$
		Since every $\gamma^*\in \Gamma^*(V_2, V_{2k-1})$ contains a connected sub-curve in $\Gamma^*(V_2, V_{2k-1})$, it follows that $m$ is $\Gamma^*(V_2, V_{2k-1})$-admissible.
		
		From \eqref{tauinfinite_P}, for any $v\in W^{\prime} \supset W^{\prime \prime}$, we have
		$$
		m (v)\leq 6\epsilon.
		$$
		Since the cardinality of $W^{\prime \prime}$ is at most $\tilde C_1$, we see that
		$$
		\sum_{v \in W^{\prime \prime}} m(v) \leq 6 C_1 \varepsilon.
		$$
      Let $\epsilon$ be small enough such that 
      \begin{align}\label{choose_epsilon}
          2\tilde C_1 \arcsin\frac{\epsilon}{1+\epsilon}<2\pi.
      \end{align}
          
       Then vertices in $W'$ do not separate $V_2- W'$ and $V_{2k-1}- W'$.
Moreover, by \eqref{tauinfinite_P1} and \eqref{choose_epsilon} we see that $V_1-W',\cdots,V_{2k}-W'$ are disjoint non-empty vertex sets in $\tilde G$.

		Now let $\beta^* \in \Gamma_{\tilde{G}}^*\left(V_2 \cap\left(V-W^{\prime}\right), V_{2 k-1} \cap\left(V-W^{\prime}\right)\right)$. Then, the union $\gamma^*=\beta^* \cup W^{\prime \prime}$ is a vertex curve in $\Gamma_G^*\left(V_2, V_{2 k-1}\right)$. 
		$$
		\sum_{v \in \beta^*} m(v) =  \sum_{v \in \gamma^*} m(v) -\sum_{v \in W^{\prime \prime}} m(v)  \geq 1-6 C_1 \varepsilon
		$$
		
		We choose $\varepsilon>0$ so small that $1-6 C_1 \varepsilon \geq 1 / \sqrt{2}$. Then $\sqrt{2} \eta$ is $\Gamma_{\tilde{G}}^*(V_2 \cap(V-W^{\prime},$ $ V_{2 k-1} \cap\left(V-W^{\prime}\right))$-admissible. Thus, 
		$$
		\begin{aligned}
			\operatorname{VEL}_{\tilde{G}}\left(V_2, V_{2 k-1}\right) & =\operatorname{MOD}\left(\Gamma_{\tilde{G}}^*\left(V_2 \cap\left(V-W^{\prime}\right), V_{2 k-1} \cap\left(V-W^{\prime}\right)\right)\right) \\
			& \leq \operatorname{area}(\sqrt{2} m)
		\end{aligned}
		$$
		
		By a similar argument as in the proof of Lemma \ref{es}, we see that
		$$
		\mathrm{area}(\sqrt{2} m)\le 18(24 + 36\pi^2).
		$$
		This contradicts \eqref{tauinfinite_P2}.

		\medskip

        Let $\operatorname{VEL}_{\tilde{G}}\left(V_i, V_j\right)$ denote the vertex extremal length between $V_i \cap\left(V-W^{\prime}\right)$ and $V_j \cap\left(V-W^{\prime}\right)$ in the subgraph $\tilde{G}$.
		By Proposition \eqref{addlemma} and Lemma~\ref{es}, we have
		\begin{align*}
			\mathrm{VEL}_{\tilde{G}}(V_2, V_{2k-1}) &\geq \sum_{n=1}^{k-1} \mathrm{VEL}_{\tilde{G}}(V_{2n}, V_{2n+1}) \\
			&\geq \frac{k-1}{4(24+36\pi^2)}.
		\end{align*}
		Since $k\ge
		72(24 + 36\pi^2)^2 + 1.$
		Then,
		\begin{equation}\label{tauinfinite_P2}
			\mathrm{VEL}_{\tilde{G}}(V_2, V_{2k-1}) > 18(24 + 36\pi^2).
		\end{equation}

		\emph{Step 2. we finish the proof.}
		\medskip
		
		We have proved that there exists a $\hat{r} > 0$ such that for any $\rho \in [\hat{r}, 2\hat{r}]$, the vertex set $V_{C(0, \rho)}(\mathcal{P}^*)$ separates $V_2$ from $V_{2k-1}$ in $G$. Then we have 
		\begin{equation}
			V_2 \cap V_{2 k-1} \subset V_{C(0, \hat{r})}(\mathcal{P}^*)\cap V_{C(0,2  \hat{r})}(\mathcal{P}^*).
		\end{equation}
		From  \eqref{tauinfinite_P} and \eqref{tauinfinite_P1}, we have
		\begin{equation}
			V_2 \cap V_{2 k-1} \subset  V_{C(0, \hat{r})}(\mathcal{P}^*)\cap V_{C(0,2  \hat{r})}(\mathcal{P}^*) \cap (W-V_0) = \emptyset,
		\end{equation}
		which implies
		$$
		V_{C(0, 4^2\tilde{r})}(\mathcal{P})\cap V_{C(0, 4^{2k-2}\tilde{r})}(\mathcal{P}) =\emptyset.
		$$
		Since $\tilde{r}$ is arbitrarily large, this contradicts $\tau(\mathcal{P}) = +\infty$.
		
	\end{proof}
	
	\begin{lem}\label{rigidity_bound_tau_infinite}
		Let $\mathcal{P}$ and $\mathcal{P}^*$ be as in Lemma~\ref{uniform_bound_infinite}. If $\tau(\mathcal{P}) = +\infty$, then there exists a uniform constant $C \geq 1$ such that for any vertex $v$,
		\[
		\frac{1}{C} \leq  \frac{r^*_{v}}{r_{v}}  \leq C.
		\]
	\end{lem}
	
	\begin{proof}
		Let $v_0$ and $v_1$ be two fixed vertices.
		Without loss of generality, we assume $D^*(v_0) = D(v_0)$ and the center $c_0$ of $C_0$ lies at the origin. 
		
		As shown in Lemma~\ref{tau_P_bound}, let $\tau^* := \limsup_{i \to +\infty} \tau(C^*_{v_i})$. Fix $\delta \in (0, 1/3)$ such that $\delta < \tau^*$. Then there exists a sequence $\{u_k\}$ of mutually distinct vertices such that $\tau\left(C_{u_k}\right) \geq 1$ and $\tau\left(C^*_{u_k}\right) \geq 3\delta$. Since both $P$ and $P^*$ are locally finite in $\mathbb{C}$, we have
		\[
		d_{\mathbb{R}^2}\left(0, C_{u_k}\right) \to \infty \quad \text{and} \quad d_{\mathbb{R}^2}\left(0, C^*_{u_k}\right) \to \infty.
		\]
		
		Let $\rho_0 > 0$ be such that
		\[
		D_{v_0} \cup D_{v_1} \cup D^*_{v_0} \cup D^*_{v_1} \subseteq D(0,\rho_0).
		\]
		Choose one $\tilde{u}$ from the sequence $\{u_k\}$ such that
		\[
		d := d_{\mathbb{R}^2}(0, P(\tilde{u})) \geq \frac{100 \rho_0}{\delta^2}, \quad
		d^* := d_{\mathbb{R}^2}(0, P^*(\tilde{u})) \geq \frac{100 \rho_0}{\delta^2}.
		\]
		Define 
		\begin{align*}
			F(z)&=\frac{r(C_{\tilde u})z}{(r(C_{\tilde u})+d)z-(r(C_{\tilde u})+d)^2+r(C_{\tilde u})^2},
			\\F^*(z)&=\frac{r(C^*_{\tilde u})z}{(r(C^*_{\tilde u})+d^*)z-(r(C^*_{\tilde u})+d^*)^2+r(C^*_{\tilde u})^2}.
		\end{align*}
		
		Then $F$ is a M\"obius transformation with $F(0) = 0$ and \( F(\hat{\mathbb{C}} \setminus D(\tilde{u})) = D(0,1) \), and similarly $F^*$ is a M\"obius transformation with $F^*(0) = 0$ and \( F^*(\hat{\mathbb{C}} \setminus D^*(\tilde{u})) = D(0,1) \). 
		
		Observe that the hyperbolic distance between $F(\infty)$ and $0$ is greater than the hyperbolic distance between $0$ and
		\[
		\frac{r(C_{\tilde u})}{d + \rho_0 + r(C_{\tilde u})}.
		\]
		It follows that
		\[
		|F(\infty)| = \frac{r(C_{\tilde u})}{d + r(C_{\tilde u})} = \frac{\tau(C_{\tilde u})}{1 + \tau(C_{\tilde u})} > \frac{1}{3} > \delta.
		\]
		Similarly,
		\[
		|F^*(\infty)| = \frac{\tau(C^*_u)}{1 + \tau(C^*_u)} \geq \frac{3\delta}{1 + 3\delta}\ge\frac{3\delta}{2} > \delta.
		\]
		
		We write $r$ and $r^*$ for $r(C_{\tilde u})$ and $r(C_{\tilde u}^*)$, from \( \rho_0 \leq \frac{\delta^2 d}{100} \) and \( \rho_0 \leq \frac{\delta^2 d^*}{100} \), for $z\in \mathbb{C}$ with $|z|\le\rho_0$, we have
		\begin{align*}
			|F(z)|&<\frac{\frac{r d\delta^2}{100}}{|2r d+d^2-(r+d)z|}\\&\le\frac{\frac{\rho d\delta^2}{100}}{2r d+d^2-(r+d)|z|}\\
			&\le\frac{\frac{r d\delta^2}{100}}{2r d+d^2-(r+d)\frac{ d\delta^2}{100}}\\
			&\le \frac{\delta}{2}.
		\end{align*}
		Similarly, $|F^*(z)|\le\frac{\delta}{2}$. Therefore,
		\[
		F(D(\rho_0)), \quad F^*(D(\rho_0)), \quad F(D(v_0)), \quad F(D(v_1)), \quad F^*(D^*(v_0)), \quad F^*(D^*(v_1))
		\]
		are all contained in the disk \( (\delta/2) \cdot \mathbb{U} \), where $\mathbb{U}$ denotes the unit disk.
		By direct computation, we have 
		\begin{align}\label{est_deri1}
			F'(z)=\frac{r }{(r+d)z-d(2r+d)}-\frac{r(r+d)z}{[(r+d)z-d(2r+d)]^2}.
		\end{align}
		Since \( d=d_{\mathbb{R}^2}(0, C_{\tilde u}) \geq 100 \rho_0 \), for the first term of \eqref{est_deri1} satisfies
		\begin{align}\label{est_deri2}
			\frac{100r}{101d(2r+d)}\le\left|\frac{r }{(r+d)z-d(2r+d)}\right|\le\frac{100r}{99d(2r+d)},~\forall z~\text{s.t}~|z|<\rho_0.
		\end{align}
		For the second term, we also have 
		\begin{align}\label{est_deri3}
			\left|\frac{r(r+d)z}{[(r+d)z-d(2r+d)]^2}\right|\le\frac{r d(r+d)}{(99d(2r+d))^2},~\forall z~\text{s.t}~|z|<\rho_0.
		\end{align}
		Combining \eqref{est_deri2} and \eqref{est_deri3}, we have
		\[
		\frac{r d(2r+d)}{(d(2r+d))^2}(\frac{100}{101}-\frac{1}{99^2})\le|F'(z)|\le \frac{r d(2r+d)}{(d(2r+d))^2}(\frac{100}{101}+\frac{1}{99^2}),~\forall z~\text{s.t}~|z|<\rho_0.
		\]
		Therefore, $\forall z,z'\in D(0,\rho_0),$ we have
		\[
		\frac{1}{2}\le\frac{|F'(z)|}{|F'(z')|}\le2.
		\]
		Thus,
		\[
		\frac{r(F(D(v_1)))r(D(v_1))}{r(D(v_1))r(F(D(v_0)))} \in [1/4, 4].
		\]
		Similarly, we have
		\[
		\frac{r(F^*(D^*(v_1)))r(D^*(v_1))}{r(D^*(v_1))r(F^*(D^*(v_0)))} \in [1/4, 4].
		\]
		
		Since \( D(v_0) = D^*(v_0) \), in order to compare \( r(D^*(v_1)) / r(D(v_1)) \), it suffices to compare
		\[
		\frac{r(F(D(v_1)))}{r(F(D(v_0)))} \quad \text{and} \quad \frac{r(F^*(P^*(v_1)))}{r(F^*(P^*(v_0)))}.
		\]
		
		Now consider the comparison of the pattern \( (1/\delta) F(P) \) with \( F^*(P^*) \). Since \( |F(\infty)| > \delta \) and \( F(P) \) are locally finite in \( \hat{\mathbb{C}} \setminus \{F(\infty)\} \), only finitely many disks intersect the unit disk $U$. Then, using Lemma~\ref{MaximumPrincipleHyper}, we get, for \( j = 0, 1 \),
		\[
		r_{\mathrm{hyp}}\left((1/\delta) F(D(v_j))\right) \geq r_{\mathrm{hyp}}\left(F^*(D^*(v_j))\right).
		\]
		Conversely, comparing \( F(P) \) with \( (1/\delta) F^*(P^*) \), and noting that \( |F^*(\infty)| > \delta \), we obtain
		\[
		r_{\mathrm{hyp}}\left(F(D(v_j))\right) \leq r_{\mathrm{hyp}}\left((1/\delta) F^*(D^*(v_j))\right).
		\]
		The lemma then follows, since all involved disks are contained in \( (1/2) \mathbb{U} \).
	\end{proof}

	\begin{lem}\label{uniform_bound_infinite}
		Let $\mathcal{T}=(V,E,F)$ be a disk triangulation, and $\Theta\in [0 ,\pi)^E$ be the prescribed intersection angle. Assume that condition (Z2) and (Z4) hold. Let $\mathcal{P}= \{C_i\}_{i\in V}$ and $\mathcal{P^*}= \{C^*_i\}_{i\in V}$ be two RCP that realize $(\mathcal{T}, \Theta)$. There is a uniform constant $C \geq 1$ such that for any vertex $v$,
		$$
		\frac{1}{C} \leq \frac{r^*_{v}}{r_{v}} \leq C.
		$$
	\end{lem}
	\begin{proof}
		Combining Lemmas~\ref{rigidity_bound_tau_finite} and~\ref{rigidity_bound_tau_infinite}, we obtain the desired result.
	\end{proof}
	
	Finally, we obtain the rigidity theorem for parabolic RCPs.
	\begin{thm}\label{rigidity_parabolic}
		Let $\mathcal{T}=(V,E,F)$ be a disk triangulation and let $\Theta\in [0 ,\pi)^E$ with $\sup_{e\in E}\Theta(e)<\pi$ be an intersection angle. Assume that condition (Z2) and (Z4) hold. Let $\mathcal{P}= \{C_i\}_{i\in V}$ and $\mathcal{P^*}= \{C^*_i\}_{i\in V}$ be two RCP that realize $(\mathcal{T}, \Theta)$. Suppose that $P$ is locally finite in the plane. Then there is a Euclidean similarity $f: \mathbb{C} \rightarrow \mathbb{C}$ such that $P^*=f(P)$.
	\end{thm}
	\begin{proof}
		
		Denote $\mathcal{P}(0)=\mathcal{P}$ and $\mathcal{P}(1)=\mathcal{P^*}$. For $i \in V$, denote $r_i(0)$ and $r_i(1)$ by  the radius of $C_i$ and $C^*_i$, respectively. Then we construct a $r(t)$ connecting $r(0)$ and $r(1)$. Let $u_i (t)= \ln r_i(t)$. Set $\tilde u_i (t)= u_i (0) + (u_i(1)- u_i(0))t$. Lemma \ref{uniform_bound_infinite} tells us that $|u_i(1)- u_i(0)|\leq \ln C$, $\forall i \in V$. Hence, we have 
		\begin{equation}\label{harmonic_proof_1}
			|\tilde u_i (t+h)- \tilde u_i (t)| \leq |u_i(1)- u_i(0)||h| \leq \ln C \cdot |h|, \quad \forall i \in V.
		\end{equation}
		
        Since $\mathcal T=(V,E,F)$ is a locally finite triangulation of the open disk,
        we choose an increasing exhaustion by finite triangulated topological disks
        \[
        \left\{\mathcal T^{[n]}=\left(V^{[n]},E^{[n]},F^{[n]}\right)\right\}_{n=1}^{\infty}
        \]
        such that each $|\mathcal T^{[n]}|$ is homeomorphic to a closed disk,
        \[
        |\mathcal T^{[n]}|\subset \operatorname{int} |\mathcal T^{[n+1]}|,
        \qquad
        \bigcup_{n=1}^{\infty}|\mathcal T^{[n]}|=|\mathcal T|,
        \]
        and every compact subset of $|\mathcal T|$ is contained in $|\mathcal T^{[n]}|$
        for all sufficiently large $n$.
        Here $|\mathcal T^{[n]}|$ denotes the underlying polyhedron of the finite
        subcomplex $\mathcal T^{[n]}$.
        Let $\partial V^{[n]}$ be all boundary vertices in $\mathcal{T}^{[n]}$, and let $\mathrm{int} V^{[n]}$ be all the interior vertices in $\mathcal{T}^{[n]}$. For any $t \in [0,1]$, from Theorem \ref{exist_boundary_radius}, we know there is a unique $\mathcal{P}^{[n]}(t)$ that realizes $(\mathcal{T}^{[n]}, \Theta)$ and its radius $r^{[n]}_i(t)= \tilde r_i(t)$ for any boundary vertex $i \in \partial V^{[n]}$. In particular, the circle pattern $\mathcal{P}^{[n]}(0)$ (resp. $\mathcal{P}^{[n]}(1)$) is exactly the $\mathcal{P}(0)$ (resp. $\mathcal{P}(1)$) restricted on $\mathcal{T}^{[n]}$. 
		
		By Lemma \ref{maximum_principle_euc}, we know $u_i^{[n]}(t+h)- u_i^{[n]}(t)$ attains its maximum and minimum at the vertex in $ \partial V^{[n]}$. Hence, by \eqref{harmonic_proof_1}, we have
		\begin{equation}\label{harmonic_proof_2}
			|u_i^{[n]}(t+h)- u_i^{[n]}(t)| \leq \ln C\cdot |h|, \quad \forall i \in V^{[n]}.
		\end{equation}
		Since $t \in [0,1]$, we have 
		$$|u_i^{[n]}(t)- u_i^{[n]}(0)| \leq  \ln C$$
		Thus, for any $t \in (0, 1)$, there is a sub-sequence $\{u^{[n_k]}(t)\}_{k}$ of $\{u^{[n]}(t)\}$ that converges as $k \rightarrow + \infty$. Let $u(t)$ be the limit of $\{u^{[n_k]}(t)\}$.  From \eqref{harmonic_proof_2}, we have
		\begin{equation}\label{harmonic_proof_3}
			|u_i(t+h)- u_i(t)| \leq \ln C\cdot |h|, \quad \forall i \in V.
		\end{equation}
		Therefore, $u_{i}(t)$ is absolutely continuous and 
		\begin{equation}
			\left|\dfrac{d u_i(t)}{dt}\right|\leq \ln C~a.e., \quad \forall i \in V.
		\end{equation}
		It is obvious that the circle pattern $\mathcal{P}(t)$ consists of circles $C_i=C_i(t)$ with radius $e^{u_i(t)}$ is an immersed circle pattern weakly realizing  $(\mathcal{T}, \Theta)$. Therefore, 
		$$
		K_i(t) =0, \quad \forall i \in V.
		$$
		So, 
		\begin{equation}\label{harmonic_proof_4}
			\begin{aligned}
				0 &= \dfrac{d K_i(u(t))}{dt}= \dfrac{\partial K_i}{\partial u_i} \dfrac{d u_i(t)}{dt}+ \sum_{j\sim i} \dfrac{\partial K_i}{\partial u_j}\dfrac{d u_j(t)}{dt}\\
				&=  \sum_{j\sim i} \dfrac{\partial K_i}{\partial u_j}\Big(\dfrac{d u_j(t)}{dt}- \dfrac{d u_i(t)}{dt}\Big)  +  \Big(\sum_{j\sim i} \dfrac{\partial K_i}{\partial u_j}+\dfrac{\partial K_i}{\partial u_i}\Big) \dfrac{d u_i(t)}{dt}\\
			\end{aligned}
		\end{equation}
		From Lemma \ref{vari}, we know 
		$$
		\sum_{j\sim i} \dfrac{\partial K_i}{\partial u_j}+\dfrac{\partial K_i}{\partial u_i}=0.
		$$
		Let $f(t)= {d u(t)}/{dt}$, by \eqref{harmonic_proof_4}, 
		$$
		\Delta_G f_i= \sum_{j \sim i} \omega_{i j}\left(f_j-f_i\right)=0
		$$
		where
		$$
		\omega_{ij}(t) = -\dfrac{\partial K_i}{\partial u_j}(t).
		$$
		From Lemma \ref{vari}, we know $\omega_{ij}(t)=\omega_{ji}(t)>0$. So, $f$ is  a harmonic function for the weighted graph $(\mathcal{T}, \omega)$.
	   Since $\mathcal{T}$ is VEL-parabolic, by Lemma \ref{partialthetabound}, and \ref{VEL_current}, we know that $(\mathcal{T}, \omega(t))$ is recurrent. Since $f$ is  a harmonic function for the weighted graph $(\mathcal{T}, \omega)$, by Lemma \ref{harmonic_current}, we know $f_i(t)= c(t)$, $\forall i \in V$. It follows that
		$$
		u_i(1)-u_i(0)= \int_0^1 c(t) dt.
		$$
		Since $ \int_0^1 c(t) dt$ is a constant, we know $r_i(1)/r_i(0)= e^{ \int_0^1 c(t) dt}$ is a constant, for any $i \in V$. This implies that $\pac^*$ and $\pac$ are images of each other by Euclidean similarities.
	\end{proof}
	
	\section{Convex trivalent hyperbolic polyhedra in $\HH^3$}\label{polyhedra_section}
	In this section, we apply the previous results on the existence and uniqueness of RCPs to characterize a class of hyperbolic polyhedra, namely infinite trivalent hyperbolic polyhedra. 
	We first give a precise definition of $3$-dimensional trivalent hyperbolic polyhedra (THP) and introduce some related concepts.   After that, we establish the correspondence between THP and RCPs. Finally, we characterize the existence and rigidity of the infinite THP.  
	
	\subsection{Basic definitions and notations}
	A \textbf{half-space} in $\mathbb{H}^3$ is the closure of one of the two regions determined by a hyperbolic plane. 
	A set of half-spaces is called \textbf{locally finite} if every compact subset of $\mathbb{H}^3$ intersects only finitely many of their boundary planes. In this section, we regard  $\mathbb{S}^2=\partial\HH^3$. The following definition of a hyperbolic polyhedron is natural and has appeared in \cite[Section 3.2]{Martelli} and \cite[Section 6.3]{Ratcliffe}. 
	
	\begin{defn}\label{def_plolyhedron}
		A \textbf{hyperbolic polyhedron} $P$ in $\mathbb{H}^3$ is defined as $P =\bigcap_{i \in I} H_i$, where $\{H_i\}_{i \in I}$ is a locally finite set of half-spaces.
	\end{defn}
	
	Clearly, every hyperbolic polyhedron is convex since it is the intersection of convex subsets in $\mathbb{H}^3$. Consider Klein's projective model for $\HH^3$, where $\HH^3$ is identified as the open unit ball in $\mathbb{R}^3\subset RP^3$, and a hyperbolic plane in $\mathbb{H}^3$ is the intersection of a hyperplane in $\mathbb{R}^3$ with $\HH^3$.
	For a convex hyperbolic polyhedron $P\subset\HH^3$, if $P=\tilde{P}\cap \HH^3$ for some locally finite convex polyhedron $\tilde{P}\subset RP^3$, we say $P$ is the hyperbolic polyhedron induced by $\tilde{P}$. Consider a vertex $v$ of $\tilde{P}$, it is called a \textbf{vertex}, \textbf{ideal vertex}, or 
	\textbf{hyperideal vertex} of $P$ if and only if each side incident to $v$ of $\tilde{P}$ intersects $\HH^3$ and $v$ is contained in $\HH^3$, $\partial\HH^3$, or $RP^3\setminus(\HH^3\cup\partial\HH^3)$ respectively.
	Let $H \subset \mathbb{H}^3$ be a half-space containing $P$.
	The non-empty intersection $F = \partial H \cap P$
	is called an \textbf{edge} or \textbf{face} of $P$ if $F$ is $1$ or $2$-dimensional, respectively. $P$ is called \textbf{trivalent} if $\tilde{P}$ is trivalent, that is, each vertex of $\tilde{P}$ is incident to exactly three faces in $\tilde{P}$. A face $f_P\subset \partial P$ is called a \textbf{compact face}, if it is a compact subset in $\HH^3$ (or equivalently, if its corresponding face in $\tilde{P}$ does not contain any ideal or hyperideal vertex). 
	
	\subsection{Characterizations of finite hyperbolic polyhedra}
    \label{section-character-finite-hp}
	Characterizing polyhedra is an ancient problem. As early as the ancient Greek period, the famous thinker Plato had discovered that there were only five regular polyhedra. An important early result was Cauchy's rigidity theorem, which states that convex polyhedra in three dimensions with congruent corresponding faces must be congruent to each other. In 1832, Steiner asked for a combinatorial characterization of convex polyhedra inscribed in the sphere, which is equivalent to the characterization of ideal hyperbolic polyhedra. About 100 years ago, Steinitz discovered that every 3-connected planar graph is a 1-skeleton of a convex polyhedron, which is an early result of the existence of polyhedra. In 1970's, Andreev \cite{An70A} characterized trivalent compact hyperbolic polyhedra in $\HH^3$ with non-obtuse dihedral angles. Andreev's result is essential for proving Thurston's hyperbolization theorem for Haken 3-manifolds \cite{Ot98}. Actually, in the famous book \cite[Chapter 13]{Th76}, Thurston interprets Andreev's result as a theorem about circle packings (i.e. the previously mentioned Koebe-Andreev-Thurston theorem\footnote{for tangential circle packings, the circle packing theorem was first discovered by Koebe \cite{Ko36} in 1936, and it takes an extremely simple and beautiful form: for every connected simple planar finite graph $G$ there is a circle packing in the plane whose intersection graph is (isomorphic to) $G$.}), which characterizes the existence and uniqueness of circle packings with non-obtuse intersection angles on the sphere. This shows deep connections between circle packings, hyperbolic polyhedra, and three-dimensional geometric topology. 
	

	\begin{thm}(\cite[Andreev]{An70A})
		Let $C$ be an abstract polyhedron with more than $4$ faces, and suppose that a dihedral angle $0<\Theta_i\leq\frac{\pi}{2}$ is given corresponding to each edge $e_i$ of $C$. There is a compact hyperbolic polyhedron $P$ whose faces realize $C$ with dihedral angle $\Theta_i$ for each edge $e_i$ if and only if the following four conditions hold:
		\begin{enumerate}
			\item [(1)] If three distinct edges $e_i$, $e_j$, $e_k$ meet at a vertex, $\Theta_i + \Theta_j + \Theta_k>\pi$.
			\item [(2)] If $\gamma$ is a prismatic $3$-circuit intersecting edges $e_i$, $e_j$, $e_k$, $\Theta_i + \Theta_j + \Theta_k<\pi$.
			\item [(3)] If $\gamma$ is a prismatic $4$-circuit intersecting edges $e_i$, $e_j$, $e_k$, $e_l$, $\Theta_i + \Theta_j + \Theta_k+\Theta_l<2\pi$.
			\item [(4)] If there is a four sided face bounded by edges $e_1$, $e_2$, $e_3$, and $e_4$, enumerated successively, with edges $e_{12}$, $e_{23}$, $e_{34}$, $e_{41}$ entering the four vertices (edge $e_{ij}$ connects to the ends of $e_i$ and $e_j$), then:
			$$\Theta_1 +\Theta_3 +\Theta_{12}+\Theta_{23} +\Theta_{34} +\Theta_{41} < 3\pi,$$
			and
			$$\Theta_2 +\Theta_4 +\Theta_{12}+\Theta_{23} +\Theta_{34} +\Theta_{41} < 3\pi.$$
			Furthermore, this polyhedron is unique up to isometries of $\HH^3$.
		\end{enumerate}
	\end{thm}
	
	Here $\gamma$ is called a \textbf{prismatic $k$-circuit}, if it is a simple closed curve formed by $k$ edges of the dual graph $C^*$, and the endpoints of the edges of $C$ intersected by $\gamma$ are all distinct. We remark that (under the above four conditions) $P$ is trivalent, moreover, the circle pattern corresponding to $P$ is precisely a RCP (see the analysis in the last subsection). Andreev also characterized finite convex ideal hyperbolic polyhedra with non-obtuse dihedral angles in \cite{An70B}. See \cite{Hodgson,RHD07,Zhou23} for detailed proofs and more recent generalizations.
	
	Breakthroughs to Steiner's question due to Rivin. In 1990s, Rivin-Hodgson \cite{RH93} completely characterized general compact convex polyhedra in $\HH^3$. They first characterized compact convex polyhedral surfaces in the de Sitter space $S^2_1$: i.e., a spherical cone metric $(M,g)$ homeomorphic to $\mathbb{S}^2$ can be isometrically embedded into the de Sitter space $S^2_1$ as a convex polyhedral surface if and only if (a) the cone angle at each cone point is greater than $2\pi$, and (b) the lengths of closed geodesics of $(M,g)$ are all strictly greater than $2\pi$. Since points in the de Sitter space $S^2_1$ correspond to half spaces in $\HH^3$, a convex polyhedron in $\HH^3$ gives rise to a convex polyhedron surface in $S^2_1$ by the polar map and vice versa. Consequently, they obtained the characterization of the extrinsic geometry of compact convex polyhedra in $\HH^3$. Rivin-Hodgson's work greatly generalizes Andreev's results above, and was proven by translating the de Sitter geometry to hyperbolic geometry via the polar map (see \cite{GHZ-advance,Gueritaud} for other proofs). Based on this work, Rivin \cite{Ri96} obtained a particularly elegant characterization for ideal hyperbolic polyhedra: 
	\begin{thm}(\cite[Rivin]{Ri96})
		Suppose that a convex ideal polyhedron $P$ in $\HH^3$ is given. Let $P^*$ denote the Poincar\'e dual of $P$, and assign to each edge $e^*$ of $P^*$ a weight $0<\Theta(e^*)<\pi$ equal to the interior dihedral angle at the corresponding edge $e$ of $P$. Then the following result holds: 
		\begin{enumerate}
			\item [(1)] If the edges $e_1^*,\cdots,e_n^*$ form the boundary of a face of $P^*$, then 
			$\sum_{i=1}^n\Theta(e_i^*)=(n-2)\pi$;
			\item [(2)] If $e_1^*,\cdots,e_n^*$ form a simple loop that is not the boundary of any 
			face of $P^*$, then 
			\[\sum_{i=1}^n\Theta(e_i^*)<(n-2)\pi.
			\]
		\end{enumerate}
		Conversely, any abstract polyhedron $P^*$ with weighted edges satisfying the above two conditions is the Poincar\'e dual of a convex ideal polyhedron $P$ with the interior dihedral angles equal to the weights.
	\end{thm}
	Later, Bao-Banahon \cite{Bao} generalized Rivin's theorem from ideal polyhedra to hyperideal case (note that in \cite{Bao}, both ideal vertex and hyperideal vertex are all called ``hyperideal vertex"). 
	\begin{thm}(\cite[Bao-Banahon]{Bao})
		\label{bao}
		Given a convex hyperideal polyhedron $P$ in $\HH^3$. Let $P^*$ denote the Poincar\'e dual of $P$, and assign to each edge $e^*$ of $P^*$ a weight $0<\Theta(e^*)<\pi$ equal to the interior dihedral angle at the corresponding edge $e$ of $P$. Then the following result holds: 
		\begin{enumerate}
			\item [(1)] If $e_1^*,\cdots, e^*_n$ is the boundary of a face $v^*$ in $P^*$, then $\sum_{i=1}^n\Theta(e_i^*)\le(n-2)\pi$. Moreover, the equality holds if and only if $v$ is an ideal vertex.
			\item [(2)] If $e_1^*,\cdots,e_n^*$ is a simple loop that is not the boundary of any face in $P^*$, then 
			\[\sum_{i=1}^n\Theta(e_i^*)<(n-2)\pi.
			\]
			\item [(3)]
			If $e_1^*,\cdots,e_n^*$ 
			form a simple path starting and ending on the boundary of the same face of $P^*$, but not contained in the boundary of that face, then 
			\[
			\sum_{i=1}^n\Theta(e_i^*)<(n-1)\pi.
			\]
		\end{enumerate}
		Conversely, any abstract polyhedron $P^*$ with weighted edges satisfying the above two conditions is the Poincar\'e dual of a convex hyperideal polyhedron $P$ with the interior dihedral angles equal to the weights.
	\end{thm}
	
	Recently, there have been some other outstanding works, such as \cite{HL17} and \cite{Zhou23}, which consider circle patterns that allow obtuse overlaps, as well as convex hyperbolic polyhedra that may have obtuse dihedral angles. The difference is that \cite{HL17} uses the Teichm\"uller space of interstices to represent the quasiconformal deformation space of circle patterns and the corresponding finite hyperbolic polyhedra, while \cite{Zhou23} extends Andreev's theorem to obtuse dihedral angles. Finally, we emphasize that the polyhedra appearing in this subsection are all finite polyhedra. We will consider infinite polyhedra in the following section.
	

	\subsection{Infinite trivalent hyperbolic polyhedra}
    \label{section-correspond-cp-hcp}
	
	Inspired by Huang-Liu \cite{HL17} and Zhou \cite{Zhou21, Zhou23}, we focus on trivalent hyperbolic polyhedra (THP) that naturally correspond to RCPs in this section. 
	Intuitively, a polyhedron $P$ with infinitely many faces is not ``complete'' in $\HH^3$, because there exists in $\overline{\HH^3}$ at least one such point on the ``boundary" of $P$: it does not belong to any face, and every neighborhood of it contains at least a vertex of $P$. More precisely, a point $p\in\overline{\HH^3}$ is called an \textbf{accumulation point} of $P$ if every neighborhood of $p$ 
	intersects infinitely many boundary planes of $\{H_i\}_{i \in I}$.  
	Because $\{H_i\}_{i \in I}$ is locally finite, all accumulation points necessarily lie in $\partial \HH^3$. In particular, a polyhedron with finitely many faces admits no accumulation points, whereas an infinite polyhedron always does. This constitutes an essential distinction between finite polyhedra and infinite polyhedra.
	
	Due to the (potential) existence of ideal vertices and hyperideal vertices, we first truncate the polyhedron $P$ at these vertices.
	This allows us to extract the geometric features of the polyhedron at infinity—such as ``non-compact" and ``incomplete"—which arise from its infinite number of faces rather than from its ideal or hyperideal vertices. Given an infinite hyperbolic polyhedron $P$, there is a canonical way to associate a truncated polyhedron $P^{\text{trun}}$, that is, at each hyperideal vertex, truncate it with its dual plane; at each ideal vertex, truncate it with a sufficiently small horoball (see section 1 in \cite{Bao} for details). If $P$ is a finite polyhedron, then $P^{\text{trun}}$ is compact in $\HH^3$. For an infinite hyperbolic polyhedron $P$, due to the existence of accumulation points, we know $P^{\text{trun}}\setminus K\neq \emptyset$ for any compact subset $K \subset \mathbb{H}^3$. This motivates the following definition.

	\begin{defn}\label{infinite_HP}
		Let $P$ be an infinite trivalent hyperbolic polyhedron. 
		\begin{itemize}
			\item [(1)] $P$ is called \textbf{parabolic} if $P^{\text{trun}}\setminus K_n$ tends to a unique point in $\partial\mathbb{H}^3$, where $\{K_n\}$ are compact in $\mathbb{H}^3$ and exhausts $\mathbb{H}^3$, or equivalently, if the accumulation points of the boundary faces of $P$ in $\HH^3$ consists of a single point $\{pt\}\in\partial\HH^3$;
			\item  [(2)] $P$ is called \textbf{hyperbolic} if $P^{\text{trun}}\setminus K_n$ tends to a half sphere in $\partial\mathbb{H}^3$, where $\{K_n\}$ are compact in $\mathbb{H}^3$ and exhausts $\mathbb{H}^3$, 
			or equivalently, if the accumulation points of the boundary faces form a circle on $\partial\HH^3$, and all the faces of $P$ are contained in a half space bounded by the hyperbolic plane induced by the circle. 
		\end{itemize}
		If all faces of $P$ are compact (meaning there are no ideal vertices or hyperideal vertices), it is called a \textbf{compact infinite THP}. Moreover, ideal vertices and  hyperideal vertices of $P$ correspond to non-compact ends of $P$.
	\end{defn}
	\begin{figure}[h]
		\centering
		\includegraphics[width=0.8\textwidth]{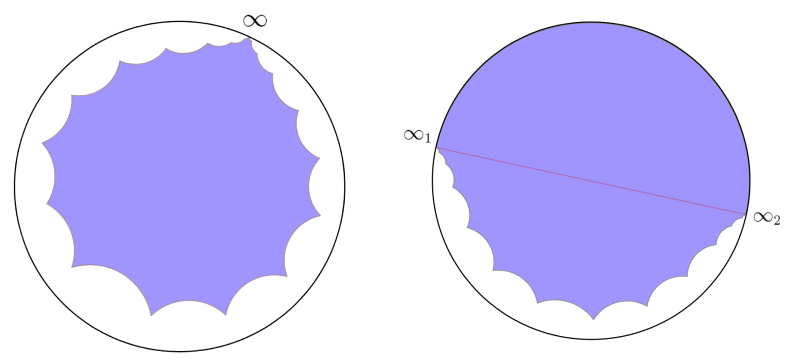}
		\caption{ Two types of THPs}
		\label{twotypesofcompactness}
	\end{figure}
	It is clear that this definition does not depend on the choice of exhausting sequence. Since it is easy to draw in two dimensions, one can use this simple two-dimensional case to imagine the three-dimensional situation. 
	In two dimensions, a parabolic ``hyperbolic polygon" has a single accumulation point, whereas a hyperbolic ``hyperbolic polygon" has two accumulation points, which are antipodal on $\partial\HH^2$ (see Figure \ref{twotypesofcompactness}). By analogy, in three dimensions, a parabolic THP again has exactly one accumulation point, while a hyperbolic THP has its accumulation points lying along an equator of the sphere $\partial\HH^3$.
	
	\medskip
	
    In the following, we establish the correspondence between THP and RCPs and prove our main results related to THP. Given an RCP $\mathcal{P} = \{C_i\}_{i \in V}$, whose contact graph $G=G(\pac)$ is induced by a disk triangulation $\mathcal{T}=(V,E,F)$. Consider the half space model of $\mathbb{H}^3$. Each circle $C_i$ corresponds to a region $H_i \subset \mathbb{H}^3$, defined as the closed convex hull of the ideal boundary points in $\mathbb{C} \setminus D_i$, where $D_i$ is the disk bounded by $C_i$. Then
	\begin{equation}\label{P_pac}
		P = \bigcap_{i \in V} H_i
	\end{equation}
	is the corresponding infinite hyperbolic polyhedron associated with the circle pattern $\mathcal{P}$. 

    \begin{prop}[Local correspondence]
    \label{prop-local-correspond}
    Let $[i,j,k]\in F$ and consider the three circle configuration $D_i$, $D_j$, $D_k$ in $\pac$. Suppose $\Theta_i$, $\Theta_j$, $\Theta_k$ satisfy at least one of (\ref{sum<pi}) or (\ref{angle-condition}). Then locally the triangle $[i,j,k]$ corresponds to a (compact) vertex, ideal vertex, or hyperideal vertex of $P$, if and only if the quantity $\Theta_i+\Theta_j+\Theta_k$ is greater than, equal to, or less than $\pi$ respectively. 
    \end{prop}
    \begin{proof}
	If the three circles $C_i$, $C_j$, $C_k$ neither form an interstice nor intersect at a common point, then the corresponding vertices lie in $\HH^3$; if these three circles intersect at a common point, then the corresponding point lies in $\partial\HH^3$; if the three circles form an interstice, they correspond to a hyperideal vertex. The last case requires some more explanation: from the perspective of Klein's projective model $\HH^3\subset \mathbb{R}^3\subset RP^3$, the projective plane containing three pairwise intersecting circles $C_i$, $C_j$, $C_k$ in $RP^3$ must have a common intersection point $v\in RP^3$. If $v\in\mathbb{R}^3$, it is clear that $v$ is a hyperideal vertex; if $v\notin\mathbb{R}^3$, one can always appropriately choose a projective transformation that preserves the intersection relationships and angles while mapping $v$ into $\mathbb{R}^3$, so $v$ remains a hyperideal vertex. In particular, if $\Theta_i$, $\Theta_j$, and $\Theta_k$ satisfy either (\ref{sum<pi}) or (\ref{angle-condition}), then based on the analysis in the last paragraph of Section \ref{section-three-circle-config}, we can conclude that the triangle face $f=[i,j,k]\in F$ corresponds to a vertex, ideal vertex, or hyperideal vertex of $P$, if and only if the quantity $\Theta_i+\Theta_j+\Theta_k$ is greater than, equal to, or less than $\pi$ respectively.
    \end{proof}
	\begin{thm}\label{RCPTHP}
    	Let $\mathcal{P}=\left\{C_i\right\}_{i \in V}$ be an RCP whose contact graph is $G = (V, E)$. Let $\Theta$ be the intersection angles of $\pac$.  
        If $(G, \Theta)$ satisfies ($Z_1$) and ($Z_2$), then the polyhedron $P$ defined in \eqref{P_pac} is a THP that is combinatorially isomorphic to the Poincaré dual of $G$, and the dihedral angle between the hyperbolic planes $\partial H_i$ and $\partial H_j$ is equal to $\Theta([i, j])$ for each $[i, j] \in E$.
	\end{thm}
	
	\begin{proof}
		
		The proof is inspired by Zhou \cite{Zhou23} and is divided into four parts.
		
		\medskip
		\emph{Part 1. Locally-finiteness:}
		\medskip
		
		We only prove the parabolic-THP case, since the proof of the hyperbolic-THP case is similar. We prove this by contradiction. Suppose that there exists a compact set $K\subset\HH^3$ such that there are infinite half-spaces intersecting $K$. Let $H^*$ be a hyperbolic plane separating $K$ from the ideal point $pt$. Then there are infinitely many hyperbolic planes, induced by $\pac$, which intersect $\partial H^*$. Let $C^*$ be the ideal boundary of $H^*$, then the circles in $\pac$ must have an accumulation point in $S^2=\partial\HH^3$. This contradicts the fact that the RCP is locally finite in $\mathbb{S}^2\backslash\{pt\}$.

		\medskip
		\emph{Part 2. Trivalence:}
		\medskip
		
		According to the definition of RCP, two disks $D_i$ and $D_j$ intersect if and only if $i\sim j$. Hence for any two faces $f_i$ and $f_j$ of $P$, the corresponding projective planes in $RP^3$ intersect if and only if the corresponding disks intersect. Thereby, for any four projective planes $f_1,f_2,f_3,f_4$ in $RP^3$ corresponding to different faces of $P$, we have
		\[
		f_1\cap f_2\cap f_3 \cap f_4=\emptyset.
		\]

		\medskip
		\emph{Part 3. Existence of an interior point:}
		\medskip
		
		Take any triangle face $[i,j,k] \in F$, it corresponds to a (compact) vertex, ideal, or hyperideal vertex of $P$. Next, we will prove these three cases separately. 

        \medskip
		\emph{Case 1}: $[i,j,k]$ corresponds to a compact vertex. By Proposition~\ref{prop-local-correspond} we have
        \[
        \Theta([i,j]) + \Theta([j,k]) + \Theta([k,i]) > \pi.
        \]
        By Lemma~\ref{sanyuangouxing_yinli}, there exists a unique point
        \[
        p = \partial H_i \cap \partial H_j \cap \partial H_k,
        \]
        which lies in the interior of $\mathbb{H}^3$.
        From Lemma~\ref{mostthree}, for any $l \in V \setminus \{i,j,k\}$ we have
        \[
        D_l \cap D_i \cap D_j \cap D_k = \emptyset.
        \]
        Hence, there exists a neighborhood $B_p$ of $p$ such that $B_p \subset H_l$ for all $l \notin \{i,j,k\}$.
        Consequently,
        \[
        B_p \cap H_i \cap H_j \cap H_k \subset P
        \]
        contains an interior point, implying that $P$ is non-degenerate, i.e., the interior of $P$ is non-empty.
        
		  \medskip
		\emph{Case 2}: $[i,j,k]$ corresponds to an ideal vertex. By Proposition~\ref{prop-local-correspond} we have
        \[
        \Theta([i,j]) + \Theta([j,k]) + \Theta([k,i]) = \pi.
        \]
        By Lemma \ref{sanyuangouxing_yinli}, there exists a unique point $p \in \partial \mathbb{H}^3$ contained in $D_i \cap D_j \cap D_k$.
        By Lemma \ref{mostthree} and the same reasoning as in the first case, we conclude that the interior of $P$ is non-empty.
        
        \medskip
        \emph{Case 3}: $[i,j,k]$ corresponds to a hyperideal vertex. By Proposition~\ref{prop-local-correspond} we have
        \[
        \Theta([i,j]) + \Theta([j,k]) + \Theta([k,i]) < \pi.
        \]
        By Lemma~\ref{sanyuangouxing_yinli}, the set $\Delta_{i j k} \backslash\left(D_i \cup D_j \cup D_k\right)$ contains an interior point $p$ in the topology of $\partial \mathbb{H}^3$.
        Hence, there exists a neighborhood $B_p$ of $p$ such that
        \[
        B_p \subset H_i \cap H_j \cap H_k.
        \]
        By Lemma~\ref{containinFLlemma}, we further have $B_p \subset H_l$ for any $l \notin \{i,j,k\}$.
        Therefore,
        \[
        B_p \subset P
        \]
        contains an interior point, which again implies that the interior of $P$ is non-empty.

		\medskip
		\emph{Part 4. Combinatorial equivalence}
		\medskip

		We claim that $P$ is combinatorial equivalent to the Poincar\'e dual of the triangulation graph $G$ determined by $\mathcal{P}$. By Lemma \ref{indispensable} we have
		\[
		D_i \setminus \bigcup_{j \neq i} D_j \neq \emptyset,
		\]
		which implies that $\partial H_i \cap \partial P \neq \emptyset$ for all $i \in V$. By the definition of RCP, we have two cases:

        \medskip
		\emph{Case 1}. if $[i,j] \notin E$, then $D_i \cap D_j = \emptyset$ and hence $\partial H_i \cap \partial H_j = \emptyset$;

        \medskip
		\emph{Case 2}. if $[i,j] \in E$, then $D_i \cap D_j \neq \emptyset$, so $\partial H_i \cap \partial H_j \neq \emptyset$.
		\medskip
        
		Therefore, each face $\partial H_i$ contributes a portion of $\partial P$, and the adjacency of faces in $P$ matches the edge connectivity in $\mathcal{P}$. Additionally, in the half-space model, the dihedral angle between $\partial H_i$ and $\partial H_j$ equals the intersection angle $\Theta([i,j])$.
        
	\end{proof}

    \begin{rem}\label{remark-rcp-notimply-z1z2}
    	\begin{figure}[h]
		\centering
		\includegraphics[width=0.8\textwidth]{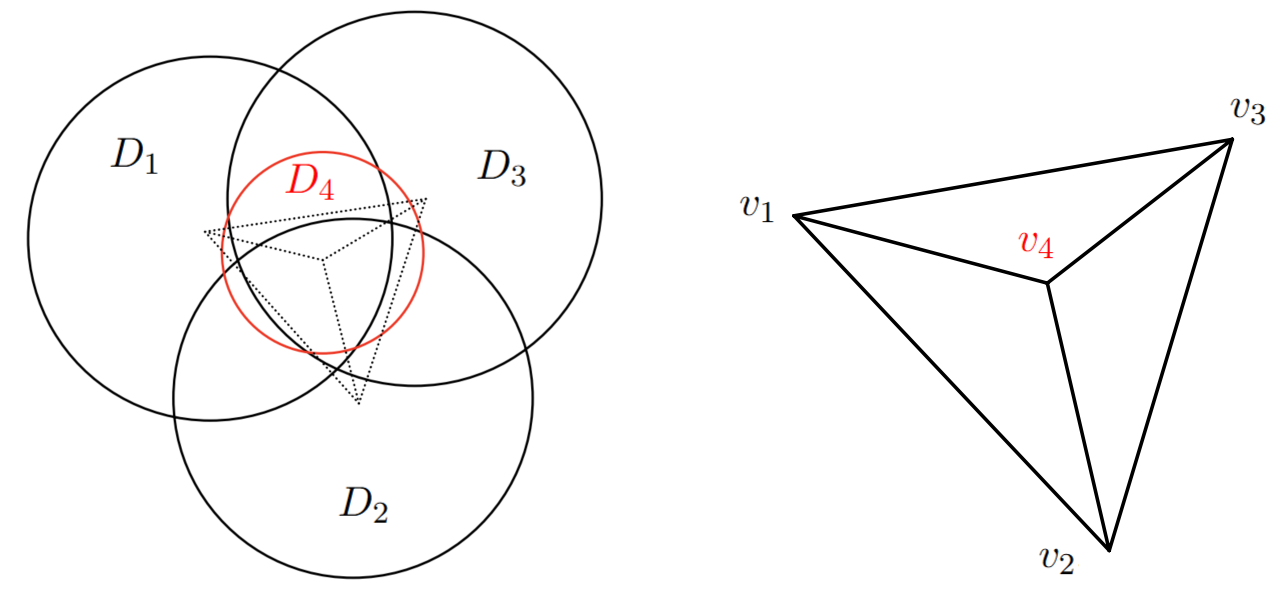}
		\caption{an example that violates conditions ($Z_1$) and ($Z_2$)}
		\label{RCPZ1Z2}
	\end{figure}
	As mentioned by He~\cite{He}, for each connected planar graph $G$ with $\Theta \in [0, \pi / 2]^E$, if there is a disk pattern $\mathcal{P}$ that realizes the data $(G, \Theta)$, then conditions ($C_1$) and ($C_2$) (see the paragraph before Definition \ref{RCP_def}, or see page 5 in \cite{He}) must hold. However, such a phenomenon does not occur in the case where $\Theta \in [0, \pi)^E$, and Figure \ref{RCPZ1Z2} provides an example that violates conditions ($Z_1$) and ($Z_2$). In fact, $\Theta_{12} + \Theta_{23} + \Theta_{31} > \pi$ by Lemma~\ref{angleleqpilemma} and $\Theta_{14}+ \Theta_{34}> \pi + \Theta_{13}$ by Lemma~\ref{containlemma}, which contradicts ($Z_1$) and ($Z_2$) respectively. In addition, Example \ref{ex-four-circle} shows that if the intersection angles are all $2\pi/3$, the radius can be selected appropriately to produce an RCP, which does not satisfy ($Z_4$). These facts can also be used to explain why He’s rigidity theorem does not require conditions ($C_1$) and ($C_2$), while our Rigidity Theorem \ref{thm-intro-rigidity-cp} cannot omit conditions ($Z_2$) and ($Z_4$).
    \end{rem}

	Conversely, we consider the inverse process of the above construction, i.e. how to go from THP to RCPs. Given a THP $P = \bigcap_{i \in V} H_i$, which is combinatorially equivalent to the Poincar\'e dual of $\mathcal{T}$, with dihedral angles that satisfy $\Theta(e^*) = \Theta(e)$ for $e\in E$.  Each region $H_i \subset \mathbb{H}^3$ corresponds to a circle $C_i$, defined as $\partial H_i \cap \partial \mathbb{H}^3$. Then
	\begin{equation}\label{pac_P}
		\mathcal{P} = \{C_i\}_{i \in V}
	\end{equation}
	is obviously a CP associated with $P$. Generally, $\pac$ is \emph{not regular} due to the reason that there may exist $[i,j]\notin E$ with $\partial H_i\cap \partial H_j\neq\emptyset$.  Next, we investigate which conditions are suitable.
	
	For any polyhedron $P = \bigcap_{i \in V} H_i$, the \textbf{dual contact graph} $G(P) = (V, E)$ is defined as the graph whose vertices correspond to the faces of $P$, with an edge connecting two vertices if and only if the corresponding faces are adjacent.  
    The \textbf{dihedral angles} $\Theta(P)$ of $P$ are given by the dihedral angles $\Theta([i,j])$ between the planes $\partial H_i$ and $\partial H_j$ for each $[i,j] \in E$. It is evident that, for any THP $P$, the dual contact graph $G(P)$ is a disk triangulation graph, and the circle pattern $\mathcal{P}$ defined in \eqref{pac_P} \emph{weakly realizes} $(G(P), \Theta(P))$, but does not \emph{realize} $(G(P), \Theta(P))$, since it may not be an RCP. The following result shows that conditions ($Z_1$) and ($Z_2$) are satisfied for every THP. This indicates that the conditions in Theorem \ref{RCPTHP} are natural and necessary.

	\begin{prop}\label{prop-thp-z1+z2}
        Let $P = \bigcap_{i \in V} H_i$ be a THP with dual contact graph $G=G(P)$ and dihedral angle $\Theta=\Theta(P)$. Then conditions ($Z_1$) and ($Z_2$) hold for $(G, \Theta)$ .
	\end{prop}
	\begin{proof}
		Firstly let us consider the necessity of the condition ($Z_1$). If a face corresponds to an ideal or hyperideal vertex of $P$, then the summation of the dihedral angles is less than or equal to $\pi$, therefore the condition ($Z_1$) holds. If the face corresponds to a hyperbolic vertex $x\in\HH^3$, then the dihedral angles associated with the vertex are inner angles of a spherical triangle in the unit tangent sphere $T_x^1\HH^3$. Therefore, the condition ($Z_1$) holds by basic spherical geometry.
		
		The necessity of ($Z_2$) follows from the Gauss-Bonnet Theorem directly. Let $e_1,\cdots e_n$ be a simple closed loop in $\mathcal{T}$ that is not the boundary of a face counting index modulo $n$. Let $v_i$ be the common point of edge $e_i$ and $e_{i+1}$ respectively. Let $H_i$ be the hyperbolic half space corresponding to the hyperbolic plane associate with $v_i$ containing the vertex $v_i$, since each $v_i$ corresponds to a hyperbolic plane of $P$. Let $P'=\cap_{i=1}^n H_i$. Then it is clear that $P\subset P'$ and $\partial P'$ is an annulus. Moreover, the intersection of the closure of $P'$ and the infinite boundary $\partial\HH$ in the Poincar\'e ball model is the boundary of a sequence of closed disks $D_1,\cdots,D_n$, where the intersection angle of $D_i$ and $D_{i+1}$ is $\Theta_{v_i}$. Using the stereographic projection $\phi$, we obtain a sequence of Euclidean disks $\{\phi(D_i)\}_{i=1}^n$, whose union is a topologically annulus in the plane, which divides $\mathbb{R}^2$ into two components. The compact component, which is not a single point, consists of closed arc $\gamma_1,\cdots,\gamma_n$ whose geodesic curvature is $k_i$ which is negative. Then with the Gauss-Bonnet Theorem, we have
        \[
        2\pi=\sum_{i=1}^n\int_{\gamma_i}k_i+\sum_{i=1}^n(\pi-\Theta_{v_i}),
        \]
        which indicates that $\sum_{i=1}^n\Theta_{v_i}<(n-2)\pi.$
	\end{proof}
	
	\begin{thm}\label{THPRCP}
		Let $P = \bigcap_{i \in V} H_i$ be a THP, and denote $G = (V, E)$ by its dual contact graph. Let $\Theta$ be the dihedral angles of $P$.  
		If  $(G, \Theta)$ satisfies ($Z_3$), then the CP $\mathcal{P} = \{C_i\}_{i \in V}$ defined as in \eqref{pac_P} is an RCP realizing $(G,\Theta)$.
	\end{thm}
	\begin{proof}
		The locally-finiteness of the corresponding circle pattern can be deduced directly from the definition of infinite parabolic and hyperbolic THP. By proposition \ref{prop-thp-z1+z2}, conditions ($Z_1$) and ($Z_2$) hold for $(G, \Theta)$.
        \begin{figure}[h]
			\centering
			\includegraphics[width=0.42\textwidth]{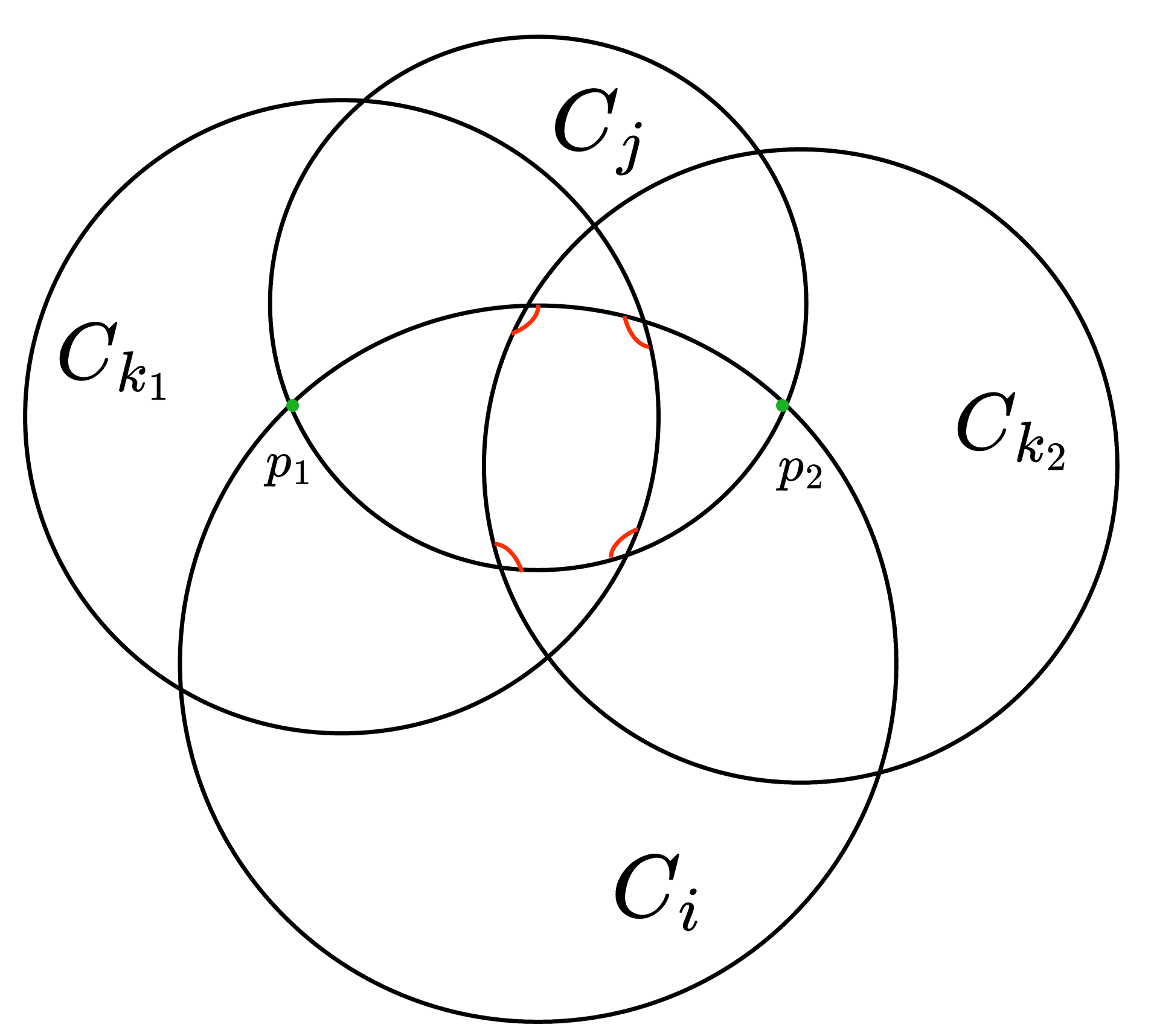}
			\caption{a four circle configuration}
			\label{RCPLEMMA}
		\end{figure}
		
		For any $i,j\in V$, recall $d(i,j)$ is the combinatorial distance in the graph induced by $\mathcal{T}$. By Lemma  \ref{containinFLlemma}, we only need to prove $$D_i \cap D_j= \emptyset,$$
		for any $i,j\in V$ with $d(i,j)=2$.
		We prove this by contradiction. If the previous property is not true, then we assume $D_i \cap D_j \neq  \emptyset.$ From Lemma \ref{containinFLlemma}, we have
		$$D_i \cap D_j \subset \bigcup_{k \in N(i)} D_k$$
		and
		$$D_i \cap D_j \subset \bigcup_{k \in N(j)} D_k.$$
		  Then it follows that 
		\begin{equation}\label{containformula}
			D_i \cap D_j \subset \bigcup_{k \in N(i)\cap N(j)} D_k.
		\end{equation}

        \medskip
		\emph{Case 1}. If $C_i\cap C_j=\emptyset$, since we are assuming $D_i\cap D_j\ne\emptyset$,
one disk must be contained in the other. Without loss of generality, we may well assume $D_i\subset D_j$. It is easy to check  that it  contradicts  Lemma \ref{containinFLlemma}.

        \medskip
		\emph{Case 2}. If $C_i\cap C_j$ contains only one point $p$, then the above property \eqref{containformula} implies that there exist a $k\in N(i)\cap N(j)$ such that
		$$p \in D_k.$$
        By Lemma \ref{containlemma}, we have
        \begin{equation}\label{RCPl1}
            \Theta([i,k])+ \Theta([k,j])\geq\pi.
        \end{equation}
		Using the condition ($Z_3$), it follows that $\Theta([i,k])+ \Theta([k,j])<\pi,$ which contradicts \eqref{RCPl1}.

        \medskip
		\emph{Case 3}. If $C_i\cap C_j$ contains two different points $p_1$ and $p_2$. From Lemma \ref{containlemma}, we know for any $k \in N(i)\cap N(j)$, if $D_k \cap D_i \cap D_j \neq \emptyset$, then $D_k$ contains exactly one point in $\{p_1, p_2\}$. Thus, there exist $k_1, k_2\in N(i)\cap N(j)$ such that $[k_1,k_2]\in E$ and $D_{k_1}, D_{k_2}$ contains $p_1, p_2$ respectively.
		Now we consider  two bigons $D_i\cap D_j$ and $D_{k_1} \cap D_{k_2}$. By Lemma \ref{containinFLlemma}, Lemma \ref{sanyuangouxing_yinli} and Corollary \ref{containlemmacor}, we know 
		$$
		D_{k_1} \cap D_{k_2} \subset D_i\cup D_j.
		$$
		Thus, there are only three cases to consider:

        \medskip
		\emph{Case 3-1}. If $D_{k_1} \cap D_{k_2} \subset D_i$, from Lemma \ref{containlemma}, we have
			$$ 
			\Theta([i,k_1])+ \Theta([i, k_2]) \geq  \Theta([k_1, k_2])+  \pi,
			$$
			which contradicts  the condition ($Z_1$).

        \medskip
        \emph{Case 3-2}. If $D_{k_1} \cap D_{k_2} \subset D_j$, from Lemma \ref{containlemma}, we have
			$$ 
			\Theta([j,k_1])+ \Theta([j, k_2]) \geq  \Theta([k_1, k_2])+  \pi,
			$$
			which contradicts   the condition ($Z_1$).
            
        \medskip		
        \emph{Case 3-3}. If $\partial (D_{i} \cap D_{j}) \cap \partial (D_{k_1} \cap D_{k_2})$ contains four points as shown in Figure \ref{RCPLEMMA}. It is clear that  
			$$ 
			\Theta([i,k_1])+ \Theta([i, k_2])+\Theta([j,k_1])+ \Theta([j, k_2]) \geq  2 \pi,
			$$
			which contradicts  the condition ($Z_3$).
	\end{proof}
	
	Based on the previous discussions, we give final proofs of Theorems \ref{thm-exist-IP}, \ref{uniformization}, and \ref{thm-rigidity-IP}.

	\begin{proof}[Proof of Theorem \ref{thm-exist-IP}]
		This theorem can be obtained directly by Theorem \ref{infinite_existence} and Theorem \ref{RCPTHP}. 
	\end{proof}
	
	\begin{proof}[Proof of Theorem \ref{uniformization}]
		The equivalence between ($U_1$) and ($U_2$) was established in Theorem \ref{uniformization_RCP}, while the equivalence between ($U_1$) and ($U_4$) was previously obtained by He-Schramm in \cite{He2}. Deriving ($U_3$) from ($U_2$) can be directly seen from Theorem \ref{RCPTHP}; conversely, deriving ($U_2$) from ($U_3$) comes directly from Theorem \ref{THPRCP}.
	\end{proof}

	\begin{proof}[Proof of Theorem \ref{thm-rigidity-IP}]
		Given two hyperbolic (or parabolic, resp.) THP $P_1$ and $P_2$, Theorem \ref{THPRCP} implies that $P_1$ and $P_2$ correspond to two RCPs $\mathcal{P}_1$ and $\mathcal{P}_2$, which are hyperbolic (or parabolic, resp.) RCPs. Hence this theorem follows directly from Theorem \ref{thm-intro-rigidity-cp}.
	\end{proof}

        At the end of the article, we show the non-rigidity phenomenon for non-parabolic THPs. Let $\Omega \neq \mathbb{U}$ be a Jordan domain with smooth boundary $\partial \Omega$, it is clear that we can construct a circle pattern $\mathcal{P}_1 = \{ C_i \}_{i \in V}$ which is locally finite in $\Omega$. We can require all of its intersection angles to be acute, thereby making it an RCP. Obviously, $(G(\mathcal{P}_1),\Theta(\mathcal{P}_1))$ satisfies all the conditions ($Z_1$), ($Z_2$) and ($Z_3$). By Theorem \ref{RCPTHP}, there is a THP $P_1$ that corresponds to $\pac_1$. Since $\pac_1$ is not locally finite in $\mathbb{C}$, by Lemma \ref{parabolic_equivalent} we know that $G(\pac_1)$ is VEL-hyperbolic. According to Theorem \ref{uniformization}, there exists an RCP $\mathcal{P}_2$ with $\carrier(\mathcal{P}_2) = \mathbb{U}$ that realizes the same data $(G(\pac_1), \Theta(\pac_1))$. From Theorem \ref{RCPTHP} again, there exists a THP $P_2$ corresponding to $\pac_2$. Moreover, the two infinite THPs $P_1$ and $P_2$ share the same combinatorial structure and corresponding dihedral angles. However, $P_1$ and $P_2$ are not isometric since $\mathcal{P}_1$ and $\mathcal{P}_2$ are not Möbius equivalent (as long as $\Omega$ is not a disk). This indicates that rigidity phenomena generally do not occur in infinite THPs, which is in sharp contrast to the case of finite polyhedra. In addition, the above construction process also indicates that 
        for infinite THPs with non-circular Jordan curve boundary, according to the Uniformization Theorem \ref{uniformization}, they can generally be understood by selecting a hyperbolic THP as their representative (having the same combinatorial structure and dihedral angle structure).

	\bigskip
	
	\bibliographystyle{plain}

	\noindent Huabin Ge, hbge@ruc.edu.cn\\
	\emph{School of Mathematics, Renmin University of China, Beijing 100872, P. R. China.}\\[-8pt]
	
	\noindent Longsong Jia, jialongsong@stu.pku.edu.cn\\
	\emph{School of Mathematical Sciences, Peking University, Beijing, 100871, P. R. China.}\\[-8pt]
	
	\noindent Hao Yu, yoho@ruc.edu.cn\\
	\emph{School of Mathematics, Renmin University of China, Beijing 100872, P. R. China.}\\[-8pt]
	
	\noindent Puchun Zhou, phzpc98@gmail.com\\
	\emph{School of Mathematical Institute, University of Oxford, UK}	
	
\end{document}